\documentclass[10pt]{article}

\RequirePackage{amsthm,amsmath,amsfonts,amssymb}
\RequirePackage{bm}
\RequirePackage{url}
\usepackage{babel}

\RequirePackage{natbib}
\usepackage[usenames]{color}
\usepackage[plainpages=false,pdfpagelabels,colorlinks=true,linkcolor=mypurple,urlcolor=mypurple,citecolor=mypurple]{hyperref}
\usepackage{appendix}
\usepackage{natbib}
\usepackage{enumitem}
\usepackage{graphicx}
\usepackage{subfigure}
\usepackage{booktabs}
\usepackage{array}
\usepackage{algorithm}
\usepackage{algpseudocode}
\usepackage{dsfont}

\usepackage{aliascnt}
\usepackage[nameinlink,capitalize]{cleveref}

\usepackage{fullpage}
\usepackage[font={footnotesize,it}]{caption}
\numberwithin{equation}{section}

\usepackage{smile}
\usepackage{tikzstyle}

\usepackage{etoolbox}
\newtoggle{vaos}
\togglefalse{vaos}
\newcommand{\aosversion}[2]{\iftoggle{vaos}{#1}{#2}}

\newcommand{\btrue}{\mathrm{True}}
\newcommand{\bfalse}{\mathrm{False}}

\newcommand{\xyy}[1]{{\color{myblue} \colorbox{myblue!30}{Xuyang} #1}}

\begin{document}

\title{Fundamental Computational Limits in Pursuing Invariant Causal Prediction and Invariance-Guided Regularization}

\author{Authors}

\author{Yihong Gu$^1$,
 ~~~Cong Fang$^2$,
 ~~~Yang Xu$^2$, 
 ~~~Zijian Guo$^3$, 
 ~~~Jianqing Fan$^1$ \footnote{Supported by NSF Grants DMS-2210833 and DMS-2412029.}\\
{$^1$Princeton University, $^2$Peking University, and $^3$Rutgers University}
}

\date{}

\maketitle

\begin{abstract}

Pursuing invariant prediction from heterogeneous environments opens the door to learning causality in a purely data-driven way and has several applications in causal discovery and robust transfer learning. However, existing methods such as ICP \citep{peters2016causal} and EILLS \citep{fan2023environment} that can attain sample-efficient estimation are based on exponential time algorithms. In this paper, we show that such a problem is intrinsically hard in computation: the decision problem, testing whether a non-trivial prediction-invariant solution exists across two environments, is NP-hard even for the linear causal relationship. In the world where P$\neq$NP, our results imply that the estimation error rate can be arbitrarily slow using any computationally efficient algorithm. This suggests that pursuing causality is fundamentally harder than detecting associations when no prior assumption is pre-offered. 

Given there is almost no hope of computational improvement under the worst case, this paper proposes a method capable of attaining both computationally and statistically efficient estimation under additional conditions. Furthermore, our estimator is a distributionally robust estimator with an ellipse-shaped uncertain set where more uncertainty is placed on spurious directions than invariant directions, resulting in a smooth interpolation between the most predictive solution and the causal solution by varying the invariance hyper-parameter. Non-asymptotic results and empirical applications support the claim. 
\end{abstract}


\section{Introduction}

How do humans deduce the cause of a target variable from a set of candidate variables when only passive observations are available? A natural high-level principle is to identify the variables that produce consistent predictions at different times, locations, experimental conditions, or more generally, across various environments. This heuristic is implemented in statistical learning by seeking invariant predictions from diverse environments \citep{peters2016causal, heinze2018invariant, fan2023environment, gu2024causality}. This approach goes beyond just learning associations in the recognition hierarchy \citep{bareinboim2022pearl} and enables the discovery of certain data-driven causal relationships without prior causal assumptions. However, existing methods that realize general invariance learning rely on explicit or implicit exhaustive searches, which are computationally inefficient. This raises the question of whether learning invariant predictions is fundamentally hard. This paper contributes to understanding the fundamental limits and introducing a novel relaxed estimator for invariance learning. Theoretically, we prove this problem is intrinsically hard using a reduction argument \citep{karp20reducibility} with novel constructions. Our theoretical message further implies that learning data-driven causality is fundamentally harder than detecting associations. On the methodological side, we propose a relaxation in two aspects: our approach balances computational efficiency and statistical accuracy on one hand while optimizing trade-offs between prediction power and robustness on the other. 

\subsection{Pursuit of Linear Invariant Predictions}
\label{sec:lip}

Suppose we are interested in pursuing the linear invariant relationship between the response variable $Y\in \mathbb{R}$ and explanatory covariate $X\in \mathbb{R}^d$ using data from multiple sources/environments. Let $\mathcal{E}$ be the set of environments. For each environment $e\in \mathcal{E}$, we observe $n$ data $\{(X_i^{(e)}, Y_i^{(e)})\}_{i=1}^n$ that are i.i.d. drawn from some distribution $(X^{(e)}, Y^{(e)}) \sim \mu^{(e)}$ satisfying
\begin{align}
\label{model:lip}
    Y^{(e)} = (\beta^\star_{S^\star})^\top X^{(e)}_{S^\star} + \varepsilon^{(e)} \qquad \text{with} \qquad \mathbb{E}[X^{(e)}_{S^\star} \varepsilon^{(e)}] \equiv 0
\end{align} where $\beta^\star$ is the true parameter that is \emph{invariant} across different environment and $S^\star=\supp(\beta^\star)$ denotes the support of $\beta^\star$, while the distribution of $\mu^{(e)}$ may vary across environments. Here we assume different environments have the same sample size $n$ for presentation simplicity. The goal is to recover $S^\star$ and $\beta^\star$ using the observed data $\mathcal{D}^{\mathcal{E}} = \{(X_i^{(e)}, Y_i^{(e)})\}_{i\in [n], e\in \mathcal{E}}$.  

Methods inferring the invariant set $S^\star$ from \eqref{model:lip} can be applied to causal discovery under the structural causal model (SCM) \citep{glymour2016causal} framework. This is because when observing environments where \emph{interventions} are applied within the covariates $X$, $S^\star = \{j: X_j \text{ is \emph{direct cause} of } Y\}$ satisfies \eqref{model:lip} and is unique in some sense when the intervention is non-degenerate and enough \citep{peters2016causal, gu2024causality}; see the discussion in \cref{sec:related}. Though initially motivated by causal discovery under the SCM framework that may be sensitive to model misspecification, pursuing invariant predictions from heterogeneous environments itself is a much more generic principle in statistical learning, or a type of inductive bias in causality \citep{buhlmann2020invariance, gu2024causality}, that can also facilitate, for example, robust transfer learning \citep{rojas2018invariant}, prediction fairness among sub-populations \citep{hebert2018multicalibration}, and out-of-distribution generalization \citep{arjovsky2019invariant}. 

Unlike the standard linear regression under which each variable $X_j$ is either \emph{truly important} ($j\in S^\star$) or \emph{exogenously spurious} \citep{fan2016guarding} ($j\notin S^\star$ but $\mathbb{E}[X_j \varepsilon] = \mathbb{E}[X_j (Y - X^\top \beta^\star)] = 0$), the set of candidate variables in \eqref{model:lip} can be decomposed into three groups:
\begin{align}
\label{eq:set-g}
    \{1,\ldots, d\} = S^\star \bigcup \overbrace{\underbrace{\left\{ j\notin S^\star: \mathrm{Cov}^{\mathcal{E}}(\varepsilon, X_j) \neq 0\right\}}_{\text{Endogenously Spurious Variables} ~G:=} \bigcup \underbrace{\left\{ j\notin S^\star: \mathrm{Cov}^{\mathcal{E}}(\varepsilon, X_j) = 0\right\}}_{\text{Exogeneously Spurious Variables}}}^{\text{Spurious Variables}~ (S^\star)^c},
\end{align} where $\mathrm{Cov}^{\mathcal{E}}(\varepsilon, X_j):=\frac{1}{|\mathcal{E}|}\sum_{e\in \mathcal{E}} \mathbb{E}[\varepsilon^{(e)} X_j^{(e)}]$ is the pooled covariance between the noise and the covariate $X_j$ across different environments. The major difference compared with standard linear regression and the main difficulty behind such an estimation problem is the presence of \emph{endogenously spurious variables} $G$ \citep{fan2014endogeneity}. 
The exogenously spurious variable is one that lacks predictive power for the noise $\varepsilon=Y-X_{S^\star}^\top \beta_{S^\star}^\star$ in population and only increases the estimation error by $n^{-1/2}$ if it is falsely included.  It usually does not cause the bias of estimation but inflates slightly the variance.  In contrast, endogenously spurious variables contribute to predicting the noise; thus, the false inclusion of any such variable results in inconsistent estimation due to the biases they create.
An illustrative example is to classify whether the object in an image is a cow $(Y=0)$ or camel $(Y=1)$ using three extracted features $X_1$=body shape, $X_2$=background color, and $X_3$= temperature or time that the photo is taken. Here $S^\star=\{1\}$ is the invariant and causal feature, while $G=\{2\}$ helps predict the noise $\varepsilon$, since cows (resp. camels) usually appear on green grass (resp. yellow sand) in the data collected. $X_3$ is exogeneously spurious:  including it does not increase estimation bias but slight variance.   From a statistical viewpoint, the core difficulty is to distinguish whether a variable is truly important, or endogenously spurious among those statistically significant variables that contribute to predicting $Y$. This is where multi-environment comes into play. There is a considerable literature on estimating the parameter $\beta^\star$ in \eqref{model:lip} \citep{peters2016causal, rothenhausler2019causal, rothenhausler2021anchor, pfister2019invariant, arjovsky2019invariant, yin2021optimization}. 

\cite{fan2023environment} first realized sample-efficient estimation for the general model \eqref{model:lip} and offered a comprehensive non-asymptotic analysis in terms of both $n$ and $\mathcal{E}$, this idea is further extended to the fully non-parametric setting in \cite{gu2024causality}. Specifically, it shows that given data from finitely many environments $|\mathcal{E}|<\infty$, one can identify $S^\star$ with $n=\infty$ under the minimal identification condition:
\begin{align}
\label{ident:lip}
    \forall S\subseteq [d] ~\text{with}~ S\cap G \neq \emptyset ~~~ \Longrightarrow ~~~ &\exists e, e'\in \mathcal{E}, \beta^{(e,S)} \neq \beta^{(e',S)}
\end{align} where $\beta^{(e,S)}:=\argmin_{\supp(\beta) \subseteq S} \mathbb{E}_{(X, Y)\sim \mu^{(e)}}[|Y-\beta^\top X|^2]$. This requires that $S^\star$ is the maximum set that preserves the invariance structure in that incorporating any endogenously spurious variables in $G$ will result in shifts in predictions across $\mathcal{E}$. Turning to the empirical counterpart, the optimal rate for linear regression can be attained therein using their proposed \emph{environment invariant linear least squares} (EILLS) estimator. This implies that as long as $\beta^\star$ can be identified under finitely many environments, unveiling the data-driven causality parameter $\beta^\star$ in \eqref{model:lip} is as statistically efficient as estimating the association counterpart in standard linear regression.

Promising through the above progress, the invariance pursuit procedure has two drawbacks. The first is about the computational burden. The estimation error is only guaranteed for the global minimizer of the objective function in \cite{fan2023environment} and \cite{gu2024causality}. An exponential-in-$d$ algorithm is adopted to find the global minimizer of the objective function that \cite{fan2023environment} proposes. Though the Gumbel trick introduced by \cite{gu2024causality} allows variants of gradient descent algorithm to perform well in practice, the nonconvexity nature is still kept and there are no theoretical guarantees on the optimization. 

The second is that the invariant model is typically conservative in its predictive performance for a new environment. Though it finds the ``maximum'' invariant set, the invariant prediction model will eliminate the endogenously spurious variables that result in heterogeneous predictions in $\mathcal{E}$. This may result in conservativeness in prediction with the help of the endogenous variables, which is the best for the adversarial environment but is not so for the prediction environment of interest. In the aforementioned cow-camel classification task, suppose $r_1=95\%$ cows (resp. camels) appear on grass (resp. sand) in the first environment $e=1$ and the spurious ratio is $r_2=70\%$ in environment $e=2$. In this case, an invariant prediction model drops the background color $X_2$ due to its variability across environments. In general, a prediction model without $X_2$ is intuitively the best when $r= 0\%$, yet potentially reduces predictive power compared to the ones including $X_2$ when evaluated in an environment with $r> 50\%$.

The above discussion gives rise naturally to the following two questions, which will be addressed in this paper. 

\begin{quote}
    \it 
    $\mathsf{Q1}$. Can statistically efficient estimation of $\beta^\star$ in \eqref{model:lip} be attained by computationally efficient algorithms in general? If not, can it be attainable under some additional conditions? 
    
    $\mathsf{Q2}$. Can we have benefits by designing methods that smoothly ``interpolate'' the estimators for the invariant causal model $\beta^\star$ and the most predictive solution $\bar{\beta} := \argmin_{\beta} \sum_{e\in \mathcal{E}} \mathbb{E}_{(X,Y)\sim\mu^{(e)}} \\ \mathbb{E}[|Y-\beta^\top X|^2]$?
\end{quote}

\subsection{Computational Barrier}

The main theoretical message this paper delivers is: {\it the problem of finding invariant solutions is intrinsically hard}. In the following, we introduce a decision problem whose fundamental computation complexity is equivalent to causal invariance learning. Denote the boolean operators $\mathrm{AND}$, $\mathrm{OR}$ and $\mathrm{NOT}$ by $\land$, $\lor$ and $\neg$, respectively. We glance at the two questions below. 

\begin{quote}
     A. \it {What does the formula below evaluate? (a) $\btrue$ (b) $\bfalse$}
    \begin{align}
    \label{eq:sat1}
        \big((\btrue \land \btrue) \lor \bfalse\big) \land (\btrue \lor \neg \btrue) \land (\neg \bfalse \lor \neg(\btrue \land \btrue))
    \end{align}
\end{quote}

\begin{quote}
     B. \it {Can we choose $v_1, v_2, v_3, v_4$ in $\{\btrue, \bfalse\}$ to make the result of the formula as $\btrue$? (a) Yes (b) No}
    \begin{align}
    \label{eq:sat2}
    \big((v_1 \land v_4) \lor v_2\big) \land (v_3 \lor \neg v_4) \land (\neg v_2 \lor \neg (v_1 \land v_3))
    \end{align}
\end{quote}
The latter question is an instance of the circuit satisfiability ({\sc CircuitSAT}) problem \citep{karp20reducibility}. The answers to both questions are (a), and \eqref{eq:sat1} offers the unique valid solution to \eqref{eq:sat2} as $(v_1,v_2,v_3,v_4)=(\btrue, \bfalse, \btrue, \btrue)$.

From an intuitive perspective, we argue that the relationship between ``finding the best linear predictor'' and ``finding any non-trivial invariant (causal) prediction'' shares some similarities with the relationship between the two questions posed above. While both scenarios involve the same setting, that is ``boolean formula'' for the second pair and ``linear model'' for the first pair, and may potentially yield the same solution, their computation complexities and hierarchy in recognition tasks differ significantly. The former ones only involve simple arithmetic calculations, are straightforward in thought, and can be solved quickly. In contrast, the latter ones will suffer from inevitable brute force attempts, require complicated reasoning, and necessitate a potentially larger time budget. The latter tasks involve reasoning using the information extracted from the corresponding former perception tasks. 

Formally, consider the testing problem {\sc ExistsLIS-2} using population-level quantities. 
\begin{problem}[{\sc ExistsLIS-2}]
\label{prob:existlis2}
Consider the case of $|\mathcal{E}|=2$. Given the positive definite covariance matrices $\Sigma^{(1)}, \Sigma^{(2)} \in \mathbb{R}^{d\times d}$ with $\Sigma^{(e)} = \mathbb{E}[X^{(e)} (X^{(e)})^\top]$ and the covariance vectors $u^{(1)}, u^{(2)} \in \mathbb{R}^d$ with $u^{(e)} = \mathbb{E}[X^{(e)} Y^{(e)}]$, it asks whether it is possible to find a non-empty prediction-invariant set $S\subseteq [d]$ such that $\beta^{(1,S)} = \beta^{(2,S)} \neq 0$. Here $\beta^{(e,S)}$ is defined in \eqref{ident:lip} and can be arithmetically calculated as $\beta^{(e,S)} = [(\Sigma_S^{(e)})^{-1} u_S^{(e)}, 0_{S^c}]$ provided $\Sigma^{(e)}$ is positive definite thus invertible.
\end{problem}
\cref{prob:existlis2} simplifies the original linear invariance pursuit problem, i.e., estimating $\beta^\star$ or $S^\star$ in \eqref{model:lip}, in several aspects: we consider only two heterogeneous environments to identify $\beta^\star$ when $G\neq \emptyset$, and it only checks the existence of solution. 

As the answer to $\mathsf{Q1}$ in \cref{sec:lip}, this paper shows that the aforementioned simplified {\sc ExistsLIS-2} problem is NP-hard, which is essentially the same as the problem {\sc CircuitSat} with an instance example \eqref{eq:sat2}. Furthermore, the NP-hardness is not because of the existence of exponentially many possible invariant solutions, it remains when $\beta^\star$ is identifiable by \eqref{ident:lip}. Many problems are classified as NP-hard, other examples include {\sc 3Sat}, {\sc MaxClique}, {\sc Partition} \citep{erickson2023algorithms}. The Cook--Levin theorem \citep{karp20reducibility} states that if there exists a polynomial time algorithm to solve any NP-hard problem, then P$=$NP, meaning all the N(ondeterministic-)P(olynomial-time) problems, which is verifiable in polynomial time, are P(olynoimal-time) problems that are solvable in polynomial time. It is suspected, but is still a conjecture \citep{bovet1994introduction, fortnow2021fifty}, that P$\neq$NP. This implies it is unlikely that there exists any polynomial-time algorithms for NP-hard problems. This paper proves the NP-hardness of {\sc ExistsLIS-2} problem and an easier problem with constraint \eqref{ident:lip} by constructing a parsimonious polynomial-time reduction from the {\sc 3Sat} problem, a simplification of {\sc CircuitSat}, to our {\sc ExistsLIS-2} problem. See the formal definition of NP-hardness and reduction in \cref{sec:computation-limits}.

In many statistical problems, though attaining correct variable selection suffers from computational barriers, 
it is possible to construct a computationally efficient and accurate estimator of the continuous parameters of interest. For example, as a convex relaxation of $L_0$ regularized least squares, $L_1$ regularized least squares can obtain $n^{-1/2}$ \citep{bickel2009simultaneous} prediction error rate in general and match the same optimal $n^{-1}$ rate under the additional yet mild restricted eigenvalue (RE) condition \citep{candes2007dantzig}\footnote{The RE condition can be relaxed by the restricted strong convexity condition. In this case, if the covariate is zero-mean Gaussian \citep{raskutti2010restricted} or sub-Gaussian \citep{rudelson2013reconstruction}, optimal estimation error can be obtained by $L_1$ regularization when $|\supp(\beta^\star)| \log p = o(n)$ provided the curvature is bounded from below, i.e., $\lambda_{\min}(\mathbb{E}[XX^\top]) \gtrsim 1$. }. On the other hand, compared with $L_0$ \citep{zhang2012general} penalty, $L_1$ penalty requires a much more restrictive, usually impossible \citep{fan2001variable, zou2006adaptive}, condition to attain variable selection consistency \citep{zhao2006model, meinshausen2006high}. It is natural to ask if obtaining a reasonable prediction error using a computationally efficient algorithm is possible in finding invariant predictions. Our result also says ``No'' if P$\neq$NP.

In summary, this paper proves that consistent variable selection and reasonable prediction error in finding invariant predictions are NP-hard. In the world of P$\neq$NP, this establishes a dilemma between computational and statistical tractability for the invariance pursuit problem, and such an impossibility result has implications for several fields and questions.  

\begin{itemize}[topsep=0pt,itemsep=0pt]
    \item[(a)] It has long been hypothesized that there may exist some intrinsic computation barrier in finding invariant solutions given that the problem has a combinatorial formulation and all the existing provable sample-efficient methods use exhaustive search explicitly or implicitly. 
    It is still open whether finding an invariant solution is fundamentally hard or can be solved by a (still not discovered) computationally efficient algorithm. We offer a definite pessimistic answer to this. 
    \item[(b)] Our established dilemma above shows that pursuing invariance is fundamentally harder than pursuing sparsity. The latter can guarantee a decent prediction error using computationally efficient algorithms under a mild assumption that does not hurt the generality of the problem, and the corresponding estimation error will decrease when we keep increasing $n$. However, these no longer apply to the former. Thus, the relaxation tricks used in the sparsity pursuit like $L_1$ regularization may not be a good fit, and potentially new relaxation techniques should be introduced to pursue invariance. 
\end{itemize} 

\subsection{Our Proposed Method}

This paper proposes a simple method that answers question $\mathsf{Q2}$ with ``Yes'' by achieving a better balance between prediction power and invariance, while partially circumventing the computational barriers as the second part of $\mathsf{Q1}$. Given data from environments $\mathcal{E}$, the population-level estimator with $n=\infty$ is the minimizer of the following objective function
\begin{align*}
    \beta^{k,\gamma} = \argmin_{\beta\in \mathbb{R}^d} \frac{1}{|\mathcal{E}|} \sum_{e\in \mathcal{E}} \mathbb{E}[|Y^{(e)} - \beta^\top X^{(e)}|^2] + \gamma \sum_{j=1}^d w_k^{\mathcal{E}}(j) \cdot |\beta_j|.
\end{align*} It regularizes the pooled least squares using pre-calculated weighted $L_1$ penalty, where the adaptive, data-driven weight $w_k^\mathcal{E}(j)$ on $|\beta_j|$ is the upper bound of the prediction variations across environments $\mathcal{E}$ when incorporating variable $x_j$; see the details in \cref{sec:method}. Here $\gamma$ is the hyper-parameter that trades off predictive power and robustness against spurious signals, and $k$ is the hyper-parameter that controls the computation budget through $w_k^{\mathcal{E}}(j)$. The key features of our proposed estimator are as follows.
\begin{itemize}[itemsep=0pt, topsep=0pt]
\item[(a)] For the computation concern, our proposed estimator provably attains the causal identification, i.e., $\beta^{k,\gamma} = \beta^\star$ for large enough $\gamma$, by paying affordable computation cost (small $k$) under some unknown low-dimensional structure among the variables. On the other hand, by increasing the computation budget $k$ to $p$, our proposal achieves the causal identification under the same assumptions as those in EILLS \citep{fan2023environment}.

\item[(b)] The estimator reaches the goal in $\mathsf{Q2}$ by tuning $\gamma$. When causal identification is attained in (a), it leads to a continuous solution path interpolating the pooled least squares solution with $\gamma=0$ and the causal solution $\beta^\star$ with large enough $\gamma$. For any fixed $\gamma$, it has a certain distributional robustness interpretation in that $\beta^{k,\gamma}$ can be represented as the maximin effects \citep{meinshausen2015maximin,guo2024statistical} over some uncertainty set. 
\end{itemize}

\subsection{Related Works and Our Contribution}
\label{sec:related}

\cite{peters2016causal} first considers \eqref{model:lip} with more distributional constraints for causal discovery. To be specific, they consider doing causal discovery that infers the direct cause of the target response $Y$, using data under different environments \citep{didelez2012direct, meinshausen2016methods}, where in each environment, some unknown interventions are applied to the variables other than $Y$. Under the modularity assumption \citep{scholkopf2012causal}, which is also referred to as autonomy \citep{haavelmo1944probability, aldrich1989autonomy} or stability \citep{dawid2010identifying}, in the SCM framework that the intervention on $X_j$ will only change the conditional distribution of $X_j$ given all its direct causes, the conditional distribution of $Y$ given all its direct causes will remain the same across these different environments. This leads to the following distributional invariance structure if a linear model with exogenous noise is further assumed: $Y^{(e)} = (\beta^\star_{S^\star})^\top X^{(e)}_{S^\star} + \varepsilon$ with $\varepsilon \sim F_\varepsilon \indep X_{S^\star}^{(e)}$ and $\mathbb{E}[\varepsilon] = 0$, where $S^\star$ is the direct cause of the target response $Y$. A hypothesis-test based method is proposed in \cite{peters2016causal} to guarantee $\mathbb{P}(\hat{S} \subseteq S^\star) \ge 1-\alpha$. However, the set $\hat{S}^\infty$ it selects when $n=\infty$  will stand in between $\emptyset$ and $S^\star$, i.e., $\emptyset \subseteq \hat{S}^\infty \subseteq S^\star$, and easily be collapsed to $\emptyset$ in most of the cases when the interventions are not enough. The idea of penalizing least squares using exact invariance regularizer \citep{fan2023environment, gu2024causality} will select variables $\hat{S}^\infty$ satisfying $S^\star \subseteq \hat{S}^\infty \subseteq \bar{S}$ as $n=\infty$ where $\bar{S}$ is the Markov blanket of $Y$, but it will eliminate any of $Y$'s child if it is intervened in a non-degenerate manner. Though causal the solution is, it may lack some predictive power under the circumstances discussed before $\mathsf{Q2}$. The estimator proposed in this paper leverages the invariance principle as an inductive bias for ``soft'' regularization instead of that for ``hard'' structural equation estimation and can alleviate the lack of predictive power in this aspect. 

There are also attempts to attain both computationally and statistically efficient estimation under \eqref{model:lip}. For example, \cite{rothenhausler2019causal, rothenhausler2021anchor} consider the case where the mechanism among all covariate and response variables $(X,Y)$ remain unchanged and linear, while the heterogeneity across environments comes from additive interventions on $X$. Estimators similar to instrumental variable (IV) regression in causal identification are proposed. This idea is further extended \citep{kania2022causal, shen2023causality}, but can not go beyond circumventing the computation barrier by assumptions similar to IV regression. This is conceptually the same as least squares that follow the prior untestable assumptions to pinpoint the unique solution and may suffer from model misspecification. \cite{li2024fairm} studies a similar model with one additional constraint -- the covariance between $X_{S^\star}$ remains the same. A seemingly computation-efficient variable selection method is proposed. However, the additional constraint seems to be superfluous in that it cannot change the NP-hardness of the problem; see \aosversion{Appendix B.1}{\cref{sec:diss1}}. Therefore, there is still a gap in attaining sample-efficient estimation by computation-efficient algorithms under mild assumptions that will not ruin the prior-knowledge blind nature of invariance pursuit. This paper makes progress in this direction. 

There is also a considerable literature on robustifying prediction using the idea of distributionally robust optimization, which finds a predictor that minimizes the worst-case risk on a set of distributions referred to as the uncertain set. The uncertain set is typically a (isotropic) sphere in postulated metric centered on the training distribution. Examples of pre-determined metrics include KL divergence \citep{bagnell2005robust}, $f$-divergence \citep{duchi2021learning} and Wasserstein distance \citep{mohajerin2018data, blanchet2019data}. Such a postulated metric is uninformative which leads to a relatively conservative solution. Our estimator is a distributionally robust estimator with an ellipsoid-shaped uncertainty set. It assigns minimal uncertainty to invariant (causal) directions while allocating greater uncertainty to spurious directions, which balances the robustness and power in a better way. 

The NP-hardness and the conjecture P$\neq$NP are used to derive computation barriers in many statistical problems, mainly about detecting sparse low-dimensional structures in high-dimensional data. For the sparse linear model, \cite{huo2007stepwise} shows finding the global minima of $L_0$ penalized least squares is NP-hard, \cite{chen2014complexity} shows the NP-hardness holds for any $L_q$ loss and $L_p$ penalty with $q\ge 1$ and $p\in [0,1)$, and \cite{chen2017strong} extends it to general convex loss and concave penalty. However, these are computation barriers tailored to specific algorithms, not the fundamental limits of the problem itself. \cite{zhang2014lower} shows when P$\neq$NP, in the absence of the restricted eigenvalue condition, any polynomial-time algorithm can not attain estimation error faster than $n^{-1/2}$, which is attained by $L_1$ regularization but is sub-optimal compared with optimal $n^{-1}$ error. There is also a considerable literature on deriving statistical sub-optimality of computationally efficient algorithms using the reduction from the planted clique problem \citep{brennan2019optimal}, such as sparse principle component \citep{berthet2013complexity, berthet2013optimal, wang2016statistical}, sparse submatrix recovery \citep{ma2015computational}. However, a reasonable error is still attainable using computationally efficient alternatives. As discussed above, this is not the case for pursuing invariance as shown by this paper. 

\medskip
\noindent \textbf{Our Contributions.} The main contributions are as follows:

\begin{itemize}[itemsep=0pt, topsep=0pt]
    \item We establish the fundamental computational limits of finding prediction-invariant solutions in linear models, which is the first in the literature. Our proof is based on constructing a novel parsimonious reduction from the {\sc 3Sat} problem to the {\sc ExistLIS-2} problem. 
    \item A simple estimator is proposed to relax the computational budget and exact invariance pursuit using two hyper-parameters. It allows for provably computational and statistical efficiency estimation of the exact invariant (causal) parameters with mild additional assumptions and also offers flexibility in trade-offing efficiency and invariance (robustness). 
\end{itemize}

\medskip
\noindent \textbf{Organization.} This paper is organized as follows. In \cref{sec:computation-limits}, we introduce the concept of NP-hardness and present our main computation barrier result accompanied by the proofs. In \cref{sec:method}, we propose our method that relaxes the computation budget and conservativeness, illustrate its distributional robustness interpretation, and present the corresponding non-asymptotic result. The proofs for the results in \cref{sec:method} are deferred to the supplement material. \cref{sec:exp} collects the real-world application. 

\medskip
\noindent \textbf{Notations.} We will use the following notations. Let $X\in \mathbb{R}^d, Y\in \mathbb{R}$ be random variables and $x, y$ be their instances, respectively. We let $[m]=\{1,\ldots, m\}$. For a vector $z = (z_1,\ldots, z_m)^\top\in \mathbb{R}^m$, we let $\|z\|_q = (\sum_{j=1}^m |z_j|^q)^{1/q}$ with $q\in [1,\infty)$ be its $\ell_q$ norm, and let $\|z\|_\infty = \max_{j\in [m]} |z_j|$. For given index set $S=\{j_1,\ldots, j_{|S|}\}\subseteq [m]$ with $j_1<\cdots<j_{|S|}$, we denote $[z]_S=(z_{j_1},\ldots, z_{j_{|S|}})^\top \in \mathbb{R}^{|S|}$ and abbreviate it as $z_S$ if there is no ambiguity. We use $A\in \mathbb{R}^{n\times m}$ to denote a $n$ by $m$ matrix, use $A_{S,T}=\{a_{i,j}\}_{i\in S, j\in T}$ to denote a sub-matrix and abbreviate it as $A_{S}$ if $S=T$ and $n=m$. For a $d$-dimensional vector $z$ and $d\times d$ positive semi-definite matrix $A$, we let $\|z\|_A = \sqrt{z^\top A z}$, and let $\lambda_{\min}(A)$ (resp. $\lambda_{\max}(A)$) be the minimum (resp. maximum) eigenvalue of $A$.

We collect data from multiple environments $\mathcal{E}$. For each environment $e\in \mathcal{E}$, we observe $n$ data $\{(X^{(e)}_i, Y_i^{(e)})\}_{i=1}^n$ which are drawn i.i.d. from $\mu^{(e)}$. We denote $\mathbb{E}[f(X^{(e)}, Y^{(e)})]=\int f(x, y)\mu^{(e)}(dx,dy)$ and $\hat{\mathbb{E}}[f(X^{(e)}, Y^{(e)})] = \frac{1}{n} \sum_{i=1}^n f(X_i^{(e)}, Y_i^{(e)})$, and define
\begin{align}
\label{eq:covariance}
    \Sigma^{(e)} = \mathbb{E}[X^{(e)} (X^{(e)})^\top], ~u^{(e)} = \mathbb{E}[X^{(e)} Y^{(e)}], ~\Sigma = \frac{1}{|\mathcal{E}|} \sum_{e\in \mathcal{E}} \Sigma^{(e)}, ~u=\frac{1}{|\mathcal{E}|} \sum_{e\in \mathcal{E}} u^{(e)}.
\end{align}
We assume there is no collinearity, i.e., $\Sigma^{(e)} \succ 0$ such that we can define the population-level best linear predictor constrained on any set $S$ in each environment $e$, $\beta^{(e,S)}:=\argmin_{\supp(\beta)\subseteq S} \mathbb{E}[|Y^{(e)} - \beta^\top X^{(e)}|^2]$, and all the environment, $\beta^{(S)}:=\argmin_{\supp(\beta)\subseteq S} \allowbreak\sum_{e\in \mathcal{E}}  \mathbb{E}[|Y^{(e)} - \beta^\top X^{(e)}|^2]$. Let the pooled least squares loss over all the environments be
\begin{align}
\label{eq:pooled-least-squares}
\mathsf{R}^{\mathcal{E}}(\beta)=\frac{1}{2|\mathcal{E}|} \sum_{e\in \mathcal{E}} \mathbb{E}[|Y^{(e)} - \beta^\top X^{(e)}|^2].
\end{align}


\section{The Fundamental Limit of Computation}
\label{sec:computation-limits}

\subsection{Preliminary: NP-hardness}

We first introduce the idea of decision problem, NP-hardness, and reduction argument. 

\begin{definition}[Decision Problem]
A decision problem $P$ is a problem whose output is 1/0, meaning Yes/No. Let $x$ be an instance of the problem, we use $|x|$ to denote the size of its input and use $\mathcal{X}_P$ to denote the set of all the problem instances. We use $\mathcal{S}_x$ to denote the set of solutions for the problem instance $x$. We use the notation $x\in \mathcal{X}_{P,1}$ if the answer to the instance $x$ is 1(Yes). Clearly, we have $x\in \mathcal{X}_{P,1} \Longleftrightarrow |\mathcal{S}_x| \ge 1$.
\end{definition}

The particular decision problem that we consider is the {\sc 3Sat} problem below.

\begin{problem}[{\sc 3Sat}]
Given a conjunctive normal form (CNF) $\bigwedge_{i=1}^k (l_{i,1} \lor l_{i, 2} \lor l_{i, 3})$ of $k$ clauses, where the literal $l_{i,u}$ is either $v_\ell$ or $\neg v_\ell$ for some boolean variable $v_\ell \in \{\mathrm{True}, \mathrm{False}\}$ with $\ell\in [n]$, it asks if there exists an assignment of the variables such that the entire formula evaluates to $\mathrm{True}$. The size of a problem instance is $k$. $\mathcal{S}_x$ is the set of assignments of $(v_\ell)_{\ell=1}^n$ to let the formula be $\mathrm{True}$.
\label{prob: 3sat}
\end{problem}

We now present an instance of the {\sc 3Sat} problem.

\begin{example}[An Instance of {\sc 3Sat} Problem]
Consider an instance $x$ with $k=9$ clauses, the input is an CNF $f=(v_1\lor v_2 \lor v_3) \land (v_1\lor v_2 \lor \neg v_3) \land (v_1\lor \neg v_2 \lor v_3) \land (v_1\lor \neg v_2 \lor \neg v_3) \land (\neg v_1\lor \neg v_2 \lor v_3) \land (\neg v_1\lor v_2 \lor \neg v_3) \land (\neg v_1\lor \neg v_2 \lor \neg v_3) \land (\neg v_4 \lor v_4 \lor v_2) \land (\neg v_1 \lor v_2 \lor v_4)$ in $n=4$ variables. It is easy to see that $\mathcal{S}_x = \{(\mathrm{True}, \mathrm{False}, \mathrm{False}, \mathrm{True})\}$ and hence the answer to above {\sc 3Sat} instance is 1(Yes).
\end{example}

We also consider a potentially easier variant of {\sc 3Sat} to be used in the section. The problem is potentially easier than {\sc 3Sat} because it pursues the same target under additional non-trivial restrictions. 

\begin{problem}[{\sc 3Sat-Unique}]
The {\sc 3Sat-Unique} problem is the same as {\sc 3Sat} under the promise that the solution is unique if exists, i.e., $\mathcal{X}_{\text{\sc 3Sat-Unique}}=\{x\in \mathcal{X}_{\text{\sc 3Sat}}, |\mathcal{S}_x| \le 1\}$.
\end{problem}

We then introduce the idea of reduction and NP-hardness.

\begin{definition}[Reduction]
\label{def:reduction}
We say $T: \mathcal{X}_P \to \mathcal{X}_Q$ is a deterministic polynomial-time reduction from problem $P$ to problem $Q$ if there exists some polynomial $p$ such that for all $x\in \mathcal{X}_P$, (1) $T(x)$ can be calculated on a deterministic Turing machine with time complexity $p(|x|)$; and (2) $T(x) \in \mathcal{X}_{Q,1}$ if and only if $x \in \mathcal{X}_{P, 1}$.

We say $T: \mathcal{X}_P \to \mathcal{X}_Q$ is a randomized polynomial-time reduction \citep{valiant1985np} from problem $P$ to problem $Q$ if there exists some polynomial $p$ such that (1) $T(x)$ can be calculated on a randomized (coin-flipping) Turning machine with computational complexity $p(|x|)$ for any $x\in \mathcal{X}_P$; (2) For all $x\in \mathcal{X}_P\setminus \mathcal{X}_{P,1}$, $T(x) \notin \mathcal{X}_{Q,1}$; (3) For all $x\in \mathcal{X}_{P,1}$, $\mathbb{P}[T(x) \in \mathcal{X}_{Q,1}] \ge 1/p(|x|)$.
\end{definition}

\begin{definition}[NP-hardness]
We say a problem $P$ is NP-hard under deterministic (resp. randomized) polynomial-time reduction if there exists deterministic (resp. randomized) polynomial-time reduction from the circuit satisfiability problem \citep{karp20reducibility} to problem $P$.
\end{definition}

The NP-hardness of a problem is widely used to measure the existence of the underlying computational barrier for the problem; examples in statistics include sparse PCA under particular regime \citep{berthet2013complexity, berthet2013optimal, wang2016statistical}, sparse regression \citep{zhang2014lower} without restricted eigenvalue condition. The underlying reason why an NP-hard problem $P$ is ``hard'' can be illustrated via the Cook--Levin theorem \citep{karp20reducibility}: the existence of any polynomial-time algorithm for the NP-hard problem under deterministic polynomial-time reduction will assert P$=$NP, which implies any NP problem, defined as the problem whose validness of solution can be verified within polynomial-time, can be solved within polynomial-time. The NP-hardness under randomized polynomial-time reduction can be understood similarly: the existence of any polynomial-time algorithm for such a problem implies any NP problem can be solved within polynomial-time with high probability, that is, for any NP decision problem $P$, we can design a polynomial-time randomized algorithm $\tilde{A}$ such that
\begin{align*}
\forall x\in \mathcal{X}_{P} \setminus \mathcal{X}_{P,1},~ \tilde{A}(x) = 0 \qquad \text{and} \qquad \forall x\in \mathcal{X}_{P, 1},~ \mathbb{P}[\tilde{A}(x)=1] \ge 1-0.01|x|^{-100}.
\end{align*}
If the conjecture ``P$\neq$NP'' holds, then the NP-hardness of a problem naturally implies ``there is no polynomial-time algorithm for the problem''. We introduce the NP-hardness under randomized polynomial-time reduction to characterize the computation barrier of the linear invariance pursuit under identification condition \eqref{ident:lip}. We have the following result for the above two problems.
\begin{lemma}
\label{prop:3sat}
The problem {\sc 3Sat} is NP-hard under deterministic polynomial-time reduction. The problem {\sc 3Sat-Unique} is NP-hard under randomized polynomial-time reduction.
\end{lemma}
\begin{proof}[Proof of \cref{prop:3sat}] The NP-hardness of {\sc 3Sat} follows from \cite{karp20reducibility}, the proof for the NP-hardness of {\sc 3Sat-Unique} can be found in \aosversion{Appendix A.4}{\cref{sec:proof-3sat}}.
\end{proof}

\subsection{The Hardness of Population-level Linear Invariance Pursuit}

When $|\mathcal{E}|=2$, we will show that finding a non-trivial invariant solution using population covariance matrices has a computation barrier similar to the {\sc 3Sat} problem. Moreover, even when $\beta^\star$ and $S^\star$ are identifiable, the computation limit remains in a similar manner to the {\sc 3Sat-Unique} problem. This claim can be rigorously delivered in the following \cref{thm:np-hard-p1}.  Without loss of generality, we assume that $X^{(e)}$ and $Y^{(e)}$ are all zero-mean random variables in each environment. 

\begin{problem}[Existence of Linear Prediction-Invariant Set]
\label{prob1}
    Let $d \in \mathbb{N}^+$ be the dimension of the explanatory covariate, and $E$ be the number of environments. Let $\Sigma^{(1)}, \ldots, \Sigma^{(E)} \in \mathbb{R}^{d\times d}$ be positive definite matrices representing the covariance matrices of $X^{(e)}$, i.e., $\Sigma^{(e)} = \mathbb{E}[X^{(e)} (X^{(e)})^\top]$, and $u^{(1)}, \ldots, u^{(E)}$ be $d$-dimensional vectors representing the covariance between $X^{(e)}$ and $Y^{(e)}$, i.e., $u^{(e)} = \mathbb{E}[X^{(e)}Y^{(e)}]$. In this case, the population-level least squares solutions can be written as $\beta^{(e,S)} = [\beta^{(e,S)}_S, 0_{S^c}]$ with $\beta^{(e,S)}_S=(\Sigma_S^{(e)})^{-1} u_S^{(e)}$ and $\beta^{(S)} = [\beta^{(S)}_S, 0_{S^c}]$ with $\beta^{(S)}_S=(\sum_{e\in \mathcal{E}} \Sigma_S^{(e)})^{-1} (\sum_{e\in \mathcal{E}} u_S^{(e)})$. \\
    We define the problem {\sc ExistLIS} as follows:
    
    \noindent \textbf{$[\mathsf{Input}]$} $\Sigma^{(1)}, \ldots, \Sigma^{(E)}$ and $u^{(1)}, \ldots, u^{(E)}$ satisfying the above constraints.\\
    \noindent \textbf{$[\mathsf{Output}]$} Returns 1(Yes) if there exists $S\subseteq[d]$ such that $\beta^{(e,S)} \equiv \beta^{(S)} \neq 0$; otherwise 0(No).
\end{problem}

We simplify the original problem, that is, unveiling $S^\star$ in \eqref{model:lip}, when $n=\infty$ from two aspects in \cref{prob1}. Firstly, we only use the first-order linear information rather than the full distribution information such that the input of the problem is of $O(d^2)$ when $|\mathcal{E}|=O(1)$. The space of {\sc ExistLIS} can be seen as a ``linear projection'' of the space of the problems that recovering $S^\star$ in \eqref{model:lip} provided $\Sigma^{(e)} \succ 0$. Secondly, we state it as a decision problem rather than a solution-solving problem: it suffices to answer whether a non-trivial invariant set exists instead of pursuing one. For simplicity in this section, we use the terminology ``invariant set'' instead of ``linear prediction-invariant set''. We define the concept of the maximum invariant set to present the same problem under the identification condition \eqref{ident:lip}.

\begin{definition}[Invariant Set and Maximum Invariant Set]
\label{def:mis}
    Under the setting of \cref{prob1}, we say a set $\bar{S}$ is a \underline{invariant set} if $\beta^{(e,\bar{S})} \equiv \beta^{(\bar{S})}$. We say a set $\bar{S}$ is a \underline{maximum invariant set} if it is an invariant set and satisfies
    \begin{align}
    \label{eq:max-invariant-set} 
        \forall S\subseteq[d], ~~ \mathrm{either}~ \left(\beta^{(S\cup \bar{S})} = \beta^{(\bar{S})}\right) ~\mathrm{or}~
        \left(\sup_{e,e'\in [E]} \|\beta^{(e,S)} - \beta^{(e',S)}\|_2 > 0\right) .
    \end{align}
\end{definition}


\begin{problem}[Existence of Linear Invariant Set under Identification]
\label{prob2}
    Problem {\sc ExistLIS-Ident} is defined as the same problem as {\sc ExistLIS} with the additional constraint that there exists a maximum invariant set $S^\dagger$.
\end{problem}

Note that $S^\dagger$ can be an empty set, under which the corresponding problem instance does not have non-trivial invariant solutions. Observe that the boolean formula $(a \lor b)$ is equivalent to the statement (if $\neg a$ then $b$). As required by \eqref{eq:max-invariant-set}, an invariant set $\bar{S}$ is a maximum invariant set if incorporating any variable that enhances the prediction performance will lead to shifts in best linear predictions. Therefore, the existence of the maximum invariant set defined in \cref{def:mis} is just a restatement of the identification condition \eqref{ident:lip}, that is, $\bar{S}$ is a maximum invariant set if and only if $S^\star =\bar{S}$ satisfies \eqref{model:lip} and \eqref{ident:lip} simultaneously. 

\cref{prob2} is an easier version of the problem of recovering $S^\star$ in \eqref{model:lip} with the identification constraint \eqref{ident:lip} in population $n=\infty$. The following example gives an instance of the problem {\sc ExistLIS-Ident}. This example also indicates that the maximum invariant set may not be unique, but all the maximum invariant sets yield the same prediction performance. 

\begin{example}[An Instance of {\sc ExistLIS-Ident} Problem] Consider an instance with $d=4$, $E=2$ and input
\begin{align*}
    (\Sigma^{(1)}, \Sigma^{(2)}) = \left(\begin{bmatrix}
        1 & 0 & \frac{2}{\sqrt{15}} & 0 \\
        0 & 1 & \frac{1}{\sqrt{15}} & 0 \\
        \frac{2}{\sqrt{15}} & \frac{1}{\sqrt{15}} & \frac{3}{5} & 0 \\
        0 & 0 & 0 & 1 \\
    \end{bmatrix}, \begin{bmatrix}
        1 & 0 & \frac{2}{\sqrt{5}} & 0 \\
        0 & 1 & \frac{1}{\sqrt{5}} & 0 \\
        \frac{2}{\sqrt{5}} & \frac{1}{\sqrt{5}} & \frac{7}{5} & 0 \\
        0 & 0 & 0 & 1 \\
    \end{bmatrix}\right),  (u^{(1)}, u^{(2)})= \left(\begin{bmatrix}
        2\\ 1\\ 2\sqrt{\frac{3}{5}} \\ 0
    \end{bmatrix}, \begin{bmatrix}
        2\\ 1\\ \frac{6}{\sqrt{5}}\\ 0
    \end{bmatrix}\right).
\end{align*} It can be seen as a ``linear projection'' of the following data-generating process with $e=\{1,2\}$ and independent standard normal random variables $\varepsilon_0,\ldots, \varepsilon_4$:
\begin{align*}
    X_1^{(e)} &\gets \varepsilon_1, ~~X_2^{(e)} \gets \varepsilon_2, ~~X_4^{(e)} \gets \varepsilon_4, \\
    Y^{(e)} &\gets 2\cdot X_1^{(e)} + X_2^{(2)} + \varepsilon_0, \\
    X_3^{(e)} &\gets \frac{(\sqrt{3})^{e-2} Y^{(e)} + \varepsilon_3}{\sqrt{5}}.
\end{align*}
It is easy to see that the sets $\emptyset, \{1\}, \{2\}, \{4\}, \{1,2\}, \{1,4\}, \{2,4\}, \{1,2\}, \{1,2,4\}$ are all invariant sets, while the sets $\{1,2\}$ and $\{1,2,4\}$ are maximum invariant sets. 
\end{example}

From the perspective of a computational problem, the existence of a maximum invariant set offers non-trivial constraints on the problem and one can construct a model where this condition fails to hold; see \cref{ex:fail} below. On the other hand, the non-existence of a maximum invariant set rarely happens under the causal discovery setting. To be specific, under the setting of the structural causal model with intervention on $X$, it is known from Theorem 3.1 in \cite{gu2024causality} that a maximum invariant set always exists if the intervention is non-degenerate, which occurs with probability $1$ under suitable measure on the intervention. 

\begin{example}[An Instance of {\sc ExistLIS} that is not {\sc ExistLIS-Ident}]
\label{ex:fail}
Consider the model of Example 4.1 in \cite{fan2023environment} with $s^{(1)}=1/2$ and $s^{(2)}=2$, that is, the SCMs in environment $e\in \{1,2\}$ are
\begin{align*}
    X^{(e)}_1 &\gets \sqrt{0.5} \varepsilon_1 \\
    Y^{(e)} &\gets X_1^{(e)} + \sqrt{0.5} \varepsilon_0 \\
    X_2^{(e)} &\gets 2^{2e-3} Y^{(e)} + \varepsilon_2 
\end{align*} with $\varepsilon_0,\ldots, \varepsilon_2$ are i.i.d. standard Gaussian random variables. It is easy to check that the sets $\emptyset, \{1\}, \{2\}$ are all invariant sets but none of them satisfies the second constraint \eqref{eq:max-invariant-set}, and the set $\{1,2\}$ is not an invariant set. So there does not exist a maximum invariant set. 
\end{example}

Given $\mathcal{X}_{\text{{\sc ExistLIS-Ident}}} \subsetneq \mathcal{X}_{\text{{\sc ExistLIS}}}$, {\sc ExistLIS} may be potentially harder than {\sc ExistLIS-Ident}. We will establish NP-hardness to both {\sc ExistLIS} and {\sc ExistLIS-Ident} to rule out the possibility that the computational hardness is because of nonidentifiability, or in other words, computational difficulty can be resolved when $S^\star$ is identifiable in \eqref{model:lip} by \eqref{ident:lip}.

\begin{theorem}
\label{thm:np-hard-p1}
When $E=2$, the problem {\sc ExistsLIS} is NP-hard under deterministic polynomial-time reduction; the problem {\sc ExistsLIS-Ident} is NP-hard under randomized polynomial-time reduction.
\end{theorem}

\cref{thm:np-hard-p1} states that there exist certain fundamental computational limits under the problem of pursuing a linear invariant prediction: the difficulties are intrinsically inherited in the problem itself -- there does not exist a polynomial-time algorithm to test whether there exists a non-trivial invariant prediction in general if P$\neq$NP.

\begin{remark}[NP-hardness under More Restrictive Conditions]
It is worth noticing that the underlying computational barrier is attributed to the nature of the problem, i.e., pursuing invariance, instead of artificial and technical difficulties. Such a barrier will remain for other cousin models and models under more restrictive conditions. Examples include (1) finding a prediction with stronger invariance condition like distributional invariance in \cite{peters2016causal}; (2) problems with row-wise sparse covariance matrices where all the covariance matrices only have constant-level non-zero entries in each row; (3) problems with well-separated heterogeneity in that the variations in prediction are large for all the non-invariant solutions. See the rigorous statement and discussion in \aosversion{Appendix A}{\cref{appendix:comp}}.
\end{remark}

We will show a much easier problem with fixed $(\Sigma^{(1)}, u^{(1)})$ structure is NP-hard.

\subsection{Proof of \cref{thm:np-hard-p1}}

The following lemma claims that we can construct a \emph{parsimonious} polynomial-time reduction from the well-known problem {\sc 3Sat} to our problem {\sc ExistLIS} that preserves the number of solutions. Given an instance $x$ of the {\sc 3Sat} problem stated in Problem \ref{prob: 3sat}, we let $\mathcal{S}_x = \{v\in \{\mathrm{True}, \mathrm{False}\}^{n}: v \text{ let the formula to be } \mathrm{True}\}$ be its set of solutions. Given an instance $y$ of the {\sc ExistLIS} problem with $E=2$, we define its solution set $\mathcal{S}_y \subseteq [d]$ as the set of all the $S$ satisfying $\beta^{(1,S)} = \beta^{(2, S)} \neq 0$. We let $k$ be the number of clauses in the instance $x$ and $d$ be the number of covariance in the instance $y$ and omit the dependency on $x$ (resp. $y$) in $k$ (resp. $d$) for presentation simplicity. 
 For an integer $m$, We let $1_{m}$ be a $m$-dimensional vector with all entries being $1$, and let $I_m$ be a $m\times m$ identity matrix.  

Unlike the standard reduction argument whose goal is to find a polynomial time reduction $T: \mathcal{X}_{\text{\sc 3Sat}} \to \mathcal{X}_{\text{\sc ExistLIS}}$ such that $\indicator\{|\mathcal{S}_x|>0\} = \indicator\{|\mathcal{S}_{T(x)}|>0\}$, we will construct a parsimonious reduction satisfying $|\mathcal{S}_x|=|\mathcal{S}_{T(x)}|$. This finer construction transfers the promise of the unique solution in {\sc 3Sat-Unique} to the promise of the identification in {\sc ExistLIS-Ident}. 

\begin{lemma}
\label{lemma:reduction}
We can construct a parsimonious polynomial-time reduction from {\sc 3Sat} to {\sc ExistLIS}: for each instance $x$ of problem {\sc 3Sat} with input size $k$, we can transform it to $y=T(x)$ of problem {\sc ExistLIS} within polynomial-time with $d=7k+1$ such that $|\mathcal{S}_y|= |\mathcal{S}_x|$. 
\end{lemma}
\begin{proof}[Proof of \cref{lemma:reduction}]
We construct the reduction as follows. Let $x$ be any {\sc 3Sat} instance with $k$ clauses. Without loss of generality, we assume that each variable has appeared at least once in some clause. For each clause, we use action ID in $\{0,\ldots, 7\}$ to represent the assignment for it. For example, for the clause $v_1 \lor \neg v_2 \lor \neg v_5$ and the action ID $6$ with binary representation $110$ means we let $v_1=\mathrm{True}$, $\neg v_2=\mathrm{True}$ and $\neg v_5=\mathrm{False}$. One will not adopt action ID 0 in a valid solution because a {\sc 3Sat} valid solution should let each clause evaluate to $\mathrm{True}$. For arbitrary $i, i'\in [k]$ and $t, t'\in [7]$, we say the action ID $t$ in clause $i$ \emph{contradicts} the action ID $t'$ in clause $i'$ if and only if $t$ will assign a boolean variable to be $\mathrm{True}$ (resp. $\mathrm{False}$) while $t'$ will assign the same boolean variable to be $\mathrm{False}$ (resp. $\mathrm{True}$). In the proof, we use $i, i'$ to represent the index in $[k]$, and use $j, j'$ to represent the index in $[   d]$. 

We construct the problem $y$ as follows: we set $d=7k+1$, and use fixed first environment $(\Sigma^{(1)}, u^{(1)})=(I_d, 1_d)$. For the second environment, we pick
\begin{align*}
    \Sigma^{(2)} = \begin{bmatrix}
        5d I_{7k} + A & \frac{1}{2} \cdot 1_{7k} \\
        \frac{1}{2} \cdot 1_{7k}^\top & 5d
    \end{bmatrix} \qquad \text{and} \qquad u^{(2)} = \begin{bmatrix}
        (5d + \frac{1}{2}) \cdot 1_{7k}\\
        5d + \frac{1}{2} k
    \end{bmatrix},
\end{align*}
where the $7k$ by $7k$ symmetric matrix $A$ is defined as
\begin{align}
\label{eq:reduction-a-mat}
    A_{7(i-1)+t,7(i'-1)+t'} = \begin{cases}
        \indicator\{t\text{ contradicts itself}\} & \qquad i=i' \text{ and } t=t' \\
        1 & \qquad i=i' \text{ and } t\neq t' \\
        1 & \qquad i\neq i' \text{ and } t \text{ contradicts } t' \\
        0 & \qquad \text{otherwise}
    \end{cases}
\end{align} for any $i,i'\in [k]$ and $t,t' \in [7]$.  It is easy to verify that $\Sigma^{(1)}$ and $\Sigma^{(2)}$ are all positive definite matrices and it is a deterministic polynomial-time reduction. Indeed, one has $\lambda_{\min}(\Sigma^{(e)}) \ge 1$ for any $e\in [2]$. By definition, $S \in \mathcal{S}_y$ if and only if $\beta^{(1,S)} = \beta^{(2,S)}$ with $|S| \ge 1$. 

The intuitions behind the constructions are as follows: (a) the construction of $(\Sigma^{(1)}, u^{(1)})$ is to enforce the entries in the valid solutions, i.e., $\beta^{(S)}$ with $S\in \mathcal{S}_y$, being either 0 or 1; (b) the positive non-integer $\frac{1}{2}$ together with the last column of $\Sigma^{(2)}$ is to make sure $d\in S$ for any $S\in \mathcal{S}_y$, which further let $|S| = k+1$ for any $S\in \mathcal{S}_y$; (c) the construction of $A$ is to connect any valid $S\in \mathcal{S}_y$ to a valid solution $v\in \mathcal{S}_x$ in a bijective manner. The above intuitions can be formally stated as follows: the first claim (a) follows directly from our construction of $(\Sigma^{(1)}, u^{(1)})$, we defer technical verification of (b) and (c) to the end of the proof.
\begin{align}
 \label{eq:equiv-reduction}
 \begin{split}
     S \in \mathcal{S}_y ~~~&\overset{(a)}{\Longleftrightarrow} ~~~ S\neq \emptyset \text{ and }\beta^{(2,S)} = \beta^{(1,S)} \text{ with } \beta^{(1,S)}_j = \indicator\{j \in S\} \\
     &\overset{(b)}{\Longleftrightarrow} ~~~  S = \mathring{S} \cup \{d\} ~\text{with}~|\mathring{S}|=k ~\text{and}~A_{j,j'} = 0~~\forall j,j'\in \mathring{S}\subseteq[7k]\\
     &\overset{(c)}{\Longleftrightarrow} ~~~ S = \mathring{S} \cup \{d\} ~\text{where}~ \mathring{S} = \{7(i-1)+a_i\}_{i=1}^k\text{ with }a_i \in [7]\text{ s.t. adopting } \\
     &~~~~~~~~~~~~ \text{action ID } a_i \text{ in clause } i\in [k] \text{ will lead to a valid solution }v\in \mathcal{S}_x\text{.} 
\end{split}
\end{align}

Based on \eqref{eq:equiv-reduction}, for any $v \in \mathcal{S}_x$, we can find a corresponding $S\in \mathcal{S}_y$: Let $v \in \{\mathrm{True},\mathrm{False}\}^n$ be the assignments of the variables and $a_i$ be the corresponding action ID induced by $v$. Then it follows from \eqref{eq:equiv-reduction} that $S=\{d\}\cup \{7(i-1)+a_i\}_{i=1}^k \in \mathcal{S}_y$. On the other hand, for any $S\in \mathcal{S}_y$, we can also find a corresponding $v\in \mathcal{S}_x$ by \eqref{eq:equiv-reduction}. Note the mapping between $\mathcal{S}_y$ and $\mathcal{A}=\{(a_i)_{i=1}^k: (a_i)_{i=1}^k\text{ is induced by some solution } v\in \mathcal{S}_x\}$ and the mapping between $\mathcal{A}$ and $\mathcal{S}_x$ are all bijective maps. So we can conclude that $|\mathcal{S}_x| = |\mathcal{A}| = |\mathcal{S}_y|$.

\noindent \underline{\emph{Proof of \eqref{eq:equiv-reduction} (b)}.} The direction $\Leftarrow$ is obvious. For the $\Rightarrow$ direction, we first show that $d\in S$ using the proof by contradiction argument. Suppose $|S| \ge 1$ but $d\notin S$, we pick $j \in S$, then
\begin{align*}
    \left[\Sigma^{(2)}_S \beta_{S}^{(2,S)}\right]_j = 5d + \sum_{j'=1}^{7k} A_{j, j'} \indicator\{j'\in S\} \neq 5d + \frac{1}{2} = u_j^{(2)}
\end{align*} where the first equality follows from the assumption $\beta^{(2,S)}_j = \beta^{(1,S)}_j = \indicator\{j\in S\}$ and $d\notin S$, and the inequality follows from the fact that $A\in \{0,1\}^{7k\times 7k}$ hence the L.H.S. is an integer. This indicates that $\beta^{(1,S)} \neq \beta^{(2,S)}$ if $|S| \ge 1$ and $d\notin S$. Given $d\in S$, we then obtain
\begin{align*}
    5d + \frac{1}{2} k = u_{d}^{(2)} = \left[\Sigma^{(2)}_S \beta_{S}^{(2,S)}\right]_{|S|} = 5d + \frac{1}{2} \sum_{j'=1}^{7k} 1\{j'\in S\} = 5d + \frac{1}{2} (|S| - 1),
\end{align*} which implies that $|S|=k+1$. Now we still have the constraint $u^{(2)}_{\mathring{S}}= \Sigma^{(2)}_{\mathring{S}} \beta^{(2,S)}_{\mathring{S}} + \frac{1}{2}\cdot 1_k$. The last claim $A_{j',j} = 0$ for any $j',j\in \mathring{S}$ then follows from this by observing that \begin{align*}
\left(5d+\frac{1}{2}\right) \cdot 1_{k}= \Sigma^{(2)}_{\mathring{S}} 1_{k} + \frac{1}{2} \cdot 1_{k}  ~\Longrightarrow~ A_{\mathring{S}} 1_k = 0 ~\overset{(i)}{\Longrightarrow}~ A_{j',j}=0 ~~\forall j',j\in \mathring{S}
\end{align*} where $(i)$ follows from the fact that $A\in \{0,1\}^{7k\times 7k}$.

\noindent \underline{\emph{Proof of \eqref{eq:equiv-reduction} (c)}.} Turning to (c). For the $\Rightarrow$ direction, $\mathring{S} = S\setminus \{d\}$ admits the form $\mathring{S}=\{7(i-1)+a_i\}_{i=1}^k$ with $a_i \in [7]$ follows from the fact that $|\mathring{S}|=k$ and for each $i\in [k]$, there should be exactly one index $7(i-1)+r$ for some $r\in [7]$ because we have $A_{7(i-1)+t,7(i-1)+t'}=1$ provided $t\neq t'$. The satisfiability of the variable assignment induced by $(a_i)_{i=1}^k$ can be realized by setting the variables based on the action ID $a_i$ starting from $i=1$ to $i=k$. The above procedure has no conflicts because if the conflicts between the assignment of a boolean variable at clause $i$ and that at clause $i'$ will lead to $A_{7(i-1)+a_i, 7(i'-1)+a_{i'}}=1$ by the definition of matrix $A$, which is contrary to the condition $A_{\mathring{S}}=0$. For the $\Leftarrow$ direction, the claim $|\mathring{S}| = k$ is obvious. For any $j, j' \in \mathring{S}$, one can write $j=7(i-1)+a_i$ and $j'=7(i'-1)+a_{i'}$. When $j=j'$, $A_{j,j'}=A_{j,j}=0$ follows from the fact that $v \notin \mathcal{S}_x$ if $a_i$ contradicts itself. We use proof by contradiction when $j\neq j'$: if $A_{j,j'} = 1$, then the action $a_i$ for clause $i$ will contradict the action $a_{i'}$ for clause $i'$ by the definition of $A$, this is contrary to the fact that the actions $(a_i)_{i=1}^k$ lead to a valid solution $v\in \mathcal{S}_x$. 
\end{proof}

The NP-hardness of {\sc ExistLIS} follows from \cref{lemma:reduction} and \cref{prop:3sat}. 

Now we are ready to establish the NP-hardness of {\sc ExistLIS-Ident}. Given the problem {\sc 3Sat-Unique} is NP-hard under randomized polynomial-time reduction by \cref{prop:3sat}, it suffices to show that we can reduce any {\sc 3Sat-Unique} problem $x$ with input size $k$ to a {\sc ExistLIS-Ident} problem $y$ with input size $d=7k+1$ under deterministic polynomial-time reduction. We let $y$ be the problem constructed from $x$ in \cref{lemma:reduction}. Now it suffices to show that $y$ is {\sc ExistLIS-Ident}, that is, $y$ satisfies the constraint in \cref{prob2}. Note that $|\mathcal{S}_y| = |\mathcal{S}_x| \in \{0,1\}$ by our parsimonious reduction in \cref{lemma:reduction} and the promise in {\sc 3Sat-Unique} problem $x$, we consider the following two cases.

\noindent \emph{Case 1. $|\mathcal{S}_x|=0$}: We claim that $S^\dagger = \emptyset$ is the maximum invariant set. In this case, $|\mathcal{S}_x| = 0$ by our reduction construction in \cref{lemma:reduction}. This implies that for all the $S\subseteq [d]$ with $|S|\ge 1$, the condition $\beta^{(e,S)} \equiv \beta^{(S)}$ will not hold, which validates \eqref{eq:max-invariant-set} for $S^\dagger$.

\noindent \emph{Case 2. $|\mathcal{S}_x|=1$}: We claim that $S^\dagger = \tilde{S}$ with $\tilde{S}$ being picked from $v \rightarrow \tilde{S}$, where $v$ is the unique solution in $x$, and $\tilde{S}$ is the mapped solution by our constructed reduction in \cref{lemma:reduction} satisfying $|\tilde{S}|=k+1$ and $\beta^{(1,\tilde{S})}_S=1_{k+1} = \beta^{(2,\tilde{S})}_S$. Given $\mathcal{S}_y = \{\tilde{S}\}$, we can claim that $\|\beta^{(1,S)} - \beta^{(2,S)}\|_2 > 0$ for any $S\notin \{\emptyset, \tilde{S}\}$, this verifies \eqref{eq:max-invariant-set} for $S^\dagger$.  \qed

\subsection{Hardness of Finding Approximate Solutions with Error Guarantees}

The claim in \cref{thm:np-hard-p1} indicates a computational barrier exists in finding an exact invariant set. At first glance, it does not rule out the possibility that there exists some polynomial-time algorithm that can find an approximate solution whose prediction is relatively close to one of the non-trivial invariant ones. The construction in \cref{thm:np-hard-p1} implicitly implies this, as demonstrated in \cref{coro:est-d4}. As a by-product, \cref{coro:est-d4} also rules out the possibility of finding a non-trivial invariant solution if one exists, as it allows for estimation errors. 

\begin{problem}
\label{prob:est-d4}
Consider the same setting as \cref{prob1} with $E=2$ and suppose further $Y^{(e)}=(\beta^{(e, [d])})^\top X^{(e)}$, i.e., there is no intrinsic noise. 

\noindent \textbf{$[\mathsf{Input}]$} $(\Sigma^{(1)}, \Sigma^{(2)})$ and $(u^{(1)}, u^{(2)})$ as in \cref{prob1}.

\noindent \textbf{$[\mathsf{Output}]$} Return a $d$-dimensional vector $\bar{\beta}$: $\bar{\beta}$ should be an approximate solution to any of the non-trivial invariant solutions if there exists a non-trivial invariant solution, that is
\begin{align}
\label{eq:est-d4}
\inf_{S: \beta^{(e,S)} \equiv \beta^{(S)} \neq 0} \frac{\|\bar{\beta} - \beta^{(S)} \|_{\Sigma}^2}{\sum_{e\in [2]} \mathbb{E}[|Y^{(e)}|^2]} < (20 d)^{-1} ~~~\text{if}~~~ \{S: \beta^{(1,S)} = \beta^{(2,S)} \neq 0\} \neq \emptyset;
\end{align}
$\bar{\beta}$ can be an arbitrary $d$-dimensional vector otherwise.

\end{problem}

\begin{corollary}
\label{coro:est-d4}
If \cref{prob:est-d4} can be solved by a polynomial-time algorithm, then {\sc 3Sat} can also be solved by a polynomial-time algorithm.
\end{corollary}

Moreover, in the construction in \cref{lemma:reduction}, the solutions are well-separated: whenever variable selection is incorrect, the resulting predictions in the two environments are not very close, and the pooled prediction also deviates from any invariant predictions; see the formal claims in \eqref{eq:gap-hetero} and \eqref{eq:gap-est-error}, respectively. The inequality \eqref{eq:gap-hetero} also rules out the possibility of finding a $o(d^{-4})$-approximate invariant set in a computationally efficient manner. 

\begin{lemma}[Relative Estimation Error Gap]
\label{lemma:gap}
In the constructed instance in \cref{lemma:reduction}, if we let $Y^{(e)}=(\beta^{(e, [d])})^\top X^{(e)}$, then the following holds,
\begin{align}
    \forall S, S^\dagger \subseteq [d], \qquad & \frac{\|\beta^{(S)} - \beta^{(S^\dagger)}\|_{\Sigma}^2}{\sum_{e\in [2]} \mathbb{E}[|Y^{(e)}|^2]} \in [\indicator\{S\neq S^\dagger\} (40d)^{-1}, 1] \label{eq:gap-est-error}\\
    \forall S \subseteq [d], \qquad & \frac{\sum_{e\in [2]} \|\beta^{(S)} - \beta^{(e,S)}\|_{\Sigma^{(e)}}^2}{\sum_{e\in [2]} \mathbb{E}[|Y^{(e)}|^2]} \in \{0\}\cup [(10d)^{-4}, 1] \label{eq:gap-hetero}
\end{align}
\end{lemma}

\begin{remark}[Dilemma between Statistical and Computational Tractability]
\label{remark:consistent}
    One can choose either the relative distance to the closest non-trivial invariant solution $\delta_1$, or the relative prediction variation defined in the L.H.S. of \eqref{eq:gap-hetero} $\delta_2$ as the ``estimation error'' of interests. 
    If P$\neq$NP, taking all the polynomial-time algorithms into consideration, \cref{coro:est-d4} claims that the worst-case estimation error $\delta_1$ is lower bounded by $(20d)^{-1}$, and \cref{lemma:gap} shows that the worst-case estimation error $\delta_2$ is lower bounded by $(10d)^{-4}$. A finer construction in \aosversion{Appendix A}{\cref{appendix:comp}} improves the error lower bounds in \eqref{eq:est-d4}, \eqref{eq:gap-est-error} and \eqref{eq:gap-hetero} to be $d^{-\epsilon}$ for any fixed $\epsilon>0$. Given that our theorem is stated at a population level, and one can estimate all the $\beta^{(e,S)}$ uniformly well provided $n \gtrsim \mathrm{poly}(d)$, we can claim that the statistical estimation error can be arbitrarily slow with polynomial-time algorithms if P$\neq$NP. 
\end{remark}

\begin{proof}[Proof of \cref{coro:est-d4}]
We use the same reduction as in \cref{lemma:reduction}. For {\sc 3Sat} instance $x$, we let $y=T(x)$ be the constructed {\sc ExistLIS} instance in \cref{lemma:reduction}. Let $\bar{\beta}$ be the output required by \cref{prob:est-d4} in the instance $y$, and $\tilde{S} = \{j: \bar{\beta}_j \ge 0.5\}$. 
Following the notations therein, we claim that
\begin{align}
\label{eq:red2}
    \tilde{S} \in \mathcal{S}_y \qquad \overset{(a)}{\Longleftrightarrow} \qquad |\mathcal{S}_y| \ge 1 \qquad {\Longleftrightarrow} \qquad x\in \mathcal{X}_{\text{{\sc 3Sat}}, 1}
\end{align} Therefore, if an algorithm $\mathsf{A}$ can take \cref{prob:est-d4} instance $y$ as input and return the desired output $\hat{\beta}(y)$ within time $O(p(|y|))$ for some polynomial $p$, then the following algorithm can solve {\sc 3Sat} within polynomial time: for any instance $x$, it first transforms $x$ into $y=T(x)$, then use algorithm $\mathsf{A}$ to solve $y$ and gets the returned $\bar{\beta}$, and finally output $\indicator\{\tilde{S} \in \mathcal{S}_y\}$. 

It remains to verify $(a)$: the $\Rightarrow$ direction is obvious. For the $\Leftarrow$ direction, suppose $|\mathcal{S}_y| \ge 1$, the estimation error guarantee in \cref{prob:est-d4} indicates that
\begin{align*}
    \|\hat{\beta} - \beta^{(S^\dagger)}\|_\infty \le \|\bar{\beta} - \beta^{(S^\dagger)}\|_2 \le \sqrt{\frac{\|\bar{\beta} - \beta^{(S^\dagger)}\|_{\Sigma}^2}{\lambda_{\min}(\Sigma)}} \overset{(i)}{<} \sqrt{\frac{(20d)^{-1} 10d^2}{2d}} \le \frac{1}{2}
\end{align*} for some $S^\dagger \in \mathcal{S}_y$. Here $(i)$ follows from the the error guarantee \eqref{eq:est-d4}, and the facts $\lambda_{\min}(\Sigma) \ge 0.5\lambda_{\min}(\Sigma^{(2)}) \ge 2d$ and $\sum_{e\in [2]} \mathbb{E}[|Y^{(e)}|] \le 10d^2$ in \aosversion{(C.2)}{\eqref{eq:variance-upper-bound}} and \aosversion{(C.1)}{\eqref{eq:eigen-ul-bound}}, respectively. This further indicates $\tilde{S} = S^\dagger$ by the fact that $\beta^{(S^\dagger)}_j = \indicator\{j\in S^\dagger\}$ for any $j\in [d]$.
\end{proof}

\subsection{Remarks and Lessons from \cref{thm:np-hard-p1}}

The fundamental limits delivered in \cref{thm:np-hard-p1} assert that realizing both computationally and statistically efficient estimation is impossible unless P$=$NP or simplifying the original problem. The latter may result in restrictive applicability. 

Two remarks on the severity of the computational barrier are worth mentioning. Firstly, in the world of P$\neq$NP, any polynomial-time algorithm can not attain certain estimation accuracy $d^{-\epsilon}$ for arbitrary fixed $\epsilon>0$ by \cref{remark:consistent}. This indicates that the computational barrier for pursuing invariance is more severe than that for other estimation problems such as pursuing sparsity \citep{zhang2014lower, wang2016statistical} in which polynomial-time algorithms can obtain a sub-optimal but still decent rate. Secondly, the computational barrier is due to pursuing invariance itself rather than picking from exponentially many invariant solutions based on some criterion or the non-identifiability of the problem. In fact, the computational barrier remains under the promise of one \emph{unique} invariant solution in \cref{thm:np-hard-p1}. 

The results and construction in \cref{thm:np-hard-p1} also imply that the computation barrier will remain under some typical potential strategies under the worst case: the construction of the identity $\Sigma^{(1)}$ implies that perfect orthogonal covariance in $|\mathcal{E}|-1$ environments will not help. Secondly, the construction of $x_d$ indicates that under the worst case, searching all the variable sets with cardinality less than $r$ cannot furnish any insights on determining whether there are invariant sets whose cardinalities are larger than or equal to $r$. Finally, a finer construction in \aosversion{Appendix A}{\cref{appendix:comp}} asserts that further imposing row-wise constant-level sparsity on all the covariance matrices will not help, or in other words, the computation difficulty is not due to the dense covariance structure.


\section{Regularization by Environment Prediction Variation}
\label{sec:method}

The results in \cref{sec:computation-limits} indicate that consistent estimation with polynomial-time algorithms is impossible under the worst-case scenario. Such a worst-case hardness remains when there is (1) perfect orthogonality in one environment, and (2) near-perfect sparsity across different environments. 
In \cref{sec:orthon}, we first impose one additional restrictive assumption, see how the computational barrier can be resolved, and derive the distributional robustness interpretation of our proposed estimator under this assumption. \cref{sec:population-est} further demonstrates the general estimator and establishes the corresponding causal identification and distributional robustness result when $n=\infty$. The finite sample estimator and the non-asymptotic results are presented in \cref{sec:empirical-est}. 

Without loss of generality, we assume the covariate is non-degenerate and (pooled) normalized.

\begin{condition}[Non-collinearity and Normalization] \label{cond:normalize}
Assume $\Sigma^{(e)} \succ 0$ for any $e\in \mathcal{E}$. Recall the definition in \eqref{eq:covariance}, we have $\Sigma_{j,j}=1$ for any $j\in [d]$.
\end{condition}

\subsection{Warmup: Orthogonal Important Covariate}
\label{sec:orthon}

Let us first impose an additional restrictive assumption \cref{cond:ortho} in the model \eqref{model:lip} and see how the computational barrier can be circumvented under this condition. In the following \cref{sec:population-est}, we shall consider a more general relaxation regime and establish a tradeoff between the additional assumption and computational complexity.

\begin{condition}
\label{cond:ortho}
For all $e\in \mathcal{E}$, $\Sigma^{(e)}_{i,j} = 0$ for any $i, j\in S^\star$ with $i\neq j$. 
\end{condition}

Recall the definition of $\beta^{(e,S)}$ and $\beta^{(S)}$.  If \cref{cond:ortho} holds, then under \eqref{model:lip} and \eqref{ident:lip} $S^\star$ can be simplied as
\begin{align*}
	S^\star = \left\{j: \forall e\in \mathcal{E}, \beta^{(e,\{j\})}_j \equiv \beta^{(\{j\})}_j\right\}
\end{align*} 
that involves only marginal regression coefficients,
where $\beta^{(\{j\})}_j$ stands for the pooled effect by simply using the $j$-th variable as the predictor. This means under \cref{cond:ortho}, one can enumerate $j\in [d]$ and screen out those $X_j$ with varying marginal regression coefficients, i.e., $X_j$ with $r_j^{(e)} \neq r_j^{(e')}$ for some $e,e'\in \mathcal{E}$, where $r_j^{(e)}=\mathbb{E}[X_j^{(e)} Y^{(e)}] / \mathbb{E}[|X^{(e)}_j|^2]$. The survived variables will furnish $S^\star$. Turning to the empirical counterpart, it is a multi-environment version of the sure-screening \citep{fan2008sure}.

The above procedure is still of a discontinuity style. Recall $\mathsf{R}^{\mathcal{E}}(\beta)$ in \eqref{eq:pooled-least-squares}, the main idea motivates minimizing the following penalized least squares
\begin{align}
\label{eq:obj-k1}
	\mathsf{Q}_{1,\gamma}(\beta) = \mathsf{R}^{\mathcal{E}}(\beta) + \gamma \overbrace{\sum_{j=1}^d |\beta_j| \underbrace{\sqrt{\frac{1}{|\mathcal{E}|} \sum_{e\in \mathcal{E}} \Sigma_{j,j}^{(e)} \left(\beta^{(e,\{j\})}_j - \beta^{(\{j\})}_j\right)^2}}_{w_1(j)}}^{\mathsf{J}_1(\beta)},
\end{align}
where the penalty term measures the discrepancy across different environments. 

Here we use $\Sigma_{j,j}^{(e)}|\beta^{(e,\{j\})}_j - \beta^{(\{j\})}_j|^2$ rather than $|\beta^{(e,\{j\})}_j - \beta^{(\{j\})}_j|^2$ since the former is $x$-scale invariant and has a better explanation in prediction. To be specific, the term $w_1(j)$ will be the same if we replace $X$ by $aX$ for any $a\in \mathbb{R}\setminus \{0\}$. More importantly, it can be explained as the variation of optimal prediction in $L_2$ norm across environments, namely,
\begin{align}
\label{eq:pv1}
	w_1(j) = \frac{1}{|\mathcal{E}|} \sum_{e\in \mathcal{E}} \int \left\{f^{(e,j)}(x) - f^{(j)}(x) \right\}^2\mu^{(e)}(dx)
\end{align} where $f^{(e,j)}(x) = \beta^{(e,\{j\})}_j x_j$ is the best linear prediction on $X_j$ in environment $e$ and $f^{(j)}(x)=\beta^{(\{j\})}_j x_j$ is the best linear prediction on $X_j$ across all environments. 

The proposed optimization program can be understood in two aspects. On the one hand, it maintains the capability to solve the invariant pursuit problem, that is, recover $\beta^\star$ from \eqref{model:lip}, when $\gamma$ is large enough. To see this, when $\gamma \asymp 1$, the introduced penalty $\gamma \mathsf{J}_1(\beta)$ will place a constant penalty on the spurious variables, i.e., $j\in G$, and will not penalize any variables in $S^\star$. Therefore, one can expect that $\beta^\star$ will be the unique minimizer of $\mathsf{Q}_{1,\gamma}(\beta)$ as $\gamma$ is large enough so that the penalty term is larger than the prediction error of using $\beta^\star$. On the other hand, it maximizes relaxed worst-case explained variance over small perturbations around the pooled least squares, defined as $\bar{\beta} := \Sigma^{-1} u$, when $\gamma$ is small. Recall the definition of pooled quantity $(\Sigma, u)$ in \eqref{eq:covariance}, the two-fold characterization of the population-level minimizer of \eqref{eq:obj-k1} can be formally delivered as follows. 

\begin{proposition}\label{prop:minimax-k1}
    Let $\mathcal{P}_\gamma(\Sigma, u) = \big\{(X,Y)\sim \mu: \mathbb{E}[XX^\top] = \Sigma, \left|\mathbb{E}[X Y] - u\right| \le \gamma\cdot \allowbreak (w_1(1),\ldots, w_1(d))^\top \big\}$ be the uncertainty set of distributions. Under \cref{cond:normalize}, $\mathsf{Q}_{1,\gamma} (\beta)$ has an unique minimizer $\beta^{\gamma}$ satisfying
    \begin{align}
    \begin{split}
	\label{eq:minimax-k1}
    		\beta^\gamma &= \argmin_{\beta} \max_{\mu \in \mathcal{P}_\gamma(\Sigma, u)} \left\{\mathbb{E}_{(X,Y)\sim \mu}\left[|Y - \beta^\top X|^2 - |Y|^2\right] \right\}.
    \end{split}
    \end{align} Moreover, under \eqref{model:lip} with $S^\star$ further satisfying \eqref{ident:lip}, if \cref{cond:ortho} holds, then $\beta^\gamma = \beta^\star$ when $\gamma \ge \gamma^\star:= \max_{j\in G} |\frac{1}{|\mathcal{E}|} \sum_{e\in \mathcal{E}} \mathbb{E}[X_j^{(e)} \varepsilon^{(e)}]| / w_1(j)$, where $w_1(j)$ is defined in \eqref{eq:obj-k1}. 
\end{proposition}

\cref{prop:minimax-k1} offers interpretations of the population-level minimizer $\beta^\gamma$ of $\mathsf{Q}_{1,\gamma}(\beta)$ for varying $\gamma$ from two perspectives. On the one hand, $\beta^\gamma$ can be interpreted as the distributionally robust prediction model over the uncertainty set $\mathcal{P}_\gamma(\Sigma, u)$: it minimizes the worst-case negative explained variance, or it is the maximin effects \citep{meinshausen2015maximin,guo2024statistical} over the uncertainty set $\mathcal{P}_\gamma(\Sigma, u)$. The uncertainty class contains all joint distributions of $(X,y)$, where the covariates $X$ have the second-order moment matrix as $\Sigma$ and the covariance between $X$ and $Y$ is perturbed around $u$. Similar to Theorem 1 in \cite{meinshausen2015maximin} and Proposition 1 in \cite{guo2024statistical}, $\beta^\gamma$ has the following geometric explanation, that
\begin{align}
\label{eq:geo}
    \beta^\gamma = \argmin_{\beta \in \Theta_\gamma} \beta^\top \Sigma \beta \qquad \text{with} \qquad \Theta_\gamma = \{\beta: |\Sigma \beta - u| \le \gamma \cdot (w_1(1),\ldots, w_1(d))^\top\}.
\end{align} This basically says that $\beta^\gamma$ is the projection of the null $\beta=0$ on the convex closed set $\Theta_\gamma$ with respect to the norm $\|\cdot\| = \|\Sigma^{1/2} \cdot \|_2$; see the proof in \aosversion{Appendix D.2}{\cref{proof:geo}}. The distributional robustness \eqref{eq:minimax-k1} and geometric interpretation \eqref{eq:geo} are independent of the invariance structure \eqref{model:lip} and further structural assumption \cref{cond:ortho}. Instead, they are attributed to the choice of $L_1$ regularization with inhomogeneous weights $(w_1(1),\ldots, w_1(d))$. This is a realization of the heuristic idea of adopting an anisotropic uncertainty ellipsoid based on the observed environments. Specifically, more uncertainty is placed on the variables predicting differently in the observed environments than those with invariant predictions. 

On the other hand, consider the case where the data generating process satisfies the invariance structure \eqref{model:lip}, the sufficient heterogeneity \eqref{ident:lip}, together with an additional structure assumption \cref{cond:ortho}. Now the above distributionally robust procedure will place zero uncertainty on the invariant, causal variables, and will place linear-in-$\gamma$ uncertainty on the spurious variables. The minimizer $\beta^\gamma$ will coincide with the true, causal parameter $\beta^\star$ when $\gamma$ is large enough. 

Let us illustrate the above ideas using the toy example below. 

\begin{figure}
\begin{center}
\begin{tabular}{ccc}
    \subfigure[]{\begin{tikzpicture}[state/.style={circle, draw, minimum size=0.8cm, scale=0.8}]
        \def\yos{3.5}

        \draw[black, rounded corners] (-1.5, -1.8) rectangle (1.5, 1.5);
	\node[state] at (0, 1) (x1) {\small $X_1$};
	\node[state] at (0, 0) (y) {\small $Y$};
	\node[state] at (-1, -1) (x2) {\small $X_2$};
	\node[state] at (1, -1) (x3) {\small $X_3$};	
	\draw[->] (x1) -- node[right] {\scriptsize $1$} (y);
	\draw[->] (y) -- node[left] {\scriptsize $2/3$} (x2);
	\draw[->] (y) -- node[right] {\scriptsize $2/3$} (x3);
	\node at (0, -1.5) {\small e=1};
	
	\draw[black, rounded corners, fill=mylightblue!10] (-1.5, -1.8-\yos) rectangle (1.5, 1.5-\yos);
	\node[state] at (0, 1-\yos) (2x1) {\small $X_1$};
	\node[state] at (0, -\yos) (2y) {\small $Y$};
	\node[state] at (-1, -1-\yos) (2x2) {\small $X_2$};
	\node[state] at (1, -1-\yos) (2x3) {\small $X_3$};	
	\draw[->] (2x1) -- node[right] {\scriptsize $1$} (2y);
	\draw[->] (2y) -- node[left] {\myred{\scriptsize $0.5$}} (2x2);
	\draw[->] (2y) -- node[right] {\myred{\scriptsize $0.25$}} (2x3);
	\node at (0, -1.5-\yos) {\small e=2};
\end{tikzpicture}} & 
    \subfigure[]{\begin{tikzpicture}[x={(-0.77*0.6cm, -0.77*0.6cm)}, y={(2*0.6cm, 0cm)}, z={(0cm, 2*0.6cm)}]
	\def\gmma{0.4}
	\def\gmmb{2}
	\def\gmmc{3.6}
	\coordinate (beta) at (0.4280906440550959, 0.7498317150979645, 0.29347093477849623);
	\coordinate (betaa) at (0.5269925335391648, 0.7444419688895091, 0.08454497708581467);
	\coordinate (betab) at (0.7789430129056212, 0.3789517130027666, 3.585425932494615e-08);
	\coordinate (betastar) at (1, 0, 0);
	\coordinate (betat) at (0.5682481274386867, 0.7401557487840454, -2.0980577927396735e-08);
	\coordinate (d1) at (-0.2472514911051864, 0.01346193576399807, 0.5223244143715526);
	\coordinate (d2) at (-0.002692387080361211, 0.5990576227909074, -0.7565626719629004);
	\draw[gray, dashed] (beta) -- (0, 0.7498317150979645, 0.29347093477849623);
	\draw[gray, dashed] (0, 0.7498317150979645, 0.29347093477849623) -- (0, 0, 0.29347093477849623);
	\draw[gray, dashed] (0, 0.7498317150979645, 0.29347093477849623) -- (0, 0.7498317150979645, 0);
	\draw[gray, dashed] (betaa) -- (0, 0.7444419688895091, 0.08454497708581467);
	\draw[gray, dashed] (0, 0.7444419688895091, 0.08454497708581467) -- (0, 0, 0.08454497708581467);
	\draw[gray, dashed] (0, 0.7444419688895091, 0.08454497708581467) -- (0, 0.7444419688895091, 0);
	\draw[gray, dashed] (betab) -- (0.7789430129056212, 0, 0);
	\draw[gray, dashed] (betab) -- (0, 0.3789517130027666, 0);

	\coordinate (p3x1) at ($(beta) - \gmmc*(d1)$);
	\coordinate (p3x2) at ($(beta) - \gmmc*(d2)$);
	\coordinate (p3x3) at ($(beta) + \gmmc*(d1)$);
	\coordinate (p3x4) at ($(beta) + \gmmc*(d2)$);
	\coordinate (p2x1) at ($(beta) - \gmmb*(d1)$);
	\coordinate (p2x2) at ($(beta) - \gmmb*(d2)$);
	\coordinate (p2x3) at ($(beta) + \gmmb*(d1)$);
	\coordinate (p2x4) at ($(beta) + \gmmb*(d2)$);
	\coordinate (p1x1) at ($(beta) - \gmma*(d1)$);
	\coordinate (p1x2) at ($(beta) - \gmma*(d2)$);
	\coordinate (p1x3) at ($(beta) + \gmma*(d1)$);
	\coordinate (p1x4) at ($(beta) + \gmma*(d2)$);
	\draw[colorgradcurve={0.4pt}{myred}{myblue}, thick] (beta) -- (betat) -- (betastar);
	\path[draw=myred!0!myblue, fill=myred!0!myblue, thick, opacity = 0.2] (p3x1) -- (p3x2) -- (p3x3) -- (p3x4) -- (p3x1);

	\path[draw=myred!88!myblue, fill=myred!88!myblue, thick, opacity = 0.2] (p1x1) -- (p1x2) -- (p1x3) -- (p1x4) -- (p1x1);
	\path[draw=myred!43!myblue, fill=myred!43!myblue, thick, opacity = 0.2] (p2x1) -- (p2x2) -- (p2x3) -- (p2x4) -- (p2x1);
	\draw[draw=myred!88!myblue, fill=myred!88!myblue] (betaa) circle (1pt) node[below] {\textcolor{myred!88!myblue}{\footnotesize $\beta^{0.4}$}};
	\draw[draw=myred!43!myblue, fill=myred!43!myblue] (betab) circle (1pt) node[below] {\textcolor{myred!43!myblue}{\footnotesize $\beta^{2}$}};

	
	\draw[draw=myred, fill=myred] (beta) circle (1pt) node[right] {\myred{{\footnotesize$\beta^{0}=$~}$\bar{\beta}$}};
	\draw[draw=myblue, fill=myblue] (betastar) circle (1pt) node[left] {\myblue{$\beta^\star${\footnotesize$=\beta^{3.6}$~}}};

	\draw (p3x2) node {\footnotesize \myblue{$\Theta_{3.6}$}};
	\draw (p2x2) node {\footnotesize \textcolor{myred!42!myblue}{$\Theta_2$}};
	\draw (p1x2) node {\footnotesize \textcolor{myred!71!myblue}{$\Theta_{0.4}$}};

	\draw[-latex] (0,0,0) -- (2,0,0) node[pos = 1.2] {$\beta_1$};
	\draw[-latex] (0,-0.3,0) -- (0,2,0) node[pos = 1.1] {$\beta_2$};
	\draw[-latex] (0,0,-0.3) -- (0,0,2) node[pos = 1.1] {$\beta_3$};
\end{tikzpicture}} & 
    \subfigure[]{\begin{tikzpicture}[x={(0.6cm, 0cm)}, y={(0cm, 1.8cm)}]
	\def\yscale{1}
	\def\xscale{1}
	\draw[->] (0, 0) -- (0, 1.3);
	\draw[->] (0, 0) -- (6, 0);
	
	\draw node at (6.2, 0) {\footnotesize $\gamma$};
	\draw node at (0, 1.45) {\footnotesize $\beta_j^\gamma$};

	\draw[draw=mygreen, line width=1pt] (0, 0.4280906440550959) -- (0.57, 0.5682481274386867) -- (3.51, 1) -- (6, 1);
	\node at (-0.3,0.4280906440550959) {\footnotesize \textcolor{mygreen}{$\beta_1$}};
	\draw[draw=myyellow, dashed, line width=1pt] (0, 0.7624616067073176) -- (0.57, 0.7401557487840454) -- (3.51, 0) -- (6, 0);
	\node at (-.3,\yscale*0.76246161) {\footnotesize \textcolor{myyellow}{$\beta_2$}};
	\draw[draw=myyellow, dotted, line width=1pt] (0, 0.29347093477849623) -- (0.57, 0) -- (6, 0);
	\node at (-.3,\yscale*0.29347093477849623) {\footnotesize \textcolor{myyellow}{$\beta_3$}};

        \draw[draw=myred!88!myblue, dashed] (0.4, 0) node[below] {\scriptsize \textcolor{myred!88!myblue}{$0.4$}} -- (0.4, 1.2);
        \draw[draw=myred!43!myblue, dashed] (2, 0) node[below] {\scriptsize \textcolor{myred!43!myblue}{$2$}} -- (2, 1.2);
        \draw[draw=myblue, dashed] (3.6, 0) node[below] {\scriptsize \textcolor{myblue}{$3.6$}} -- (3.6, 1.2);

	\def\yoffset{2}
	\draw[->] (0, 0-\yoffset) -- (0, 1.3-\yoffset);
	\draw[->] (0, -\yoffset) -- (6, -\yoffset);
	
	\draw node at (6.2, 0-\yoffset) {\footnotesize $\sqrt{\gamma}$};
	\draw node at (0, 1.45-\yoffset) {\footnotesize $\beta_{\mathtt{FAIR},j}^\gamma$};

	\draw[draw=mygreen, line width=1pt] (0, 0.4280906440550959-\yoffset) -- (1.82, 0.4280906440550959-\yoffset);
	\draw[draw=mygreen, line width=1pt] (1.82, 0.4280906440550959-\yoffset) circle (1pt);
	\draw[draw=mygreen, line width=1pt] (1.83, 0-\yoffset) circle (1pt);
	\draw[draw=mygreen, line width=1pt] (1.83, 0-\yoffset) -- (4.87, 0-\yoffset);
	\draw[draw=mygreen, line width=1pt] (4.87, 0-\yoffset) circle (1pt);
	\draw[draw=mygreen, line width=1pt] (4.88, 1-\yoffset) circle (1pt);
	\draw[draw=mygreen, line width=1pt] (4.88, 1-\yoffset) -- (6, 1-\yoffset);
	\node at (-0.3,0.4280906440550959-\yoffset) {\footnotesize \textcolor{mygreen}{$\beta_1$}};

	\draw[draw=myyellow, dashed, line width=1pt] (0, 0.7498317150979645-\yoffset) -- (1.82, 0.7498317150979645-\yoffset);
	\draw[draw=myyellow, line width=1pt] (1.82, 0.7498317150979645-\yoffset) circle (1pt);
	\draw[draw=myyellow, line width=1pt] (1.83, 0.95403938-\yoffset) circle (1pt);
	\draw[draw=myyellow, dashed, line width=1pt] (1.83, 0.95403938-\yoffset) -- (4.87, 0.95403938-\yoffset);
	\draw[draw=myyellow, line width=1pt] (4.87, 0.95403938-\yoffset) circle (1pt);
	\draw[draw=myyellow, line width=1pt] (4.88, 0-\yoffset) circle (1pt);
	\draw[draw=myyellow, dashed, line width=1pt] (4.88, 0-\yoffset) -- (6, 0-\yoffset);
	\node at (-0.3,0.76246161-\yoffset) {\footnotesize \textcolor{myyellow}{$\beta_2$}};

	\draw[draw=myyellow, dotted, line width=1pt] (0, 0.29347093477849623-\yoffset) -- (1.82, 0.29347093477849623-\yoffset);
	\draw[draw=myyellow, line width=1pt] (1.82, 0.29347093477849623-\yoffset) circle (1pt);
	\draw[draw=myyellow, line width=1pt] (1.83, 0.37339422-\yoffset) circle (1pt);
	\draw[draw=myyellow, dotted, line width=1pt] (1.83, 0.37339422-\yoffset) -- (4.87, 0.37339422-\yoffset);
	\draw[draw=myyellow, line width=1pt] (4.87, 0.37339422-\yoffset) circle (1pt);
	\draw[draw=myyellow, line width=1pt] (4.88, 0-\yoffset) circle (1pt);
	\draw[draw=myyellow, dotted, line width=1pt] (4.88, 0-\yoffset) -- (6, 0-\yoffset);
	\node at (-0.3,\yscale*0.29347093477849623-\yoffset) {\footnotesize \textcolor{myyellow}{$\beta_3$}};

\end{tikzpicture}} 
\end{tabular}
\end{center}
\caption{(a) A structural causal model illustration of the multi-environment model in \cref{ex2}: the arrow from node $u$ to node $v$ with number $s$ means there is a linear causal effect $s$ of $u$ on $v$. (b) visualize the uncertainty set $\Theta_\gamma$ in three checkpoints of $\gamma\in \{0.4, 2, 3.6\}$ and regularization path of the proposed estimator \eqref{eq:minimax-k1} in the three-dimensional parameter space $\beta \in \mathbb{R}^3$. For each $\gamma$, the uncertainty set $\Theta_\gamma$ is a two-dimensional plane filled by colors changing from \myred{red} to \myblue{blue} as $\gamma$ increases. The upper panel of (c) depicts how the population level solution $\beta^\gamma \in \mathbb{R}^3$ changes according to $\gamma$ in each coordinate $j\in [3]$: the causal variable is represented by \mygreen{green} solid line, and the two spurious (reverse causal) variable are represented by \myyellow{yellow} dashed ($\beta_2$) and dotted ($\beta_3$) lines, respectively. The lower panel of (c) plots the counterpart for the FAIR-Linear estimator in \cite{gu2024causality}. }
\label{fig:reg}
\end{figure}
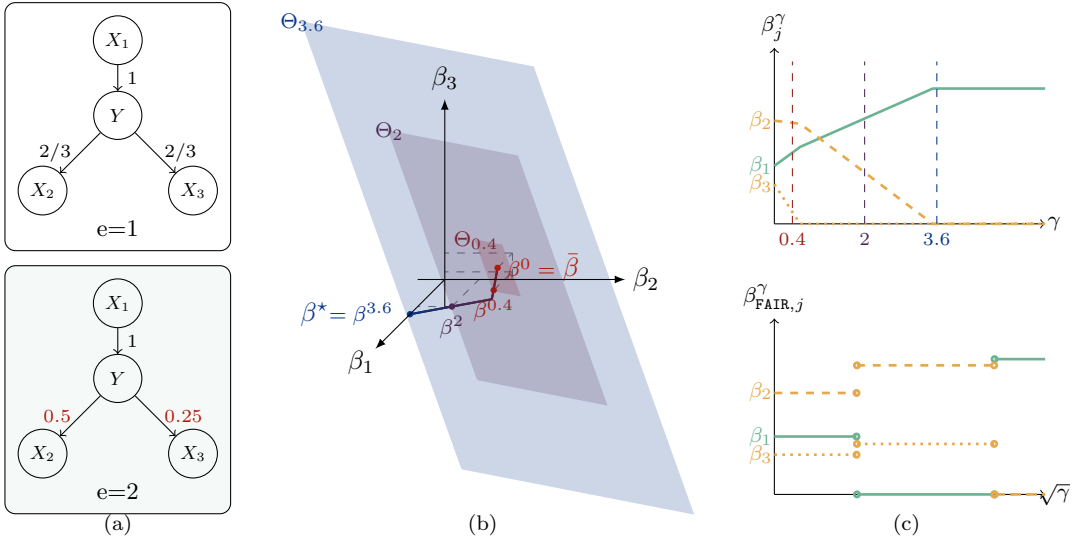

\begin{example} \label{ex2}
Consider the following data-generating process with $d=3$, $\mathcal{E}=\{1,2\}$ and independent standard normal random variables $\varepsilon_0,\ldots, \varepsilon_3$, the cause-effect relationship and the intervention effects are illustrated in \cref{fig:reg} (a). The constant factors before $\varepsilon_j$ with $j\ge 2$ are added to ensure $X_j$ has a unit variance. 
\begin{align*}
\begin{split}
    X_1^{(1)} &\gets \varepsilon_1, \\
    Y^{(1)} &\gets X_1^{(1)} + \varepsilon_0, \\
    X_2^{(1)} &\gets (2/3) \cdot Y^{(1)} + (1/3) \cdot\varepsilon_2, \\
    X_3^{(1)} &\gets (2/3) \cdot Y^{(1)} + (1/3) \cdot\varepsilon_3; \\
\end{split} \qquad \text{and} 
\begin{split}
    X_1^{(2)} &\gets \varepsilon_1, \\
    Y^{(2)} &\gets X_1^{(2)} + \varepsilon_0, \\
    X_2^{(2)} &\gets 0.5 \cdot Y^{(2)} + (\sqrt{2})^{-1} \cdot \varepsilon_2, \\
    X_3^{(2)} &\gets 0.25 \cdot Y^{(2)} + \sqrt{7/8} \cdot \varepsilon_3. \\
\end{split}
\end{align*}
\end{example}

In \cref{ex2}, $X_1$ is the invariant (causal) variable, while $X_2$ and $X_3$ are all endogenous spurious (reverse causal) variables as shown in \cref{fig:reg} (a). They have identical spurious predictive powers in environment $e=1$, and variable $X_3$ is confronted with stronger perturbations than $X_2$ in environment $e=2$. The invariance structure is well identified with $S^\star=\{1\}$ satisfying \eqref{model:lip} and \eqref{ident:lip} simultaneously. The prediction variation in \eqref{eq:pv1} are $(w_1(1), w_1(2), w_1(3)) = (0, 1/6, 1/4)$. 

\cref{fig:reg} (b) visualize the maximin effect \eqref{eq:minimax-k1} over the uncertainty set shaped by the prediction variation. For given fixed $\gamma$, the uncertainty set in $\mathbb{E}[XY]$ in \eqref{eq:minimax-k1} does not place uncertainty on the causal variable $X_1$, while it places a relatively small uncertainty $\gamma/6$ on the variables $X_2$ which suffers from less perturbation, and a relatively large uncertainty $\gamma/4$ on the variable $X_3$ that predicts more differently in observed environments $\mathcal{E}$. This two-dimensional uncertainty plane in covariance space further yields the two-dimensional uncertainty plane centered on the pooled least squares $\bar{\beta}$ in the solution space after the affine transformation $x\to \Sigma^{-1} x$ as shown in \cref{fig:reg} (b). The uncertainty sets $\Theta_\gamma$ all lie in the same hyper-plane and their diameter scales linearly with $\gamma$. The corresponding population-level minimizer $\beta^\gamma$ is the projection of the null $\beta=0$ on $\Theta_\gamma$. This leads to a solution path that connects the most predictive solution $\bar{\beta}$ and the causal solution $\beta^\star$ continuously. When $\gamma$ is smaller than the critical threshold, such a prediction $\beta^\gamma$ still leverages part of the spurious variables for prediction and will have better prediction over $\beta^\star$ and $\bar{\beta}$ when it is deployed in an environment where the reverse causal effects are still positive but slightly shrinkage, for example, $X_3 \gets Y/3\sqrt{2}+ \sqrt{8/9} \varepsilon_3$. Such a solution $\beta^\gamma$ stands in between $\beta^\star$ and $\bar{\beta}$: it is more robust than $\bar{\beta}$ and less conservative than $\beta^\star$. As a comparison, the FAIR-Linear \citep{gu2024causality} estimator that solves the hard-constrained structural estimation problem is less flexible in this regard, as shown in the lower panel of \cref{fig:reg} (c), it adopts certain hard threshold and choose either to include or eliminate the spurious variables. 

\subsection{Interpolating between the Orthogonal and General Cases}
\label{sec:population-est}

The population-level minimizer of \eqref{eq:obj-k1} can solve the linear invariance pursuit in \eqref{model:lip} efficiently within time complexity $O((|\mathcal{E}|+n)d + T_{\mathrm{Lasso}}(|\mathcal{E}|\cdot n, d))$, where $T_{\mathrm{Lasso}}(N, d)$ is the complexity of running a $d$-variate $N$-sample Lasso. However, the estimation can only be guaranteed when \cref{cond:ortho} holds, and it may fail when \cref{cond:ortho} does not hold. Here, we introduce a more general relaxation balancing estimation error and time complexity. 

Instead of calculating the prediction variation of the marginal linear predictor for each variable $X_j$, we consider calculating the prediction variation of the predictors using variable size less or equal to $k$. For the population-level counterpart, it minimizes the following objective
\begin{align}
\label{eq:obj-k}
	\mathsf{Q}_{k,\gamma}(\beta) = \mathsf{R}^{\mathcal{E}}(\beta) + \gamma \overbrace{\sum_{j=1}^d |\beta_j| \cdot \underbrace{\sqrt{\min_{S: j\in S, |S| \le k}  \frac{1}{|\mathcal{E}|} \sum_{e\in \mathcal{E}} \|\beta^{(e,S)}_S - \beta^{(S)}_S\|_{\Sigma^{(e)}_S}^2	}}_{w_k(j)}}^{\mathsf{J}_k(\beta)}
\end{align} with some computational budget hyper-parameter $k \in \mathbb{N}$.

As $k$ grows or equivalently as more computational budget is paid, the space of instances that can be solved enlarges and will finally coincide with that of EILLS or FAIR when $k\ge |S^\star|$. On the other hand, if the computational budget we can pay is relatively limited, one can still probably solve some problem instances with low-dimensional structures as elaborated in the following \cref{thm:causal-ident}.

\begin{condition}[Restricted Invariance]\label{cond:restricted-invariance}    For any $j\in S^\star$, there exists some $S \subseteq [d]$ with $|S|\le k$ and $j\in S$ such that $\beta^{(e,S)} \equiv \beta^{(S)}$ for any $e\in \mathcal{E}$.  
\end{condition}

Note that when \cref{cond:restricted-invariance} holds, for all $j \in S^\star$, the weight $w_k(j)$ in the penalty term is equal to $0$.  On the other hand, for a large enough $\gamma$, all endogenous variables will be excluded due to a positive $w_{k}(j)$. Hence, the object \eqref{eq:obj-k} will screen out all endogenously spurious variables and meanwhile minimize the prediction errors using the remaining variables.  
\cref{cond:restricted-invariance} naturally holds when $k \ge |S^\star|$. When $k < |S^\star|$, \cref{cond:restricted-invariance} requires a stronger identification condition than the invariance assumption \eqref{model:lip} such that all the invariant variables $X_j$ with $j\in S^\star$ can be identified using a smaller set $S_j$ with $|S_j| \le k < |S^\star|$. This is a generic condition and can hold under different circumstances. For example, there are some shared group-orthogonal structures in the set $S^\star$ such as $\Sigma_{S^*}^{(e)}$ admits a block diagonal structure with the maximum block size $\leq k$, which includes the diagonal case in \cref{cond:ortho} as a specific instance, or the insufficiency of interventions on the ancestors of $S^\star$, for example, all the ancestors of $Y$ are free of intervention.   \aosversion{Proposition B.2}{\cref{prop:ri}} in the appendix further offers conditions under which \cref{cond:restricted-invariance} holds. The following two theorems generalize \cref{prop:minimax-k1} for growing $k$. 

\begin{theorem}\label{thm:minimax-k}
        Let $\mathcal{P}_{\gamma,k}(\Sigma, u) = \big\{(X,Y)\sim \mu: \mathbb{E}[XX^\top] = \Sigma, \left|\mathbb{E}[X_jY] - u_j\right| \le \gamma \cdot w_k(j) ~\forall j\in [d]\big\}$ be the uncertain set of distributions. 
        Under \cref{cond:normalize}, $\mathsf{Q}_{k,\gamma} (\beta)$ has a unique minimizer $\beta^{k,\gamma}$ satisfying
	\begin{align}\label{eq:minimax-k}
            \beta^{k,\gamma} &= \argmin_{\beta} \max_{\mu \in \mathcal{P}_{\gamma,k}(\Sigma, u)} \left\{\mathbb{E}_{(X,Y)\sim \mu}[|Y - \beta^\top X|^2 - |Y|^2] \right\}
	\end{align} 
\end{theorem}

\begin{theorem}
\label{thm:causal-ident}
Under the setting of \cref{thm:minimax-k}, assume the invariance structure \eqref{model:lip} holds with $S^\star$ satisfying \eqref{ident:lip}. Suppose further that \cref{cond:restricted-invariance} holds, then $\beta^{k,\gamma} = \beta^\star$ when $\gamma \ge \gamma^\star_k$ with $\gamma_k^\star:= \max_{j\in G} |\frac{1}{|\mathcal{E}|} \sum_{e\in \mathcal{E}} \mathbb{E}[X_j^{(e)} \varepsilon^{(e)}] | / w_k(j)$.
\end{theorem}

\begin{remark}
One can show that $\gamma^\star_k$ is uniformly upper bounded by
    \begin{align*}
    \gamma^\star_k \le \left\{\min_{e\in \mathcal{E}}\lambda_{\min}(\Sigma^{(e)}) \right\} \cdot \sqrt{\gamma^*}
    \end{align*} where $\gamma^*$ is the critical threshold, or the signal-to-noise ratio in heterogeneity in \cite{fan2023environment}.  It was defined on a square scale, so a square root is taken here; see the formal definition of $\gamma^*$ in \aosversion{(D.2)}{\eqref{eq:gamma-star}} in the appendix.  This indicates that one does not need to adopt a potentially larger hyper-parameter to achieve causal identification compared with EILLS in \cite{fan2023environment}, recalling the scaling \cref{cond:normalize}. 
\end{remark}
Similar to \cref{prop:minimax-k1}, the first distributional robustness interpretation \eqref{eq:minimax-k} in \cref{thm:minimax-k} is due to adopting inhomogeneous $L_1$ penalization on the variables based on a finer prediction variation $(w_k(1), \cdots, w_k(d))$ observed in the environments $\mathcal{E}$ than the marginal counterpart $(w_1(1), \cdots, w_1(d))$. The second theorem \cref{thm:causal-ident} states that when additional structural assumption \eqref{cond:restricted-invariance} holds, the causal parameter $\beta^\star$ under \eqref{model:lip} with \eqref{ident:lip} can be identified by our estimator when $\gamma$ is large enough.

\subsection{Empirical-level Estimator and Non-asymptotic Analysis}
\label{sec:empirical-est}

Turning to the empirical counterpart, for given $k$ and $\gamma$, we consider minimizing the following empirical-level penalized least squares
\begin{align}
\label{eq:obj-empirical-k}
\begin{split}
    \hat{\beta}^{k, \gamma} &= \argmin_{\beta} \overbrace{\frac{1}{2n|\mathcal{E}|} \sum_{e\in \mathcal{E}, i\in [n]} \left(Y_i^{(e)} - \beta^\top X_i^{(e)} \right)^2 + \gamma \cdot \sum_{j=1}^d |\beta_j| \sqrt{\hat{w}_k(j)}}^{\hat{\mathsf{Q}}_{k,\gamma}(\beta)}, \\
    & ~~~~~~~~~~~~~~~~~~~~~~~~ \text{with} ~~~ \hat{w}_k(j) = \inf_{S \subseteq [d], |S|\le k, j\in S} \frac{1}{|\mathcal{E}|} \sum_{e\in \mathcal{E}} \left\|\hat{\beta}_S^{(e,S)} - \hat{\beta}^{(S)}_S\right\|_{\hat{\Sigma}^{(e)}_{S}}^2.
\end{split}
\end{align}
The weighted $L_1$-penalty aims at attenuating the endogenously spurious variables.   This will be applied to the low-dimensional regime $d=o(n)$.
Under the high-dimensional regime $d \gtrsim n$, we further add another $L_1$ penalization with hyper-parameter $\lambda$, which aims at reducing exogenously spurious variables:
\begin{align}
\label{eq:l1-estimator}
 \hat{\beta}^{k, \gamma, \lambda} = \argmin_{\beta\in \mathbb{R}^d}  \hat{\mathsf{Q}}_{k,\gamma}(\beta) + \lambda \|\beta\|_1.
\end{align}

For the theoretical analysis, we impose some standard assumptions used in linear regression. 
\begin{condition}[Regularity]
\label{cond:regularity}
The following conditions hold:
\begin{itemize}
\item[(a)] \underline{(Data Generating Process)} We collect data from $|\mathcal{E}| \in \mathbb{N}^+$ environments. For each environment $e\in \mathcal{E}$, we observe $(X_1^{(e)}, Y_1^{(e)}), \ldots, (X_n^{(e)}, Y_n^{(e)}) \overset{i.i.d.}{\sim} \mu^{(e)}$. The data from different environments are also independent.
\item[(b)] \underline{(Non-collinearity and Normalization)} Assume $\Sigma^{(e)} \succ 0$ for any $e\in \mathcal{E}$. Recall the definition in \eqref{eq:covariance}, we have $\Sigma_{j,j}=1$ for any $j\in [d]$.
\item[(c)] \underline{(Sub-Gaussian Covariate and Noise)} There exists some constants $\sigma_x \in [1,\infty)$ and $\sigma_y \in \mathbb{R}^+$ such that
\begin{align*}
	\forall e\in \mathcal{E} ~~~~~~~~~~ &\mathbb{E} \left[\exp\left\{v^\top (\Sigma^{(e)}_S)^{-1/2}X^{(e)}_S\right\}\right] \le \exp\left(\frac{\sigma_x^2}{2} \cdot \|v\|_2^2\right) ~~ \forall S\subseteq [d], v \in \mathbb{R}^{|S|},\\
                                            & \mathbb{E} \left[\exp\left\{\lambda Y^{(e)}\right\}\right] \le \exp\left(\frac{\lambda^2 \sigma_y^2}{2}\right) ~~ \forall \lambda 
                                            \in \mathbb{R}.
\end{align*}
\item[(d)] \underline{(Relative Bounded Covariance)} There exists a constant $b \in [1,\infty)$ such that
\begin{align*}
\forall e\in \mathcal{E} \text{ and } S\subseteq [p] \qquad \lambda_{\max}(\Sigma_S^{-1/2} \Sigma_S^{(e)} \Sigma_S^{-1/2}) \le b.
\end{align*}
\end{itemize}
To simplify the presentation, let $c_1$ be such that $c_1 \ge \max\{b, \sigma_x, \sigma_y\}$ and $|\mathcal{E}| \le n^{c_1}$. 
\end{condition}

These assumptions are standard in the analysis of linear regression. It is easy to see the sub-Gaussian covariate conditions hold with $\sigma_x=1$ when $X^{(e)}\sim \mathcal{N}(0,\Sigma^{(e)})$.  The sub-Gaussian condition can be relaxed by the finite fourth-moment conditions with robust inputs; see \cite{fan2020robust}. Our error bound is independent of $\sup_{e\in \mathcal{E}} \lambda_{\max}(\Sigma^{(e)})$ given fixed $b$. The maximum eigenvalue $\lambda_{\max}(\Sigma^{(e)})$ may grow with $d$ in the presence of highly correlated covariates such as factor models \citep{fan2022latent, fan2024factor}. It is also easy to see that $b\le |\mathcal{E}|$ by observing that
\begin{align}
\label{eq:b<E}
    \lambda_{\max}(\Sigma_S^{-1/2} \Sigma_S^{(e)} \Sigma_S^{-1/2}) \le \left\{\lambda_{\min}\left((\Sigma^{(e)}_S)^{-1/2} \Sigma_S (\Sigma_S^{(e)})^{-1/2} \right)\right\}^{-1} \le |\mathcal{E}|.
\end{align} 

The following theorem establishes the $L_2$ error bound with respect to $\beta^{k,\gamma}$ identified in \cref{thm:minimax-k} in the low-dimensional regime.

\begin{theorem}
\label{thm:main-lowdim}
Assume \cref{cond:regularity} holds. There exists a constant $\tilde{C}=O(\mathrm{poly}(c_1))$ such that if $n\ge \tilde{C} \max\{d, k \log d, t\}$ and $t\ge \log n$, then with probability at least $1-e^{-t}$,
\begin{align*}
    \|\hat{\beta}^{k, \gamma} - \beta^{k, \gamma}\|_2 \le \tilde{C} \cdot \sqrt{\frac{d}{n}} \cdot \left\{\frac{\gamma}{\kappa} \sqrt{t + \log(n) + k\log d} + \sqrt{\frac{1+t/d}{\kappa\cdot|\mathcal{E}|}}\right\},
\end{align*}
where $\kappa=\min_{e\in \mathcal{E}} \lambda_{\min}(\Sigma^{(e)})$.
\end{theorem}

As shown in \cref{thm:minimax-k} and \cref{thm:causal-ident}, the invariance hyper-parameter $\gamma$ interpolates the most predictive solution, the pooled least squares $\bar{\beta} = \Sigma^{-1} u$, with $\gamma=0$ and the most robust solution, the invariant (causal) solution $\beta^\star$, with large enough $\gamma\ge \gamma_k^\star$ in a smooth manner when the additional condition \cref{cond:restricted-invariance} holds. Under the regime of $\kappa \asymp \gamma_k^\star \asymp 1$, our proposed empirical estimator converges to the target $\bar{\beta}$ at the rate of $\{d/(n \cdot |\mathcal{E}|)\}^{1/2}$ on one hand $\gamma = 0$. On the other hand, combining it with \cref{thm:causal-ident}, we also have the convergence rate to the causal parameter $\beta^\star$, that is,  
\begin{align*}
    \mathbb{P}\left[\|\hat{\beta}^{k,\gamma} - \beta^\star\|_2 \le C \sqrt{\{\log(n)+k\log(d)\}\frac{d}{n}}\right] \ge 1-n^{-10}
\end{align*} with the proper choice of $\gamma \asymp \gamma_k^\star \asymp 1$. When $0<\gamma <  \gamma_k^\star$, the estimator $\hat{\beta}^{k,\gamma}$ serves as an invariance information guided distributionally robust estimator, whose variance of the empirical estimator lies in between the two. 

Turning to the high-dimensional regime, we have the following result. The main message is that the proposed estimator in \eqref{eq:l1-estimator} can handle the high-dimensional covariates in a similar spirit to Lasso \citep{tibshirani1997lasso, bickel2009simultaneous} for the sparse linear model with the help of another $L_1$ penalty. 

\begin{theorem}
\label{thm:high-dim}
Assume \cref{cond:regularity} holds. Denote $S^{k,\gamma} = \supp(\beta^{k,\gamma})$. There exists a constant $\tilde{C} = O(\mathrm{poly}(c_1))$ such that if $n\ge \tilde{C}(k+\kappa^{-1}|S^{k,\gamma}|)\log d$, then with probability at least $1-(nd)^{-10}$, 
\begin{align*}  
    \|\hat{\beta}^{k, \gamma, \lambda} - \beta^{k, \gamma}\|_2 \le \frac{12 \sqrt{|S^{k,\gamma}|}}{\kappa} \lambda ~~~ \text{if} ~~~ \lambda \ge \tilde{C}\left(\gamma\cdot\sqrt{\frac{k\log d+\log n}{n}} + \sqrt{\frac{\log d+\log n}{n\cdot|\mathcal{E}|}}\right).
\end{align*}
\end{theorem}

\section{Real Data Applications}
\label{sec:exp} 
In this section, we compare our method \emph{invariance-guided regularization (IGR)} with other estimators in two real data applications: daily stock log-return prediction and earth climate system prediction. Our proposed method attains more robust predictions compared with predecessors. We summarize the framework in \cref{alg:igr-algo}. In the two applications, we simply adopt $k=2$. The performance of varying $k\in \{1,2,3\}$ is similar; see \aosversion{Appendix F.4}{\cref{sec:different-k}}.

\begin{algorithm}
    \caption{Linear Regression with Invariance-Guided Regularization (IGR)}
    \begin{algorithmic}[1]
    \State \textbf{Input:} Training environments $\{\mathcal{D}^{(e)}\}_{e\in \mathcal{E}}$ with $\mathcal{D}^{(e)} = \{(X_i^{(e)}, Y_i^{(e)})\}_{i=1}^n$; validation environment $\mathcal{D}^{(\mathrm{valid})}$.
    \State \textbf{Input:} computational budget $k$.
    \State \textbf{Input:} candidate sets of hyper-parameters $\Gamma$ and $\Lambda$.
    \State For each pair of hyper-parameters $(\gamma,\lambda) \in \Gamma \times \Lambda$, calculate $\hat{\beta}^{k,\gamma,\lambda}$ using \eqref{eq:l1-estimator} on training environments.
    \State Choose hyper-parameters as
    \begin{align}
    \label{eq:hyper-parameter}
        \hat{\gamma}, \hat{\lambda} \in \argmin_{\gamma \in \Gamma, \lambda \in \Lambda} \frac{1}{|\mathcal{D}^{(\mathrm{valid})}|} \sum_{(X_i,Y_i)\in \mathcal{D}^{(\mathrm{valid})}} (X_i^\top \hat{\beta}^{k, \gamma, \lambda} - Y_i)^2
    \end{align}
    \State \textbf{Output: } $\beta^{k, \hat{\gamma}, \hat{\lambda}}$.
    \end{algorithmic}
    \label{alg:igr-algo}
    \end{algorithm}

\subsection{Stock Log-return Prediction}\label{sec:stock}

We follow \citet{varambally2023discovering} and use the daily log-returns of 100 stocks from S\&P 100, defined as the differences in the logarithms of the closing prices of successive days. We denote the daily log-returns of these stocks as $\{Z_{t,j}\}_{t\in[T],j\in[100]}$ where $T$ is the length of the sequence.
In this study, we focus on predicting the stocks in the Real Estate sector: American Tower (Symbol: \texttt{AMT}) and Simon Property Group (Symbol: \texttt{SPG}). For the task of predicting the outcome variables $\{\texttt{AMT}, \texttt{SPG}\}$ with index $j_0 \in [100]$, the target response variable is $Y_t=Z_{t, j_0}$, and the covariate is $X_t = \{Z_{t-1, j} : j \in [100]\} \cup \{Z_{t, j} : j \neq j_0\}$, the same as \citep{varambally2023discovering}.

We use data from 800 consecutive days starting in August 2018 and partition this time series into seven segments: days 1–100 ($\mathcal{D}_1$) and 101–200 ($\mathcal{D}_2$) serve as the two training environments, days 201–400 ($\mathcal{D}_3$) as the validation environment, and days 401–500, 501–600, 601–700, and 701–800 as the four test environments denoted as $\{\mathcal{D}_{3+i}\}_{i=1}^4$. This partitioning is motivated by the results of \citet{varambally2023discovering}, which indicate that the market behavior between the two training time spans differs significantly. We set both $X$ and $Y$ to be zero-mean in each environment to remove the effect of the trend. 

We fix the computational budget $k = 2$ and compare our method with Causal Dantzig \citep{rothenhausler2019causal}, Anchor Regression \citep{rothenhausler2021anchor} and DRIG \citep{shen2023causality} with the aid of $L_1$ penalty (if applicable), along with PCMCI$^+$ \citep{runge2020discovering} with the aid of $L_2$ penalty. The hyper-parameters for all models are determined via the validation set $\mathcal{D}^{(\mathrm{valid})}=\mathcal{D}_3$ using the criterion similar to \eqref{eq:hyper-parameter}. Here the hyper-parameters in the two prediction tasks are determined independently. We finally evaluate each method using the worst-case out-of-sample $R^2$ across the four test environments defined as 
\begin{align}
\label{eq:worst-exp}
    \min_{e\in\{4, 5, 6, 7\}} R^2_{\mathtt{oos}, e} ~~\text{with}~~ R_{\mathtt{oos},e}^2=1-\frac{\sum_{(X, Y)\in \mathcal{D}_e}(Y-\hat{Y}(X))^2}{\sum_{(X, Y)\in \mathcal{D}_e}Y^2}
\end{align} where $\hat{Y}(X)$ is the model's prediction. Here we use the $R^2$ rather than the mean squared error in \eqref{eq:hyper-parameter} to present the result to illustrate the challenge of this task, given most of the previous methods have negative out-of-sample $R^2$, indicating that their fitted models are even worse than simply using the null prediction model. 

This process is repeated $100$ times. For each trial, we use a random sample of $90$ data in each training environment $\mathcal{D}_1, \mathcal{D}_2$ to fit the model. The average $\pm$ standard deviation of the worst-case out-of-sample $R^2$ is reported in \cref{table:stock}. 

\begin{table}[htb!]
    \footnotesize
    \centering
    \begin{tabular}{crr}
      \hline
      \multicolumn{1}{c}{Data} &
      \multicolumn{1}{c}{\texttt{AMT}} & 
      \multicolumn{1}{c}{\texttt{SPG}}\\
      \hline
      IGR (Ours) & {$0.131\pm0.074$}& {$0.048\pm 0.039$}\\
      Causal Dantzig&$-0.150\pm0.296$&$-0.006\pm 0.072$\\
      Anchor/Lasso &$ -0.199\pm0.097$& $-0.018\pm 0.021$\\ 
      DRIG&$ -0.553\pm0.309$&$-0.201\pm 0.099$\\
      PCMCI$^+$&$ 0.051\pm0.075$&$ -0.057\pm0.041$\\
      \hline
    \end{tabular}
	\caption{The average $\pm$ standard deviation of the worst-case out-of-sample $R^2$ \eqref{eq:worst-exp} for predicting the stocks $\mathtt{AMT}$ and $\mathtt{SPG}$ using different estimators.}
    \label{table:stock}
\end{table}

We can see that our method outperforms competing methods in terms of robustness, as it provides more consistent estimations across different environments. In particular, our method achieves a positive worst-case out-of-sample $R^2$ when predicting \texttt{SPG}, while the other methods result in negative $R^2$ values. 
To qualitatively illustrate why most of the other competing methods yield negative $R^2$ values, we apply LASSO with an $L_1$ penalty parameter of $0.125$ on the training data in the $\texttt{AMT}$ task to select covariates. Using the selected covariates, we refit the target on the training environments $\mathcal{D}_1,\mathcal{D}_2$, as well as one of the test environments $\mathcal{D}_6$. As shown in \cref{fig:stable}, the resulting estimations differ drastically, highlighting strong heterogeneity across environments. This observation partially explains why other methods may produce negative $R^2$ values.

\begin{figure}[htb!]
    \centering
    \includegraphics[width=1.0\linewidth]{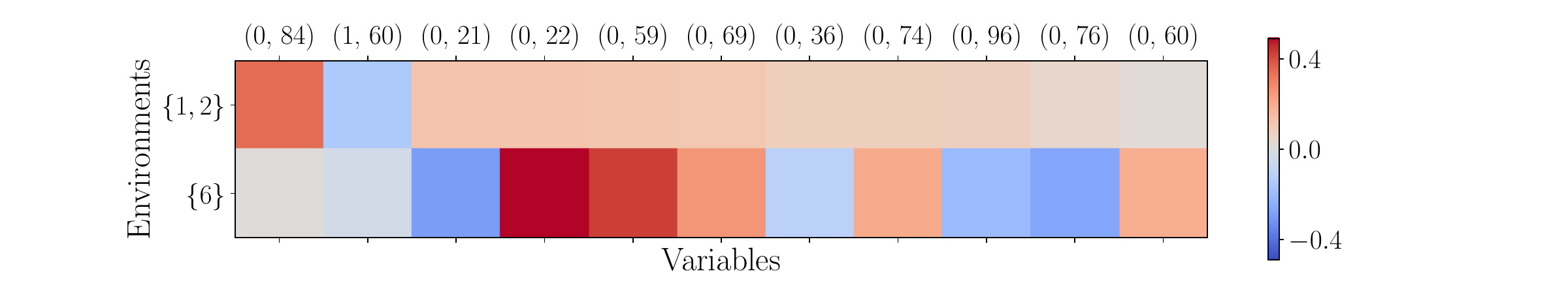}
    \caption{The estimated coefficients of the selected variables are shown for $\mathcal{D}_1 \cup \mathcal{D}_2$ and $\mathcal{D}_6$. Warm colors represent positive coefficients, while cool colors indicate negative coefficients. Variables are denoted as $(\tau, j)$, where $\tau \in \{0, 1\}$ represents the time lag, and $j$ indicates the stock index.}
    \label{fig:stable}
\end{figure}

\subsection{Climate Dynamic Prediction}\label{sec:climate}
We apply our method to the NCEP-NCAR Reanalysis Dataset~\citep{1996BAMS...77..437K} provided by NOAA PSL, Boulder, Colorado, USA. The dataset is widely used in atmospheric and climate research. It comprises 10512 global grid points with a resolution of 2.5 degrees in both latitude and longitude, spanning multiple vertical levels, and is available on a daily timescale. The dataset encompasses a range of meteorological properties, including air temperature (\texttt{air}),  clear sky upward solar flux (\texttt{csulf}), surface pressure (\texttt{pres}), sea level pressure (\texttt{slp}), and others.

We treat the aforementioned four properties $\{\texttt{air}, \texttt{csulf},\texttt{slp},\texttt{pres}\}$ as four independent tasks. For each property $a\in \{\texttt{air}, \texttt{csulf},\texttt{slp},\texttt{pres}\}$, we perform time-series prediction on $\{Z_{a,t,j}\}_{t\in [T], j\in [r]}$, where $Z_{a,t,j}$ represents the (pre-processed) measure of the property $a$ in geometric region $j$ at timestep $t$. We omit the dependency on $a$ and denote it as $\{Z_{ t,j}\}_{t\in [T], j\in [r]}$ when $a$ is clear from context. See how the data is pre-processed in \aosversion{Appendix F.1}{\cref{sec:pre-processing}}. 

In our experiment, we consider the datasets for the years 1950 ($\mathcal D_1$) and 2000 ($\mathcal D_2$) as the two training environments, the year 2010 as the validation environment ($\mathcal D_3$), and the year 2020 as the test environment ($\mathcal D_4$), all on a daily timescale. The target is to predict a set of variables $\mathcal{Y}=\{j_1,\ldots,j_w\}\subseteq [60]$, namely $Y_t=(Z_{t,j_1},\ldots,Z_{t,j_w})^\top$, using all the variables from the past seven days as covariates, namely,
\begin{align*}
    X_t:=\left(Z_{t-1,1},\ldots,Z_{t-1,60},Z_{t-2,1},\ldots,Z_{t-2,60},\ldots,Z_{t-7,1},\ldots,Z_{t-7,60}\right).
\end{align*} Here $\mathcal{Y}$ is the variables that can be predicted by $X_t$ significantly better than simply using the null prediction model; see the formal procedure on determining $\mathcal{Y}$ in \aosversion{Appendix F.2}{\cref{sec:construction-of-target}}.

For competing estimators, we consider PCMCI$^+$ \citep{runge2020discovering} and Granger causality \citep{granger1969investigating} with the aid of $L_2$ penalty, along with the following three causality-oriented linear models: Causal Dantzig \citep{rothenhausler2019causal}, Anchor Regression~\citep{rothenhausler2021anchor} and DRIG \citep{shen2023causality} with the aid of $L_1$ penalty (if applicable). We use mean squared error (MSE) as both the validation metric and test metric, which is defined as
\begin{align}\label{eq:MSE}
    \mathrm{MSE}_{e}=\sum_{(X_t,Y_t)\in\mathcal{D}_e}\|Y_t-\hat{Y}(X_t)\|_2^2 \qquad e\in [4]
\end{align}
where $\hat{Y}(X_t)$ is the model's prediction. The hyper-parameters for each model are tuned using the validation environment $\mathcal{D}_3$ as described in \cref{alg:igr-algo}. 

\begin{table}[htb!]
    \centering
    \footnotesize
    \begin{tabular}{cccccc}
    \hline
        Data & \texttt{air} &\texttt{csulf} & \texttt{pres} & \texttt{slp}\\
        \hline
        IGR(Ours)&$ 3.7838 \pm 0.3281$&$2.0523\pm 0.0883$ &$1.6077\pm 0.1122$&$3.0466\pm0.1955$\\
        Causal Dantzig&  $4.3742\pm0.2099$&$2.6197\pm0.0455$  &$ 2.0429\pm0.1502$&$3.5819\pm0.3161$\\
        LASSO&  $3.9171\pm 0.3194$ &$2.1327\pm 0.0567$& $1.6726\pm 0.0897$&$3.1261\pm 0.1887$\\
        Anchor &$3.9007\pm0.2394$ & $2.1142\pm 0.0615$& $1.6638\pm 0.0981$&$3.1235\pm 0.1622$\\
        DRIG & $ 3.9579\pm 0.2594$&$ 2.1844\pm 0.1233$& $1.7235\pm0.1176$&$3.2890\pm 0.1618$\\
        Granger&$ 4.3174\pm 0.3842$&$2.3182\pm 0.0736$&$1.8470\pm 0.1222$ &$3.5308\pm 0.1484$\\
        PCMCI$^+$ & $4.3533\pm 0.3062$&$2.4024\pm 0.0422$& $1.9499\pm 0.1213$&$3.6627\pm 0.2711$\\
        \hline
    \end{tabular}
    \caption{The average $\pm$ standard deviation of the mean squared error~\eqref{eq:MSE} of the four tasks air temperature (\texttt{air}),  clear sky upward solar flux (\texttt{csulf}), surface pressure (\texttt{pres}) and sea level pressure (\texttt{slp}) using different estimators.}
    \label{table: climate}
\end{table}

This process is repeated $100$ times. For each trail, we use a random sample of $300$ data in each training environment $\mathcal{D}_1, \mathcal{D}_2$ to fit the model. The average $\pm$ standard deviation of the average mean squared error on the test environment $\mathcal{D}_4$ of each method for each task is reported in \cref{table: climate}. The quantitative results show that our method outperforms all competing methods across all tasks, indicating that IGR can provide more robust predictions. We also qualitatively visualize the causal relation detected by our method; see \aosversion{Appendix F.4}{\cref{appendix:climate-causal}}.

\appendix

\bibliographystyle{apalike2}
\bibliography{main.bbl}

\newpage
\section*{Supplemental Materials}

The supplemental materials are organized as follows:

\begin{itemize}[itemsep=0pt, leftmargin=20mm]
    \item[\cref{appendix:comp}] provides additional discussions about the computation barrier omitted in the main text.
    \item[\cref{sec:diss}] contains the discussions omitted in the main paper.
    \item[\cref{sec:proof:barrier}] contains the proofs for the computation barrier results.
    \item[\cref{sec:proof:population}] contains the proofs for the population-level results.
    \item[\cref{sec:proof:non-asymptotic}] contains the proofs for the finite sample results.
\end{itemize}

\section{More Discussions on Computational Barriers}
\label{appendix:comp}

In this section, we (1) answer the following question related to the NP-hardness of the problem {\sc ExistLIS}; and (2) prove \cref{prop:3sat} in the main text. 

\begin{itemize}
    \item[Q1] The setup of \eqref{model:lip} are searching for a set with weaker invariance conditions adopted in causal discovery literature. Can stronger invariance conditions like the full distributional invariance in \cite{peters2016causal} help?
    \item[Q2] The NP-hardness of 0/1 Knapsack Problem relies on the exponential total budget, and there is $\mathrm{poly}(\text{number of items}, \text{total budget})$ algorithm. Is the NP-hardness in {\sc ExistLIS} due to the existence of many varying solutions with heterogeneity signal $e^{-\Omega(d)}$ such that a computationally efficient algorithm is possible if all the non-invariant solutions have large heterogeneity signals?
    \item[Q3] The covariance matrices in the construction \cref{lemma:reduction} is dense. Is computationally efficient estimation attainable when the covariance matrices are all sparse? 
\end{itemize}

The brief answers to the above questions are all ``No'', and the rigorous statements can be found in the following subsections.

\subsection{Stronger Invariance Condition Cannot Help}
\label{sec:barrier-1}

We consider the following task which is a special case of pursuing the distributional invariance, under which the conditional independence test can be easily done by doing simple calculations on the full covariance matrix on $(X^\top, Y)^\top$. To be specific, $(X, Y)$ are multivariate normal distributed with positive definite full covariance matrix in each environment. The decision problem can be described as follows.

\begin{problem}
\label{prob:exist-dis}
Let $d \in \mathbb{N}^+$ be the dimension of the explanatory covariate, and $E$ be the number of environments. We assume that for each $e\in \{1,\ldots, E\}$, 
\begin{align*}
    \begin{bmatrix}
        X^{(e)} \\
        Y^{(e)}
    \end{bmatrix} \sim \mathcal{N}\left(0, \begin{bmatrix}
        \Sigma^{(e)} & u^{(e)} \\
        (u^{(e)})^\top & v^{(e)}
    \end{bmatrix}\right)
\end{align*} where $\Sigma^{(1)}, \ldots, \Sigma^{(E)} \in \mathbb{R}^{d\times d}$ are positive definite matrices, $u^{(1)}, \ldots, u^{(E)}$ are $d$-dimensional vectors, and $v^{(e)}$ is a scalar satisfying $v^{(e)} > (u^{(e)})^\top (\Sigma^{(e)})^{-1} u^{(e)}$. We say a set $S$ is a non-trivial distribution-invariant set if
\begin{align}
\label{eq:dis}
    \beta^{(e,S)} \equiv {\beta}^{(S)} \neq 0 \qquad \text{and} \qquad v^{(e)} - (\beta^{(e,S)}_S)^\top \Sigma_S^{(e)} \beta^{(e,S)}_S \equiv v_{\varepsilon}.
\end{align}
    We define the problem {\sc ExistDIS} as follows:
    
    \noindent \textbf{$[\mathsf{Input}]$} $\Sigma^{(1)}, \ldots, \Sigma^{(E)}$,  $u^{(1)}, \ldots, u^{(E)}$ and $v^{(1)}, \ldots, v^{(E)}$  satisfying the above constraints.\\
    \noindent \textbf{$[\mathsf{Output}]$} Returns 1 if there exists a non-trivial distribution-invariant set otherwise 0.

    \noindent We define the problem {\sc ExistDIS-Unique} as the same problem with the promise that the non-trivial distribution-invariant set is unique if exists.
\end{problem}

The following lemma shows that \eqref{eq:dis} is equivalent to the full distribution invariance condition (Assumption 1 in \cite{peters2016causal}) under the setting in \cref{prob:exist-dis}.

\begin{lemma}
\label{lemma:equiv}

    Under the setting of \cref{prob:exist-dis}, $S$ satisfies \eqref{eq:dis} if and only if 
    \begin{align}
    \label{eq:dis-invar}
        \exists \bar{\beta} \in \mathbb{R}^d, \beta_{S} \neq 0 ~~ \text{s.t.} ~~ 
        Y^{(e)} = X^{(e)}_S \bar{\beta}_S + \varepsilon^{(e)} ~~ \text{with} ~~ \varepsilon^{(e)} \sim F_\varepsilon \indep X_{S}^{(e)} ~~~~ \forall e\in \{1,\ldots, E\}        
    \end{align}
\end{lemma}

It is then easy to see that the problem {\sc ExistDIS-Unique} corresponds to the case where the non-trivial invariant set is unique if it exists and ICP \citep{peters2016causal} can uniquely identify $S^\star$. We have the following result. The proof idea is that we construct the problem such that the additional invariant noise-level constraint trivially holds for all the prediction-invariant solutions.

\begin{theorem}
\label{thm:np-hard-dis}
When $E=2$, the problem {\sc ExistsDIS} is NP-hard under deterministic polynomial-time reduction, the problem {\sc ExistsDIS-Unique} is NP-hard under randomized polynomial-time reduction.
\end{theorem}

\subsection{NP-hardness Remains when It is Well-Separated}
\label{sec:np-est-error}

For any fixed $\epsilon \in (0,1)$, consider the following restricted version of the problem.
\begin{problem}[Existence of Linear Invariant Set under $\epsilon$-Separation]
\label{prob-sep}
    For any fixed constant $\epsilon > 0$, Problem {\sc Exist-$\epsilon$-Sep-LIS} is defined as the same problem as {\sc ExistLIS} with the additional $\epsilon$-separation conditions as follows 
    
    \noindent (a) $1 \le \mathbb{E}[|Y^{(e)}|^2] \le 1000$ for any $e\in[E]$; 
    
    \noindent (b) $\frac{1}{|E|}\sum_{e=1}^E\|\beta^{(e,S)}_S - \beta^{(S)}_S\|_{\Sigma_S^{(e)}}^2 \in \{0\} \cup [ d^{-\epsilon}/1280,\infty)$ for any $S\subseteq [d]$.

    \noindent (c) $\|\beta^{(S)}-\beta^{(S^\dagger)}\|_{\Sigma}^2\in [\indicator\{S\neq S^\dagger\} d^{-\epsilon}/1280,\infty)$ for any $S \subseteq [d]$ and any invariant set $S^\dagger$.
\end{problem}
\newcommand{\probsep}{{\sc Exist-$\epsilon$-Sep-LIS}}

Condition (a) promises $O(1)$ variance for the response, which is a typical regime considered by linear regression analysis. Condition (b) enforces the prediction variation should be $\Omega(d^{-\epsilon})$ if it is not an invariant set, and condition (c) assures that the non-invariant prediction should be $\Omega(d^{-\epsilon})$ away from the invariant prediction. The next theorem confirms that NP-hardness remains under this restrictive case. 

\begin{theorem}
\label{thm:gap2}
For any fixed $\epsilon>0$, the problem {\sc Exist-$\epsilon$-Sep-LIS} is NP-hard under deterministic polynomial-time reduction. 
\end{theorem}

The above construction also naturally implies that the computation barrier remains if our target is to find a solution close to some invariant solution within $O(d^{-\epsilon})$ error. 

\begin{problem}
\label{prob:est-d-eps}
Consider the problem~\probsep{} with $E=2$ and suppose further $Y^{(e)}=(\beta^{(e, [d])})^\top X^{(e)}$, i.e., there is no intrinsic noise. The input is the same, and it is required to output $\hat{\beta} \in \mathbb{R}^d$ such that $\inf_{S: \beta^{(e,S)} \equiv \beta^{(S)} \neq 0} \|\hat{\beta} - \beta^{(S)} \|_{\Sigma}^2 \le d^{-\epsilon}/4$ if $\{S: \beta^{(e,S)} \equiv \beta^{(S)} \neq 0\} \neq \emptyset$, its output can be an arbitrary $d$-dimensional vector otherwise. 
\end{problem}
\begin{corollary}
\label{coro:est-d-eps}
If \cref{prob:est-d-eps} can be solved by a worst-case polynomial-time algorithm, then {\sc 3Sat} can also be solved by a worst-case polynomial-time algorithm.
\end{corollary}

The key idea is to divide the set $[d]$ into two blocks: a block with size $d^{\epsilon/3}$, whose construction is similar to \cref{lemma:reduction}, and a remaining auxiliary block, where there is no invariant solution in this block and the predictive variance is carefully controlled. It is interesting to see if a similar result holds for $\epsilon=0$. We leave it for future studies.

\subsection{NP-Hardness Remains for Row-wise $O(1)$-Sparse Covariance}

The following theorem shows the problem {\sc ExistLIS} is NP-hard even when each row or column of matrix $\Sigma^{(e)},e\in\mathcal{E}$ has only $O(1)$ non-zero elements.

\begin{theorem}
\label{thm:sparse-construction}
Consider the problem {\sc ExistLIS} with the additional constraint that for any $e\in[E]$, each row of matrix $\Sigma^{(e)}$ has no more than $C$ non-zero elements for some universal constant $C>0$. The above problem is NP-hard under deterministic polynomial-time reduction when $E=2$. 
\end{theorem}

The proof idea is as follows. We first reduce the general {\sc 3Sat} problem $x$ with $k$ clauses to another {\sc 3Sat} problem $x'$ with $O(k^2)$ clauses. In $x'$, each variable at most appears on $15$ times. This will further lead to a row-wise sparse $A$ in \cref{lemma:reduction}. A finer construction will also adopted to distribute the constraints imposed by the last dense row of $\Sigma^{(2)}$ into $O(k^2)$ sparse rows. 

\subsection{Proof of \cref{prop:3sat}}
\label{sec:proof-3sat}

The proof is similar to Theorem 1.1 in \cite{valiant1985np}. We will use the following lemma akin to their Lemma 2.1. For any $u, v\in \{0,1\}^n$, we let $u\cdot v$ be the inner product over GF[2] of $u, v$.

\begin{lemma}
\label{lemma:uv=0}
    Given any {\sc 3Sat} formula $f$ with $n$ variables $v = (v_i)_{i=1}^n$ and $w \in \{0,1\}^n$, let $\mathcal{S}$ be the set of all the variables that make the formula evaluates to be true. One can construct a {\sc 3Sat} formula $f'$ in at most $2n$ variables $v'$ and $3(n-1)+k+1$ clauses such that there exists a bijective between its solution set $\mathcal{S}'$ and $\mathcal{S} \cap \{v: v\cdot w = 0\}$.
\end{lemma}
\begin{proof}[Proof of \cref{lemma:uv=0}]
Let $i_1, \ldots, i_m$ be the indices with $w_{i_j}=1$. The constraint $w \cdot v=0$ can be written as $v_{i_1} \otimes \cdots \otimes v_{i_m} = 0$, where $\otimes$ is the XOR operation,  which is equivalent to
\begin{align*}
v_{i_1} \otimes v_{i_2} = t_1, ~~v_{i_3} \otimes t_{1} = t_2, ~~~~\cdots ~~~~ v_{i_{m-1}} \otimes t_{m-2} = t_{m-1}, ~~ v_{i_m} \otimes t_{m-1} = t_m, ~~ t_m = 0
\end{align*}
with another $m \le n$ binaries variables $t_1, \cdots, t_{m-1}, t_m$. The constraint $x_1 \otimes x_2 = x_3$ is equivalent to the following 4-clause {\sc 3Sat} formula
\begin{align*}
    (\neg x_1 \lor \neg x_2 \lor \neg x_3) \land (x_1 \lor \neg x_2 \lor x_3) \land (\neg x_1 \lor x_2 \lor x_3) \land (x_1 \lor x_2 \lor \neg x_3)
\end{align*} The last constraint $t_m = 0$ can be written as the clause $(\neg t_m \lor \neg t_m \lor \neg t_m)$. 
\end{proof}

The rest of the proof is the same as that in Theorem 1.1: we can construct a randomized polynomial reduction from {\sc 3Sat} to {\sc 3Sat-Unique}. \qed

\section{Omitted Discussions}
\label{sec:diss}

\subsection{Discussion on \cite{li2024fairm}}
\label{sec:diss1}
\begin{problem}
\label{prob:cic}
Under the same setting as \cref{prob1} with $E=2$, it takes $(\Sigma^{(1)}, \Sigma^{(2)})$ and $(u^{(1)}, u^{(2)})$ as input and is required to determine whether there exists $S\subseteq [d]$ with $|S|\ge d/7$ such that $\Sigma_S^{(1)} = \Sigma_S^{(2)}$ and $u_S^{(1)} = u_S^{(2)}$. 
\end{problem}
Here we test the existence of any ``large'', namely $|S|\ge d/7$, covariance-invariant set rather than any covariance-invariant set, this is because $S$ is covariance-invariant in \cref{prob:cic} will imply $\{j\}$ is covariance-invariant for any $j\in S$. Testing the existence of a univariate invariant set is trivial and has $O(d \cdot E)$ algorithm. The proof is similar to \cref{thm:np-hard-p1} by letting $u^{(1)}=u^{(2)}$. 
\begin{theorem}
\label{thm:cic}
 \cref{prob:cic} is NP-hard.
\end{theorem}

\subsection{Discussion on \cref{cond:restricted-invariance}}

To show when \cref{cond:restricted-invariance} holds for small $k$ under the structural causal model framework, we first introduce the setting of SCM with intervention on $X$; see also \cite{gu2024causality} Section 3. We first introduce the definition of SCM and the setting considered.

\begin{definition}[Structural Causal Model] A structural causal model $M=(\mathcal{S}, \nu)$ on $p$ variables $Z_1,\ldots, Z_p$ can be described using $p$ assignment functions $\{f_1,\ldots, f_p\} = \mathcal{S}$: 
\begin{align*}
    Z_j \leftarrow f_j(Z_{\mathtt{pa}(j)}, U_j) \qquad j=1,\ldots, p,
\end{align*} 
where $\mathtt{pa}(j) \subseteq \{1,\ldots, p\}$ is the set of parents, or the direct causes, of the variable $Z_j$, and the joint distribution $\nu(du) = \prod_{j=1}^p \nu_j(du_j)$ over $p$ independent exogenous variables $(U_1,\ldots, U_p)$. For a given model $M$, there is an associated directed graph $\mathfrak{G}(M)=(V, E)$ that describes the causal relationships among variables, where $V=[p]$ is the set of nodes, $E$ is the edge set such that $(i,j)\in E$ if and only if $i\in \mathtt{pa}(j)$. $\mathfrak{G}(M)$ is acyclic if there is no sequence $(v_1,\ldots, v_k)$ with $k\ge 2$ such that $v_1=v_k$ and $(v_i, v_{i+1}) \in E$ for any $i\in [k-1]$.
\end{definition}

As in \cite{peters2016causal}, we consider the following data-generating process in $|\mathcal{E}|$ environments. For each $e\in \mathcal{E}$, the process governing $p=d+1$ random variables $Z^{(e)} = (Z_1^{(e)},\ldots, Z_{d+1}^{(e)})=(X_1^{(e)}, \ldots, X_d^{(e)},Y^{(e)})$ is derived from an SCM $M^{(e)}(\mathcal{S}^{(e)}, \nu)$. We let $e_0 \in \mathcal{E}$ be the observational environment for reference and the rest are interventional environments. We let $\mathfrak{G}$ be the directed graph representing the causal relationships in $e_0$, and simply let $\mathfrak{G}$ be shared across $\mathcal{E}$ without loss of generality. We assume $\mathfrak{G}$ is acyclic. In each environment $e\in \mathcal{E}$, the assignments are as follows:
\begin{align}
\label{eq:scm-model} 
\begin{split}
    X_j^{(e)} &\leftarrow f_j^{(e)}(Z_{\mathtt{pa}(j)}^{(e)}, U_j), \qquad \qquad  j=1,\ldots, d \\
    Y^{(e)} &\leftarrow f_{d+1}(X_{\mathtt{pa}(d+1)}^{(e)}, U_{d+1}).
\end{split}
\end{align}
Here the distribution of exogenous variables $(U_1,\ldots, U_{d+1})$, the cause-effect relationship $\{\mathtt{pa}(j)\}_{j=1}^{d+1}$ represented by $\mathfrak{G}$, and the structural assignment $f_{d+1}$ are \emph{invariant} across $e\in \mathcal{E}$, while the structural assignments for $X$ may vary among $e\in \mathcal{E}$. The heterogeneity, which is emphasized by superscript $(e)$ is due to the arbitrary interventions on the variables $X$. We use $Z_{\mathtt{pa}(j)}$ to emphasize that $Y$ can be the direct cause of some variables in the covariate vector. 

We denote $I\subseteq [d]$, defined as $I:=\{j: f^{(e)}_j \neq f^{(e_0)}_j ~~ \text{for some} ~ e\in \mathcal{E}\}$, be the set of variables intervened, 
We summarize the above data-generating process as a condition. 

\begin{condition}
\label{cond:scm-model}
Suppose $\{M^{(e)}\}_{e\in \mathcal{E}}$ are defined by \eqref{eq:scm-model}, $\mathfrak{G}$ is acyclic, and $f_{d+1}$ is a linear function.
\end{condition}

\begin{proposition}
\label{prop:ri}
    Under the model \eqref{model:lip} with regularity condition \cref{cond:regularity}, suppose one of the following conditions holds.

    \begin{itemize}[itemsep=0pt]
        \item[(a)] There exists a partition of $S^\star = \cup_{l=1}^L S^\star_l$ such that $\mathbb{E}[X_{S^\star_l}^{(e)} (X_{S^\star_r}^{(e)})^\top] = 0$ for any $l\neq r$ and $|S_l^\star|\le k$. 
        \item[(b)] Assume \cref{cond:scm-model} holds such that we can define the ancestor set recursively as $\mathtt{at}(j) = \mathtt{pa}(j) \cup \bigcup_{k\in \mathtt{pa}(j)} \mathtt{at}(k)$. We have $I\cap \mathtt{at}(d+1)=\emptyset$, $S^\star = \mathtt{pa}(d+1)$, and $k\ge 1$.
    \end{itemize}

    Then \cref{cond:restricted-invariance} holds.
\end{proposition}

\begin{proof}[Proof of \cref{prop:ri}]
We first prove (a). To be specific, we show that
\begin{align*}
    \forall e\in \mathcal{E}, l\in [L], \qquad \beta^{(e,S_l^\star)}_{S_l^\star} = \beta^{\star}_{S_l^\star}.
\end{align*} It follows from \cref{cond:regularity} and the definition of least squares that 
\begin{align*}
    \beta^{(e,S_l^\star)}_{S_l^\star} &= \left[\Sigma_{S_l^\star}^{(e)}\right]^{-1} \mathbb{E}[X_{S_l^\star}^{(e)} Y^{(e)}] \\
    &= \left[\Sigma_{S_l^\star}^{(e)}\right]^{-1} \mathbb{E} \left[X_{S_l^\star}^{(e)} \left(\varepsilon^{(e)} + (X_{S_l^\star}^{(e)})^\top \beta^\star_{S_l^\star} + \sum_{r\neq l}  (X_{S_{r}^\star}^{(e)})^\top \beta^\star_{S_{r}^\star} \right)\right] \\
    &\overset{(i)}{=} \left[\Sigma_{S_l^\star}^{(e)}\right]^{-1} \left\{\mathbb{E} \left[X_{S_l^\star}^{(e)} \varepsilon^{(e)} \right] + \Sigma_{S_l^\star}^{(e)} \beta^\star_{S^\star_l} + \sum_{r \neq l}  \mathbb{E}[X_{S^\star_l}^{(e)} (X_{S^\star_r}^{(e)})^\top]\beta^\star_{S^\star_r} \right\} \\
    &= \beta^{\star}_{S_l^\star}
\end{align*} where (i) follows from the exogeneity of $X_{S^\star}$ in \eqref{model:lip} and (a).

Now we prove (b). Given the condition in (b), we have for any $j\in S^\star=\mathtt{pa}(d+1)$, and $e,e'\in \mathcal{E}$
\begin{align*}
    \beta_j^{(e)} = \frac{\mathbb{E}[X^{(e)}_j Y^{(e)}]}{\mathbb{E}[|X_j^{(e)}|^2]} = \frac{\mathbb{E}[X^{(e')}_j Y^{(e')}]}{\mathbb{E}[|X_j^{(e')}|^2]} = \beta_j^{(e')}.
\end{align*}

\end{proof}

\section{Proofs for Computation Fundamental Limits}
\label{sec:proof:barrier}

\subsection{Proof of \cref{lemma:gap}}
\label{sec:lemma:gap}

\begin{proof}[Proof of \eqref{eq:gap-est-error}]
We first establish the upper bound in \eqref{eq:gap-est-error}. It follows from the definition of $\beta^{(S)}$ that
\begin{align*}
    \|\beta^{(S)} - \beta^{(S^\dagger)}\|_{\Sigma}^2 &= u_S \Sigma_S^{-1} u_S + u_{S^\dagger} \Sigma_{S^\dagger}^{-1} u_{S^\dagger} - 2u_{S\cap S^\dagger} \Sigma_{S\cap S^\dagger}^{-1} u_{S\cap S^\dagger}\\
    &\overset{(a)}{\le} 2 \times \frac{1}{2} \sum_{e\in [2]} \mathbb{E}[|Y^{(e)}|^2]
    \le \sum_{e\in [2]} \mathbb{E}[|Y^{(e)}|^2].
\end{align*} here $(a)$ follows from the fact that the pooled full covariance matrix $\begin{bmatrix}
        \Sigma & u \\
        u^\top & \frac{1}{2} \sum_{e\in \mathcal{E}} \mathbb{E}[|Y^{(e)}|^2]
    \end{bmatrix}$ is positive semi-definite. 

Now we turn to the lower bound. We denote $\tilde{A} = \begin{bmatrix}
        A & \frac{1}{2} 1_k \\
        \frac{1}{2} 1_k^\top & 0
    \end{bmatrix}$. It is easy to see that $\|\tilde{A}\|_F \le d$, combining this with the fact that $\|\tilde{A}\|_2 \le \|\tilde{A}\|_F$, the maximum and minimum eigenvalue of $\Sigma^{(2)}$ can be controlled by 
\begin{align}
\label{eq:eigen-ul-bound}
4d \le 32k - \|\tilde{A}\|_2 \le \lambda_{\min}(\Sigma^{(2)}) \le \lambda_{\max}(\Sigma^{(2)}) \le 5d + \|\tilde{A}\|_2 \le 6d
\end{align}
When there is no intrinsic noise, the variance of $Y^{(e)}$ can be exactly calculated as
\begin{align*}
    \mathbb{E}[|Y^{(1)}|^2] = (u^{(1)})^\top (\Sigma^{(1)})^{-1} u^{(1)} = d
\end{align*} and upper bounded as 
\begin{align*}
    \mathbb{E}[|Y^{(2)}|^2] = (u^{(2)})^\top (\Sigma^{(2)})^{-1} u^{(2)} \le \left(\lambda_{\min}(\Sigma^{(2)}) \right)^{-1} \|u^{(2)}\|_2^2 \le \frac{1}{4d} \times d\times (6d)^2 \le 9d^2.
\end{align*} Therefore, we have
\begin{align}
\label{eq:variance-upper-bound}
    \sum_{e\in [2]} \mathbb{E}[|Y^{(e)}|^2] \le 10d^2. 
\end{align}

On the other hand, by \eqref{eq:eigen-ul-bound}, we obtain
\begin{align}
\label{eq:diff-lb}
\begin{split}
    \|\beta^{(S)} - \beta^{(S^\dagger)}\|_{\Sigma}^2 &= \left(\beta^{(S)} - \beta^{(S^\dagger)}\right)^\top \Sigma \Sigma^{-1} \Sigma \left(\beta^{(S)} - \beta^{(S^\dagger)}\right) \\
    &\ge \frac{1}{\lambda_{\max}(\Sigma)} \left\|\Sigma\left(\beta^{(S)} - \beta^{(S^\dagger)}\right)\right\|_2^2 \ge \frac{1}{7d} \left\|\Sigma\left(\beta^{(S)} - \beta^{(S^\dagger)}\right)\right\|_2^2.
\end{split}
\end{align}
We denote $S_1 = S\setminus S^\dagger$ and $S_2 = S^\dagger \setminus S$. We will establish the lower bound on $\|\Delta\|_2^2$ for $\Delta=\Sigma(\beta^{(S)} - \beta^{(S^\dagger)}) \in \mathbb{R}^d$ when $S\neq S^\dagger$. Given $S\neq S^\dagger$, one has either $S_1 \neq \emptyset$ or $S_2 \neq \emptyset$. Without loss of generality, we assume that $S_2 \neq \emptyset$. 

First, one has
\begin{align*}
    \|\beta^{(S)}_S\|_2 &= \left\| (\Sigma_S)^{-1} u_S \right\|_2
    = \left\|\left(\frac{5d+1}{2} I_{|S|} + \frac{1}{2} \tilde{A}_S \right)^{-1} u_S \right\|_2 \\
    &= \left\|\left(I_{|S|} + \frac{1}{5d+1} \tilde{A}_S\right)^{-1} \frac{2}{5d+1} u_S \right\|_2 \\
    &\overset{(a)}{\le} \left\|\left(I_{|S|} + \frac{1}{5d+1} \tilde{A}_S\right)^{-1} - I_{|S|} \right\| \frac{2}{5d+1} \|u_S\|_2 + \frac{2}{5d+1} \|u_S\|_2\\
    &\overset{(b)}{\le} \left(1 + 2\frac{d}{5d+1}\right) \frac{2}{5d+1} \frac{5d+1+0.5k}{2} \sqrt{d} \\
    &\le (1 + 2/5) \times (1 + 0.5/5) \sqrt{d} \le 1.5 \sqrt{d}.
\end{align*} Here $(a)$ follows from the triangle inequality, $(b)$ follows from the fact that $\|(I+M)^{-1} -I\|_2 \le 2\|M\|$ if $\|M\| \le 0.5$. Pick $j\in S_2$, it follows from the above upper bound, the fact $j\notin S$ and Cauchy Schwarz inequality that
\begin{align*}
    \Delta_j &= \Sigma_{j, S}^\top \beta_S^{(S)} - u_j \\
    &\le \|\tilde{A}_{j, S}\|_2 \|\beta_S^{(S)}\|_2 - \frac{1}{2}(5d + 0.5 + 1) \\
    &\le 1.5 d - 2.5 d - 0.75 \le -d - 0.75.
\end{align*} This further yields that $\|\Delta\|_2^2 \ge \|\Delta_j\|^2 \ge d^2$. 
Combining it with \eqref{eq:diff-lb} and \eqref{eq:variance-upper-bound} completes the proof of the lower bound.

\end{proof}

\begin{proof}[Proof of \eqref{eq:gap-hetero}]
For the upper bound, we have
\begin{align*}
    \sum_{e\in [2]} \|\beta^{(S)} - \beta^{(e,S)}\|_{\Sigma^{(e)}}^2 &= \min_{\supp(\beta)\subseteq S}  \sum_{e\in [2]} \|\beta^{(e,S)} - \beta\|_{\Sigma^{(e)}}^2 \\ &\le \sum_{e\in [2]} \|\beta^{(e,S)}\|_{\Sigma^{(e)}}^2 \overset{(a)}{=} \sum_{e\in [2]} u_S^{(e)} (\Sigma_S^{(e)})^{-1} u_S^{(e)} \\
    &\overset{(b)}{\le} \sum_{e\in [2]} \mathbb{E}[|Y^{(e)}|^2]. 
\end{align*} Here $(a)$ follows from the definition of $\beta^{(e,S)}$, $(b)$ follows from the fact that the following covariance matrix $\begin{bmatrix}
        \Sigma^{(e)} & u^{(e)} \\
        (u^{(e)})^\top & \mathbb{E}[|Y^{(e)}|^2]
    \end{bmatrix}$ is positive semi-definite. 

Turning to the lower bound, 
\begin{align*}
    \sum_{e\in [2]} \|\beta^{(S)} - \beta^{(e,S)}\|_{\Sigma^{(e)}}^2 &\ge \|\beta^{(S)} - \beta^{(1,S)}\|_2^2 \\
    &= \left\|\Sigma_S^{-1} (0.5 u_S^{(1)} + 0.5 u_S^{(2)}) - u_S^{(1)}\right\|_2^2 \\
    &\ge [\lambda_{\max}(\Sigma)]^{-2} \left\|0.5 u_S^{(1)} + 0.5 u_S^{(2)} - 0.5 u_S^{(1)} - 0.5 \Sigma_S^{(2)} u_S^{(1)} \right\|_2 \\
    &\ge [4\lambda_{\max}(\Sigma)]^{-2} \|2\Sigma_S^{(2)} u_S^{(1)} - u_S^{(2)} \|_2^2
\end{align*} Observe that all the entries in the vector $2\Sigma_S^{(2)} u_S^{(1)} - u_S^{(2)}$ are integer. Then unless $\Sigma_S^{(2)} u_S^{(1)} = u_S^{(2)}$, in other words, $S$ is a invariant set by \cref{def:mis}, we have $\|2\Sigma_S^{(2)} u_S^{(1)} - u_S^{(2)} \|_2^2 \ge 1$. Therefore, we have
\begin{align*}
\sum_{e\in [2]} \|\beta^{(S)} - \beta^{(e,S)}\|_{\Sigma^{(e)}}^2 &\ge \|\beta^{(S)} - \beta^{(1,S)}\|_2^2 \ge [4\lambda_{\max}(\Sigma)]^{-2} \ge 784d^{-2}
\end{align*} if $S$ is not a invariant set. Combining it with the upper bound \eqref{eq:variance-upper-bound} completes the proof. 

\end{proof}

\subsection{Proofs in \cref{sec:barrier-1}}

\begin{proof}[Proof of \cref{lemma:equiv}] Denote $\hat{\varepsilon}^{(e)} = Y^{(e)} - (\beta^{(e,S)})^\top X^{(e)}$, we have
\begin{align*}
    \text{Cond \eqref{eq:dis-invar}} \qquad &\Longleftrightarrow \qquad \beta^{(e,S)} \equiv \bar{\beta}_S \neq 0 \text{ and } F_\varepsilon \sim \hat{\varepsilon}^{(e)} \indep X_S^{(e)} \\
    & \overset{(a)}{\Longleftrightarrow} \qquad \beta^{(e,S)} \equiv \bar{\beta}_S \neq 0 \text{ and } \mathrm{var}(\hat{\varepsilon}^{(e)}) \equiv v_\varepsilon \qquad \overset{(b)}{\Longleftrightarrow} \qquad \text{Cond \eqref{eq:dis}}
\end{align*} where (a) follows from the fact that $(X, Y)$ are multivariate Gaussian under which independence is equivalent to uncorrelatedness and the fact that $\hat{\varepsilon}^{(e)}$ is also Gaussian, (b) follows from the fact that
\begin{align*}
    \mathrm{var}(\hat{\varepsilon}^{(e)}) &= \mathbb{E}[|Y^{(e)}|^2] - 2 (\beta^{(e,S)}_S)^\top \mathbb{E}[X_S^{(e)} Y^{(e)}] + (\beta^{(e,S)}_S)^\top \Sigma_{S}^{(e)} \beta^{(e,S)}_S \\
    &= v^{(e)} - (\beta^{(e,S)}_S)^\top \Sigma_{S}^{(e)} \beta^{(e,S)}_S.
\end{align*}
\end{proof}

\begin{proof}[Proof of \cref{thm:np-hard-dis}]
The proof is similar to that of \cref{thm:np-hard-p1}. For each instance $x$, we use the same reduction construction of $(\Sigma, u)$ in problem $y$ constructed in \cref{lemma:reduction} and let
\begin{align*}
    v^{(1)} = 100 d^{5} + k + 1 \qquad v^{(2)} = 100 d^5 + 5d(k+1) + k,
\end{align*} this furnishes a new problem $\tilde{y}$ of {\sc ExistDIS}. It is easy to see that
\begin{align*}
    \forall e\in [2], \qquad v^{(e)} - (u^{(e)}_S)^\top (\Sigma_S^{(e)})^{-1} u^{(e)}_S \ge 1.
\end{align*} Moreover, for any valid solution $S\in \mathcal{S}_y$, one has
\begin{align*}
    v^{(1)} - \beta^{(1,S)}_S \Sigma_S^{(1)} \beta^{(1,S)}_S = v^{(1)} - 1_{|S|}^\top 1_{|S|} = v^{(1)} - (k+1) = 100 d^5 
\end{align*}and
\begin{align*}
    v^{(2)} - \beta^{(2,S)}_S \Sigma_S^{(2)} \beta^{(2,S)}_S &= v^{(2)} - 1_{|S|}^\top \Sigma^{(2)}_S 1_{|S|} \\
    &\overset{(a)}{=} v^{(2)} - \left(5d + 2 (|S| - 1) \frac{1}{2} + 1_k^\top (5d I_d + A_{\mathring{S}}) 1_{k}\right) \\
    &\overset{(b)}{=} v^{(2)} - 5d (1 + k) - k = 100 d^5.
\end{align*} Here $(a)$ follows from the fact that $d\in S$ provided $S\in \mathcal{S}_y$, $(b)$ follows from the fact that $A_{i,j}=0$ for any $i,j\in \mathring{S}$ and $|S| = k+1$ provided $S\in \mathcal{S}_y$. This further yields that $\mathcal{S}_y\subseteq \mathcal{S}_{\tilde{y}}$. Combined with the fact that $\mathcal{S}_{\tilde{y}} \subseteq \mathcal{S}_{y}$, one further has $\mathcal{S}_{\tilde{y}}=\mathcal{S}_{{y}}$. The rest of the proof follows similarly. 
\end{proof}

\subsection{Proof of \cref{thm:gap2}}

We adopt a similar reduction idea as that in \cref{lemma:reduction}. Without loss of generality, we assume $k\ge 10^4$ and $\epsilon<0.5$. 

We first introduce one additional notation. For any integer $\ell>0$, we define the positive definite $\ell\times\ell$ matrix $H_\ell$ as follows:
\begin{align}
\label{eq:reduction-g-mat-2}
    (H_\ell)_{j,j'} = \begin{cases}
        2 & \qquad j=j' \\
        1 & \qquad \text{otherwise}
    \end{cases}
\end{align}
for any $j,j'\in[\ell]$. Namely, $H_\ell =I_\ell+1_\ell 1_\ell^\top$ for any $\ell\ge 1$. One can thereby obtain $H_\ell^{-1}=I_\ell-\tfrac{1}{\ell+1}1_\ell 1_\ell^\top$.

\noindent {\sc Step 1. Construct the Reduction. } For any {\sc 3Sat} instance $x$ with input size $k$, we construct an {\sc ExistLIS} instance $y$ with size $d = \lceil k^{3/(\epsilon)} \rceil$ as follows:
\begin{align*}
    \Sigma^{(1)} = \begin{bmatrix}
         32k \cdot I_{7k+1}&0 \\
        0  & H_{d-7k-1}
    \end{bmatrix} \qquad \text{and} \qquad u^{(1)} = (k^{-1})\begin{bmatrix}
         32k\cdot1_{7k+1}\\
        1_{d-7k-1}
    \end{bmatrix},
\end{align*}
and
\begin{align*}
    \Sigma^{(2)} = \begin{bmatrix}
         32k\cdot I_{7k} + A & \frac{1}{2} \cdot 1_{7k}&0 \\
        \frac{1}{2} \cdot 1_{7k}^\top &  32k&0\\
        0&0&H_{d-7k-1}
    \end{bmatrix} \qquad \text{and} \qquad u^{(2)} = (k^{-1})\begin{bmatrix}
        (32k + \frac{1}{2}) \cdot 1_{7k}\\
         32k + \frac{1}{2} k\\
        -3\cdot1_{d-7k-1}
    \end{bmatrix}.
\end{align*}
One can observe that both $\Sigma^{(1)}$ and $\Sigma^{(2)}$ are respectively composed by an upper-left $(7k+1)\times(7k+1)$ matrix and a lower-right $(d-7k-1)\times(d-7k-1)$ matrix $H_{d-7k-1}$.
Recall that in the Proof of \eqref{eq:gap-est-error} we introduce the notation of matrix $\tilde{A} = \begin{bmatrix}
        A & \frac{1}{2} 1_k \\
        \frac{1}{2} 1_k^\top & 0
    \end{bmatrix}$, and we have $\|\tilde{A}\|_2\le\|\tilde{A}\|_F\le 7k+1$. Then similar to \eqref{eq:eigen-ul-bound}, the maximum and minimum eigenvalue of $\Sigma^{(2)}_{[7k+1]}$ can be controlled by
    \begin{align}\label{eq:eigen-ul-bound-2}
        24k\le 32k - \|\tilde{A}\|_2\le \lambda_{\min}(\Sigma^{(2)}_{[7k+1]})\le \lambda_{\max}(\Sigma^{(2)}_{[7k+1]})\le 32k+\|\tilde{A}\|_2\le 40k.
    \end{align}
 Combining with the fact that  $H_{\ell}$ is positive definite for any $\ell\ge 1$, we can conclude that both $\Sigma^{(1)}$ and $\Sigma^{(2)}$ are positive definite, and the above reduction can be calculated within polynomial time.

Now it suffices to show that (1) The above construction is a parsimonious reduction; and (2) The instance $y$ lies in the problem \probsep{}. In order to complete the remaining proof, it is helpful to observe that there are three modifications in this construction compared to the construction in \cref{lemma:reduction}.

\begin{itemize}
\item[(a)] We introduce an auxiliary $(d-7k-1)$-dimension part $[d]\setminus [7k+1]$. We will show that this part is precluded by any invariant set.

\item[(b)] We change the diagonal coordinates in $\Sigma_{[7k+1]}^{(1)}$ from $1$ to $32k$, and those in $\Sigma_{[7k+1]}^{(2)}$ from $5d$ to $32k$ to make $\mathbb{E}[|Y^{(1)}|^2]\asymp \mathbb{E}[|Y^{(2)}|^2]
$. We also change the coordinates of $u_{[7k+1]}^{(1)}$ and $u_{[7k+1]}^{(2)}$ accordingly.

\item[(c)] We add a $k^{-1}$ multiplicative factor in $u^{(1)}$ and $u^{(2)}$ to let $\mathbb{E}[|Y^{(1)}|^2],\mathbb{E}[|Y^{(2)}|^2] \asymp 1$. This will also result in all the $\beta^{(e,S)}$ and $\beta^{(S)}$ being multiplied by the same $k^{-1}$ factor.
\end{itemize}

\noindent {\sc Step 2. Verification of Parsimonious Reduction.} 
We first claim that the auxiliary $(d-7k-1)$-dimension part $[d]\setminus [7k+1]$ is precluded by any invariant set, namely
\begin{align}
\label{eq:empty-larger-than-7k}
\forall  S^\dagger \in \mathcal{S}_y \qquad \Longrightarrow \qquad S^\dagger \cap \{7k+2,\ldots, d\} = \emptyset.
\end{align}

To this end, we adopt the proof-by-contradiction argument. To be specific, if $j\in S^\dagger$ for some $j\in \{7k+2, \ldots, d\}$ and $S^\dagger\in\mathcal{S}_y$, then the equations $\Sigma^{(1)}_{S^\dagger}\beta^{(S^\dagger)}_{S^\dagger}=u^{(1)}_{S^\dagger}$ and $\Sigma^{(2)}_{S^\dagger}\beta^{(S^\dagger)}_{S^\dagger}=u^{(2)}_{S^\dagger}$ yields 
\begin{align*}
    u^{(1)}_j=\left[\Sigma^{(1)}_{S^\dagger}\beta^{(S^\dagger)}\right]_j=\Sigma_{j,S^\dagger}^{(1)}\beta^{(S^\dagger)}
\quad\text{and}\quad
    u^{(2)}_j=\left[\Sigma^{(2)}_{S^\dagger}\beta^{(S^\dagger)}\right]_j=\Sigma_{j,S^\dagger}^{(2)}\beta^{(S^\dagger)}.
\end{align*}
However, in our construction $\Sigma_{j,S^\dagger}^{(1)}=\Sigma_{j,S^\dagger}^{(2)}$ while $u^{(1)}_{j}\neq u^{(2)}_{j}$. This leads to a contradiction. Therefore, an invariant set should not contain any element in $\{7k+2,\ldots,d\}$.

By \eqref{eq:empty-larger-than-7k}, we have the following statements similar to \eqref{eq:equiv-reduction} in the proof of \cref{lemma:reduction}.
\begin{align}
 \label{eq:equiv-reduction-eps-sep}
 \begin{split}
     S^\dagger \in \mathcal{S}_y ~~~&\overset{(a)}{\Longleftrightarrow} ~~~ S^\dagger\neq \emptyset \text{ and }\beta^{(2,S^\dagger)} = \beta^{(1,S^\dagger)} \text{ with } \beta^{(1,S^\dagger)}_j = (k^{-1})\indicator\{j \in S^\dagger\} \\
     &\overset{(b)}{\Longleftrightarrow} ~~~  S^\dagger = \mathring{S} \cup \{7k+1\} ~\text{with}~|\mathring{S}|=k ~\text{and}~A_{j,j'} = 0~~\forall j,j'\in \mathring{S}\subseteq[7k]\\
     &\overset{(c)}{\Longleftrightarrow} ~~~ \mathring{S} = \{7(i-1)+a_i\}_{i=1}^k\text{ with }a_i \in [7]\text{ s.t. adopt action ID } a_i \\
     &~~~~~~~~~~~~ \text{ in clause } i\in [k] \text{ will lead to a valid solution }v\in \mathcal{S}_x\text{.} 
\end{split}
\end{align}
We emphasis that the proof of \ref{eq:equiv-reduction-eps-sep}(a) and (c) are essentially identical to those of \eqref{eq:equiv-reduction}. {For completeness, we prove (b)}.

\noindent \underline{\sc Proof of \eqref{eq:equiv-reduction-eps-sep} (b).} The proof is almost identical to the proof of \eqref{eq:equiv-reduction}(b) since the major difference is the $k^{-1}$ multiplicative factor. The direction $\Leftarrow$ is obvious. For the $\Rightarrow$ direction, we first show that $7k+1\in S^\dagger$ using the proof by contradiction argument. Suppose $|S^\dagger| \ge 1$ but $7k+1\notin S^\dagger$, we pick $j \in S^\dagger$, then
\begin{align*}
    k\left[\Sigma^{(2)}_{S^\dagger} \beta_{S^\dagger}^{(2,S^\dagger)}\right]_j = 32k + \sum_{j'=1}^{7k} A_{j, j'} \indicator\{j'\in S^\dagger\} \neq 32k + \frac{1}{2} = k\cdot u_j^{(2)}
\end{align*} where the first equality follows from the assumption $\beta^{(2,S^\dagger)}_j = \beta^{(1,S^\dagger)}_j = k^{-1} \indicator\{j\in S^\dagger\}$ and $7k+1\notin S^\dagger$, and the inequality follows from the fact that $A\in \{0,1\}^{7k\times 7k}$ hence the L.H.S. is an integer. This indicates that $\beta^{(1,S^\dagger)} \neq \beta^{(2,S^\dagger)}$ if $|S^\dagger| \ge 1$ and $7k+1\notin S^\dagger$. Given $7k+1\in S^\dagger$, we then obtain
\begin{align*}
    32k + \frac{1}{2} k = k\cdot u_{7k+1}^{(2)} = k\left[\Sigma^{(2)}_{S^\dagger} \beta_{S^\dagger}^{(2,S^\dagger)}\right]_{|S^\dagger|} = 32k + \frac{1}{2} \sum_{j'=1}^{7k} 1\{j'\in S^\dagger\} = 32k + \frac{1}{2} (|S^\dagger| - 1),
\end{align*} which implies that $|S^\dagger|=k+1$. Now we still have the constraint $u^{(2)}_{\mathring{S}}= \Sigma^{(2)}_{\mathring{S}} \beta^{(2,S^\dagger)}_{\mathring{S}} + \frac{1}{2}\cdot 1_k$. The last claim $A_{j',j} = 0$ for any $j',j\in \mathring{S}$ then follows from this by observing that \begin{align*}
\left(32k+\frac{1}{2}\right) \cdot 1_{k}= \Sigma^{(2)}_{\mathring{S}} 1_{k} + \frac{1}{2} \cdot 1_{k}  ~\Longrightarrow~ A_{\mathring{S}} 1_k = 0 ~\overset{(i)}{\Longrightarrow}~ A_{j',j}=0 ~~\forall j',j\in \mathring{S}
\end{align*} where $(i)$ follows from the fact that $A\in \{0,1\}^{7k\times 7k}$.

Therefore, we can conclude that this mapping is a parsimonious polynomial-time reduction from {\sc 3Sat} to {\sc ExistLIS}. Given the conditions (1) -- (3) further holds as verified below, the instance is an \probsep{} instance. Hence the problem \probsep{} is NP-hard. \qed

\begin{lemma}
\label{lemma:vf111}
The above constructed instance is an \probsep{} instance.     
\end{lemma}

\begin{proof}[Proof of \cref{lemma:vf111}]
\noindent {\sc Step 1 Calculating the Variance of $Y^{(e)}$ for $e\in\{1,2\}$. } Now we calculate $\mathbb{E}[|Y^{(1)}|^2]$ and $\mathbb{E}[|Y^{(2)}|^2]$. Without loss of generality we consider the cases where $Y^{(e)}$ is a linear combination of $X^{(e)}$ for $e\in\{1,2\}$, under which $\mathbb{E}[
|Y^{(e)}|^2]=(u^{(e)})^\top(\Sigma^{(e)})^{-1}u^{(e)}$ for $e\in\{1,2\}$. 

For $e=1$, we have
\begin{align*}
    \mathbb{E}[
|Y^{(1)}|^2]&=(u^{(1)})^\top(\Sigma^{(1)})^{-1}u^{(1)}\\
&\overset{(a)}{=}(u_{[7k+1]}^{(1)})^\top(\Sigma_{[7k+1]}^{(1)})^{-1}u_{[7k+1]}^{(1)}+(u_{[d]\setminus[7k+1]}^{(1)})^\top(\Sigma_{[d]\setminus[7k+1]}^{(1)})^{-1}u_{[d]\setminus[7k+1]}^{(1)}\\
&=(k^{-1})^{2}\left(\frac{1}{32k}(32k)^2(7k+1)+1_{d-7k-1}^\top H_{d-7k-1}^{-1}1_{d-7k-1}\right).
\end{align*}
Here $(a)$ follows from the fact that $\Sigma^{(1)}$ is a block diagonal matrix.
It follows from the identity $1_\ell^\top H_\ell^{-1}1_\ell=\ell/(1+\ell)$ that
\begin{align*}
    1&<(k^{-1})^{2}\frac{1}{32k}(32k)^2(7k+1)
    \\&\le \mathbb{E}[
|Y^{(1)}|^2]\\&\le (k^{-1})^{2}\left(32k \cdot (7k+1)+\frac{d-7k-1}{d-7k-1+1}\right)\\&< 256.
\end{align*}

Similarly, for $\mathbb{E}[
|Y^{(2)}|^2]$, following from the fact that $\Sigma^{(2)}$ is block diagonal, we obtain
\begin{align*}
    \mathbb{E}[
|Y^{(2)}|^2]&=(u^{(2)})^\top(\Sigma^{(2)})^{-1}u^{(2)}\\
&=(u_{[7k+1]}^{(2)})^\top(\Sigma_{[7k+1]}^{(2)})^{-1}u_{[7k+1]}^{(2)}+(u_{[d]\setminus[7k+1]}^{(2)})^\top(\Sigma_{[d]\setminus[7k+1]}^{(2)})^{-1}u_{[d]\setminus[7k+1]}^{(2)}\\
    &=(u^{(2)}_{[7k+1]})^\top(\Sigma^{(2)}_{[7k+1]})^{-1}u^{(2)}_{[7k+1]}+9(k^{-1})^{2}1_{d-7k-1}^\top H_{d-7k-1}^{-1}1_{d-7k-1}.
\end{align*}
Recall that $\lambda_{\max}(\Sigma_{[7k+1]}^{(2)})\le  40k$ and $\lambda_{\min}(\Sigma_{[7k+1]}^{(2)})\ge 24k$, we have
\begin{align*}
     1<(k^{-1})^{2}(40k)^{-1}(32k)^2(7k+1) &\le (u^{(2)}_{[7k+1]})^\top(\Sigma^{(2)}_{[7k+1]})^{-1}u^{(2)}_{[7k+1]}\\&\le (k^{-1})^{2}(24k)^{-1}(32k+k)^2(7k+1)\\&<999.
\end{align*}
Therefore, 
\begin{align*}
    1<\mathbb{E}[
|Y^{(2)}|^2]<999+9(k^{-1})^{2}<1000.
\end{align*}
Hence we can conclude that $1\le \mathbb{E}[|Y^{(1)}|^2],\mathbb{E}[|Y^{(2)}|^2] \le 1000$.

\noindent{\sc Step 2. Calculating the Prediction Variation.} Now we lower bound the heterogeneity gap $\frac{1}{|\mathcal{E}|}\sum_{e\in \mathcal{E}}\|\beta^{(e,S)}_S - \beta^{(S)}_S\|_{\Sigma_S^{(e)}}^2 \ge d^{-\epsilon}/1280$ when $S$ is not an invariant set as \cref{def:mis}. Denote $S_1=S\cap[7k+1]$ and $ S_2=S\setminus [7k+1] $. We divide it into two cases when $\beta^{(1,S)} \neq \beta^{(2,S)}$:
 
\noindent \emph{Case 1. $S_2\neq \emptyset$}: Observe $\Sigma^{(1)}$ and $\Sigma^{(2)}$ are block diagonal matrices, we have
\begin{align}
\label{eq:beta-S2}
\begin{split}
\beta^{(1,S)}_{S_2}&=\beta^{(1,S_2)}_{S_2}=H_{|S_2|}^{-1}u_{S_2}^{(1)},\\
\beta^{(2,S)}_{S_2}&=\beta^{(2,S_2)}_{S_2}=H_{|S_2|}^{-1}u_{S_2}^{(2)}=-3H_{|S_2|}^{-1}u_{S_2}^{(1)},\\
\beta^{(S)}_{S_2}&=\beta^{(S_2)}_{S_2}=(H_{|S_2|}+H_{|S_2|})^{-1}(u_{S_2}^{(1)}+u_{S_2}^{(2)})=-H_{|S_2|}^{-1}u_{S_2}^{(1)}.
\end{split}
\end{align}
Substituting the above terms, we can lower bound the heterogeneity gap as
\begin{align*}
    &\frac{1}{2}\left(\|\beta_{S}^{(1,S)} - \beta_{S}^{(S)}\|_{\Sigma^{(1)}_S}^2+\|\beta_{S}^{(2,S)} - \beta_{S}^{(S)}\|_{\Sigma^{(2)}_S}^2\right)\\
        &=\frac12\left(\|\beta_{S_1}^{(1,S_1)} - \beta_{S_1}^{(S_1)}\|_{\Sigma^{(1)}_{S_1}}^2+\|\beta_{S_1}^{(2,S_1)} - \beta_{S_1}^{(S_1)}\|_{\Sigma^{(2)}_{S_1}}^2\right)\\
    &\qquad\qquad\qquad+\frac12\left(\|\beta_{S_2}^{(1,S_2)} - \beta_{S_2}^{(S_2)}\|_{\Sigma^{(1)}_{S_2}}^2+\|\beta_{S_2}^{(2,S_2)} - \beta_{S_2}^{(S_2)}\|_{\Sigma^{(2)}_{S_2}}^2\right)\\
    &\ge \frac{1}{2}\left(\|\beta_{S_2}^{(1,S_2)} - \beta_{S_2}^{(S_2)}\|_{\Sigma^{(1)}_{S_2}}^2+\|\beta_{S_2}^{(2,S_2)} - \beta_{S_2}^{(S_2)}\|_{\Sigma^{(2)}_{S_2}}^2\right)\\
    &=
    \frac12\left(
    4({u_{S_2}^{(1)}})^\top H_{|S_2|}^{-1}u_{S_2}^{(1)}+
    4({u_{S_2}^{(2)}})^\top H_{|S_2|}^{-1}u_{S_2}^{(2)}
    \right)\\
    &=\frac{4|S_2|}{|S_2|+1}(k^{-1})^{2} \ge 2\cdot k^{-2}.
\end{align*}

\noindent \emph{Case 2. $S_2=\emptyset$}: In this case, we must have $\beta^{(1,S)}_{S_1} \neq \beta^{(2,S)}_{S_1}$ because $\beta_{(S_1)^c}^{(e,S)} = 0$ for any $e\in \{1,2\}$. At the same time, 
\begin{align*}
    &\frac12\left(\|\beta_{S}^{(1,S)} - \beta_{S}^{(S)}\|_{\Sigma^{(1)}_S}^2+\|\beta_{S}^{(2,S)} - \beta_{S}^{(S)}\|_{\Sigma^{(2)}_S}^2\right)\\
    &\overset{(a)}{=}\frac12\left(\|\beta_{S_1}^{(1,S_1)} - \beta_{S_1}^{(S_1)}\|_{\Sigma^{(1)}_{S_1}}^2+\|\beta_{S_1}^{(2,S_1)} - \beta_{S_1}^{(S_1)}\|_{\Sigma^{(2)}_{S_1}}^2\right)\\&\ge\frac{\lambda_{\min}\left(\Sigma^{(1)}_{[7k+1]}\right)\wedge\lambda_{\min}\left(\Sigma^{(2)}_{[7k+1]}\right)}{2}\left(\|\beta_{S_1}^{(1,S_1)} - \beta_{S_1}^{(S_1)}\|_2^2+\|\beta_{S_1}^{(2,S_1)} - \beta_{S_1}^{(S_1)}\|_2^2\right)\\
    &\overset{(b)}{\ge} \frac{24k}{2}\left(\|\beta_{S_1}^{(1,S_1)} - \beta_{S_1}^{(S_1)}\|_2^2+\|\beta_{S_1}^{(2,S_1)} - \beta_{S_1}^{(S_1)}\|_2^2\right)\\
    &\overset{(c)}{\ge}  \frac{24k}{2}\left(\frac12\|\beta_{S_1}^{(1,S_1)} - \beta_{S_1}^{(2,S_1)}\|_2^2\right) = 6k\|\beta_{S_1}^{(1,S_1)} - \beta_{S_1}^{(2,S_1)}\|_2^2.
\end{align*}

Here $(a)$ follows from the fact that $\Sigma^{(1)}$ and $\Sigma^{(2)}$ are block diagonal and $S_2=\emptyset$; $(b)$ follows from the fact that $\lambda_{\min}(\Sigma^{(1)}_{[7k+1]}),\lambda_{\min}(\Sigma^{(2)}_{[7k+1]})\ge 24k$; and $(c)$ follows from the fact that $\|a-c\|_2^2 + \|b-c\|_2^2 \ge \min_{x} \|a-x\|_2^2 + \|b-x\|_2^2 \ge \|a-(a+b)/2\|_2^2 + \|a-(a+b)/2\|_2^2 \ge 0.5\|a-b\|_2^2$ for any vector $a, b, c$. 

Recall that $\Sigma^{(2)}_{S_1}\beta_{S_1}^{(2,S)}=u_{S_1}^{(2)}$ and $\lambda_{\max}(\Sigma_{S_1}^{(2)})\le 40k$, then
\begin{align*}
    6k\|\beta_{S_1}^{(1,S_1)} - \beta_{S_1}^{(2,S_1)}\|^2
    &= 6k \left(\Sigma^{(2)}_{S_1} \beta^{(1,S)}_{S_1} - u^{(2)}_{S_1}\right)^\top (\Sigma^{(2)}_{S_1})^{-2} \left(\Sigma^{(2)}_{S_1} \beta^{(1,S)}_{S_1} - u^{(2)}_{S_1}\right)\\
    &\ge \frac{6k}{\lambda_{\max}(\Sigma^{(2)}_{S_1})^2} \left\| \Sigma^{(2)}_{S_1} \beta^{(1,S)}_{S_1} - u^{(2)}_{S_1} \right\|_2^2 \\
    &\ge \frac{6k}{(40k)^2}\left\| \Sigma^{(2)}_{S_1} \beta^{(1,S)}_{S_1} - u^{(2)}_{S_1} \right\|_2^2\\
    &=\frac{6k}{(80k^2)^2}\left\| (2\Sigma^{(2)}_{S_1}) (k\beta^{(1,S)}_{S_1}) - (2k)\cdot u^{(2)}_{S_1} \right\|_2^2.
\end{align*}
Combining $\beta^{(1,S)}_{S_1}=(k^{-1})1_{|S_1|}$ and the definition of $\Sigma^{(2)}$ and $u^{(2)}$, we obtain that each coordinate of the vector $(2\Sigma^{(2)}_{S_1}) (k\beta^{(1,S)}_{S_1}) - (2k)\cdot u^{(2)}_{S_1}$ is an integer. At the same time, we also have
\begin{align*}
(2\Sigma^{(2)}_{S_1}) (k\beta^{(1,S)}_{S_1}) - (2k)\cdot u^{(2)}_{S_1} = 2k \Sigma_{S_1}^{(2)} \left( \beta_{S_1}^{(1,S)} - \beta_{S_1}^{(2,S)}\right) \neq 0
\end{align*} because $\Sigma_{S_1}^{(2)}$ has full rank, which further yields $\|(2\Sigma^{(2)}_{S_1}) (k\beta^{(1,S)}_{S_1}) - (2k)\cdot u^{(2)}_{S_1}\|_2^2 \ge 1$. So we can conclude that
\begin{align*}
    &\frac{1}{2}\left(\|\beta_{S}^{(1,S)} - \beta_{S}^{(S)}\|_{\Sigma^{(1)}_S}^2+\|\beta_{S}^{(2,S)} - \beta_{S}^{(S)}\|_{\Sigma^{(2)}_S}^2\right) \ge \frac{6k}{(80k^2)^2}\ge \frac{1}{1280}k^{-3}
\end{align*} under Case 2. 
 Combing the above two cases together, we can conclude that
 \begin{align*}
     \frac{1}{2}\left(\|\beta^{(1,S)} - \beta^{(S)}\|_{\Sigma^{(1)}_S}^2+\|\beta^{(2,S)} - \beta^{(S)}\|_{\Sigma^{(2)}_S}^2\right)\ge k^{-3}/1280\ge d^{-\epsilon}/1280.
 \end{align*}

\noindent{\sc Step 3. Calculating the Gap between $\beta^{(S)}$ and $\beta^{(S^\dagger)}$.} Let $S^\dagger$ be arbitrary invariant set according to \cref{def:mis} and $S$ be any set that does not equal to $S^\dagger$. We keep adopting the notation $S_1=S\cap[7k+1], S_2=S\setminus [7k+1]$, and divide it into two cases.

\noindent \emph{Case 1. $S_2\neq \emptyset$}:
In this case, from the calculations above we have $\beta^{(S)}_{S_2}=-H_{|S_2|}^{-1}u_{S_2}^{(1)}$. On the other hand, $\beta^{(S^\dagger)}_{S_2}=0$ for any invariant set $S^\dagger$ according to \eqref{eq:empty-larger-than-7k}. Combing the two facts together yields
\begin{align*}
    \|\beta^{(S)}-\beta^{(S^\dagger)}\|^2_{\Sigma}\ge \|\beta^{(S_2)}_{S_2}\|^2_{H_{|S_2|}}&=({u_{S_2}^{(1)}})^\top H_{|S_2|}^{-1}u_{S_2}^{(1)} \\
    &\ge \frac{|S_2|}{|S_2|+1}(k^{-1})^{2}
    \ge \frac{1}{2}k^{-2}\ge d^{-\epsilon}/2.
\end{align*} 

\noindent \emph{Case 2. $S_2= \emptyset$}:
In this case, since $S_2=\emptyset$, one must have $S\subset [7k+1]$. On the other hand, in \eqref{eq:empty-larger-than-7k} we show that any invariant set $S^\dagger$ should also be a subset of $[7k+1]$. In this case, we claim that a stronger statement holds, that for any pair of distinct subsets $S,S'$ in $[7k+1]$, one has $\|\beta^{(S)}-\beta^{(S')}\|_{\Sigma}^2\ge d^{-\epsilon}/1280$.

Recall that in \eqref{eq:eigen-ul-bound-2} we obtain $24k\le\lambda_{\min}(\Sigma^{(2)}_{[7k+1]})\le \lambda_{\max}(\Sigma^{(2)}_{[7k+1]})\le 40k$. This implies
$28k\le\lambda_{\min}(\Sigma_{[7k+1]})\le \lambda_{\max}(\Sigma_{[7k+1]})\le 36k$. It follows from the assumption $S_2 = \emptyset$, our construction of $\Sigma$
\begin{align}
\label{eq:diff-lb-2}
\begin{split}
    \|\beta^{(S)} - \beta^{(S')}\|_{\Sigma}^2 &=\|\beta^{(S)}_{[7k+1]} - \beta^{(S')}_{[7k+1]}\|_{\Sigma_{[7k+1]}}^2\\
    &= \left(\beta_{[7k+1]}^{(S)} - \beta_{[7k+1]}^{(S')}\right)^\top \Sigma_{[7k+1]} \Sigma^{-1}_{[7k+1]} \Sigma_{[7k+1]} \left(\beta_{[7k+1]}^{(S)} - \beta_{[7k+1]}^{(S')}\right) \\
    &\ge \frac{1}{\lambda_{\max}(\Sigma_{[7k+1]})} \left\|\Sigma_{[7k+1]}\left(\beta_{[7k+1]}^{(S)} - \beta^{(S^\dagger)}_{[7k+1]}\right)\right\|_2^2 \\
    &\ge \frac{1}{36k} \left\|\Sigma_{[7k+1]}\left(\beta_{[7k+1]}^{(S)} - \beta_{[7k+1]}^{(S')}\right)\right\|_2^2.
\end{split}
\end{align}

First, we provide an upper bound for $\|\beta^{(S)}\|_2$, as $S\subseteq [7k+1]$ by assumption $S_2 =\emptyset$, 
\begin{align*}
    \|\beta^{(S)}\|_2 &= \|\beta^{(S)}_S\|_2 = \left\| (\Sigma_S)^{-1} u_S \right\|_2
    = \left\|\left(32k I_{|S|} + \frac{1}{2} \tilde{A}_S \right)^{-1} u_S \right\|_2 \\
    &= \left\|\left(I_{|S|} + \frac{1}{64k} \tilde{A}_S\right)^{-1} \frac{1}{32k} u_S \right\|_2 \\
    &\overset{(a)}{\le} \left(\left\|\left(I_{|S|} + \frac{1}{64k} \tilde{A}_S\right)^{-1} - I_{|S|} \right\|+1\right) \frac{1}{32k} \|u_S\|_2\\
    &\overset{(b)}{\le}\left(\frac{1}{32k}+1\right)\frac{1}{32k}\|u_{[7k+1]}\|_2\\
    &\le \left(\frac{1}{32k}+1\right) \frac{1}{32k} k^{-1} \sqrt{7k+1}(32+1/4)k \\
    &\le 3k^{-1/2}.
\end{align*}
Here $(a)$ follows from the triangle inequality, $(b)$ follows from the fact that $\|(I+M)^{-1} -I\|_2 \le 2\|M\|$ if $\|M\| \le 0.5$. Hence $\|\beta^{(S)}\|_2\le 3k^{-1/2}$ for any $ S\subset [7k+1]$. Similarly $\|\beta^{(S')}\|_2\le 3k^{-1/2}$.

Since $S\neq S'$, there exists some $j\in[7k+1]$ such that $j\in (S \setminus S') \lor (S' \setminus S)$. Without loss of generality, we assume $j\in S'\setminus S$.
Then it follows from the above upper bound, the fact $j\in S'\setminus S$ and Cauchy Schwarz inequality that
\begin{align*}
    \Delta_j &= \Sigma_{j, S}^\top \beta_S^{(S)} - u_j \\
    &\le \|\tilde{A}_{j, S}\|_2 \|\beta_S^{(S)}\|_2 - (k^{-1})\frac{1}{2}(32k) \\
    &\le (8k)^{1/2}(3k^{-1/2}) - 32/2 \le -1.
\end{align*} This further yields that $\|\Delta\|_2^2 \ge \Delta_j^2 \ge 1$. Combining \eqref{eq:diff-lb-2}, we have $\|\beta^{(S)} - \beta^{(S')}\|_{\Sigma}^2\ge\frac{1}{36k}\ge d^{-\epsilon}/36$. Combining Case~1 and Case~2, we complete the lower bound for the gap between $\beta^{(S)}$ and $\beta^{(S^\dagger)}$.

\end{proof}

\subsection{Proof of \cref{coro:est-d-eps}}

We use the same reduction as in \cref{thm:gap2}. For any $\epsilon>0$ and {\sc 3Sat} instance $x$, we let $y=T_\epsilon(x)$ be the constructed \probsep{} instance in \cref{thm:gap2}. Let $\hat{\beta}$ be the output required by \cref{prob:est-d-eps} in the instance $y$, and $\tilde{S} = \{j\in[7k+1]: \hat{\beta}_j \ge k^{-1}/2\}$. 
Following the notations therein, we claim that
\begin{align}
\label{eq:red2}
    \tilde{S} \in \mathcal{S}_y \qquad \overset{(a)}{\Longleftrightarrow} \qquad |\mathcal{S}_y| \ge 1 \qquad {\Longleftrightarrow} \qquad x\in \mathcal{X}_{\text{{\sc 3Sat}}, 1}
\end{align} Therefore, if an algorithm $\mathsf{A}$ can take \cref{prob:est-d-eps} instance $y$ as input and return the desired output $\hat{\beta}(y)$ within time $O(p(|y|))$ for some polynomial $p$, then the following algorithm can solve {\sc 3Sat} within polynomial time: for any instance $x$, it first transforms $x$ into $y=T_\epsilon(x)$, then use algorithm $\mathsf{A}$ to solve $y$ and gets the returned $\hat{\beta}$, and finally output $\indicator\{\tilde{S} \in \mathcal{S}_y\}$. 

It remains to verify $(a)$: the $\Rightarrow$ direction is obvious. For the $\Leftarrow$ direction, suppose $|\mathcal{S}_y| \ge 1$, the estimation error guarantee in \cref{prob:est-d-eps} indicates that
\begin{align*}
    \|\hat{\beta}_{[7k+1]} - \beta_{[7k+1]}^{(S^\dagger)}\|_\infty \le \|\hat{\beta} - \beta^{(S^\dagger)}\|_2 \le \sqrt{\frac{\|\hat{\beta} - \beta^{(S^\dagger)}\|_{\Sigma}^2}{\lambda_{\min}(\Sigma)}} \overset{(a)}{<} \sqrt{{0.25d^{-\epsilon}}} \le \frac{1}{2}k^{-1}
\end{align*} for some $S^\dagger \in \mathcal{S}_y$. Here $(a)$ follows from the the error guarantee in \cref{prob:est-d-eps}, and the fact $\lambda_{\min}(\Sigma) \ge 1$ derived in the proof of \cref{thm:gap2}. This further indicates $\tilde{S} = S^\dagger$ by the fact that $S^\dagger\subset[7k+1]$ and  $\beta^{(S^\dagger)}_j = (k^{-1})\indicator\{j\in S^\dagger\}$ for any $j\in [7k+1]$ derived in the proof of \cref{thm:gap2}. \qed

\subsection{Proof of \cref{thm:sparse-construction}}

\noindent {\sc Step 1. Sparse Reduction For 3SAT Problem.} We first show that there exists a parsimonious polynomial-time reduction $T$ from {\sc 3Sat} problem to the {\sc 3Sat} problem where in each instance all boolean variables appear no more than $15$ times.

To be specific, given a {\sc 3Sat} instance $x$ with $k$ clauses and $n$ boolean variables $\{v_m\}_{m=1}^n$ where obviously $n\le 3k$, we construct the new instance $x'=T(x)$ as follows, we first introduce $n\times k$ boolean variables $\{w_{m,i}\}_{m\in[n],i\in[k]}$. For each $i\in [k]$ and $m\in[n]$, if boolean variable $v_m$ appears in clause $i$ of the original instance $x$, we replace the variable $v_m$ with $w_{m,i}$. Then all the original variables $\{v_m\}_{m\in[n]}$ are completely replaced, and each variable in $\{w_{m,i}\}_{m\in[n],i\in[k]}$ appears no more than $3$ times.

Secondly, we need to add the following $n\times(k-1)$ additional constraints
\begin{align}
\label{eq:reduction-additional-constraints}
    w_{m,1}=w_{m,2},\quad
    w_{m,2}=w_{m,3},\quad 
    \cdots\quad
    w_{m,k-1}=w_{m,k},\quad \forall m\in [n].
\end{align}
Note that a constraint $ w=w'$ is equivalent to
\begin{align}
    \begin{split}
        (\neg w \lor \neg w' \lor  w^\circ) &\land (w \lor \neg w' \lor w^\circ) \land (\neg w \lor w' \lor w^\circ) \land (w \lor w' \lor  w^\circ)\\ 
    &\land (\neg w \lor w' \lor \neg w^\circ) \land (w \lor \neg w' \lor \neg w^\circ)
    \end{split}
    \label{eq:=-to-3sat}
\end{align}
with an additionally introduced boolean variable $w^\circ$ that is forced to be $\mathrm{True}$ by the first four clauses in \eqref{eq:=-to-3sat}. Hence the constraints \eqref{eq:reduction-additional-constraints} can be translated into $6n(k-1)$ clauses, with additionally introduced $n(k-1)$ variables $\{w_{\ell}^\circ\}_{\ell=1}^{n(k-1)}$. Finally, in instance $x'$ there are $k'=k+6n(k-1)<18k^2$ clauses in total. Each boolean variable in $\{w_{m,i}\}_{m\in[n],i\in[k]}$ appears no more than $3+2\times 6=15$ times, and each additionally introduced boolean variable in $\{w_{\ell}^\circ\}_{\ell=1}^{n(k-1)}$ appears no more than $6$ times.

Now we prove that the mapping $T$ we construct is a parsimonious polynomial-time reduction, namely, for any valid solution $v\in \mathcal{S}_x$, setting $w_{m,i}=v_m$ for $m\in[n],i\in[k]$ and $w_\ell^\circ=\mathrm{True}$ for $\ell\in[n(k-1)]$ leads to a valid solution $w\in\mathcal{S}_{x'}$, and such mapping from $\mathcal{S}_x$ to $\mathcal{S}_{x'}$ is a bijection.

The verification of injection is obvious. Now we prove it is a surjection. For any valid solution $w$ of instance $x'$, the constraints \eqref{eq:reduction-additional-constraints} require $w_{m,1}=\cdots=w_{m,k}$ for $m\in[n]$. Hence setting $v_m=w_{m,1}$ for $m\in[n]$ leads to a valid solution $v\in\mathcal{S}_x$ whose image is $w$. This completes to proof for the bijection.

\noindent {\sc Step 2. Construction of ExistLIS-Ident Problem.} Next, we construct the $7k'\times 7k'$ matrix $A$ that corresponds to the {\sc 3Sat} instance $x'$, as shown in \eqref{eq:reduction-a-mat}. Namely,
\begin{align*}
    A_{7(i-1)+t,7(i'-1)+t'} = \begin{cases}
        \indicator\{t ~\text{contradicts itself}\} & \qquad i=i' \text{and } t=t' \\
        1 & \qquad i=i' \text{and } t\neq t' \\
        1 & \qquad i\neq i' \text{and } t \text{contradicts } t' \\
        0 & \qquad \text{otherwise}
    \end{cases}
\end{align*} for any $i,i'\in [k']$ and $t,t' \in [7]$. We define a $k'\times k'$ symmetric matrix $B$ as follows:
\begin{align}
\label{eq:reduction-b-mat}
    B_{i,i'} = \begin{cases}
        1 & \qquad |i-i'|=1 \\
        0 & \qquad \text{otherwise}
    \end{cases}
\end{align}
for any $i,i'\in[k']$. Matrix $B$ can be seen as the adjacency matrix of a connected graph over $k'$ vertices. We define matrix $K\in\mathbb{R}^{k'\times 7k'}$ as follows
\begin{align}
\label{eq:reduction-c-mat}
    K_{i,j} = \begin{cases}
        1 & \qquad 7(i-1)< j\le 7i \\
        0 & \qquad \text{otherwise}
    \end{cases}
\end{align}
for any $i\in[k'], j\in[7k']$. We construct its corresponding {\sc ExistLIS} instance $y$ with $|y|=8k'$ as follows:
\begin{align*}
    \Sigma^{(1)} = I_{8k'} \qquad \text{and} \qquad u^{(1)} = 1_{8k'},
\end{align*}
and
\begin{align*}
    \Sigma^{(2)} = \begin{bmatrix}
        1000 I_{7k'} + A & \frac{1}{2} K^\top \\
        \frac{1}{2}K & 1000I_{k'}+\frac{1}{8}B\\
    \end{bmatrix} \qquad \text{and} \qquad u^{(2)} = \begin{bmatrix}
        (1000 + \frac{1}{2}) \cdot 1_{7k'}\\
        (1000+\frac{1}{2})1_{k'}+\frac{1}{8}B1_{k'}
    \end{bmatrix}.
\end{align*}
One can easily verify both $\Sigma^{(1)}$ and $\Sigma^{(2)}$ are positive definite from the fact that $\Sigma^{(2)}$ is diagonally dominant, and $H_{\ell}$ is positive definite for any $\ell\ge 1$. Note that $A_{7(i-1)+s,7(j-1)+t}\neq 0$ immediately implies the $i$-th clause and the $i'$-th clause have shared variable. Since each variable appears no more than $15$ times, one clause shares common variables with up to $3\times 15$ other clauses. Then we can conclude that each row of matrix $A$ has no more than $7\times(3\times 15+1)=322$ non-zero elements. Combining with the fact that there are no more than 2 non-zero elements in each row of $B$ and no more than 7 non-zero elements in each row/column of $K$, we can conclude that for any $e\in\mathcal{E}$, each row of matrix $\Sigma^{(e)}$ has no more than $322+7+2+1<400$ non-zero elements.

Similar to \eqref{eq:equiv-reduction} in the proof of \cref{lemma:reduction}, we claim the following and defer the proof to the end of this step.
\begin{align}
 \label{eq:append-sparse-equiv-reduction}
 \begin{split}
     S^\dagger \in \mathcal{S}_y ~~~&\overset{(a)}{\Longleftrightarrow} ~~~  \emptyset\neq S^\dagger\subset [8k']~\text{and }~\beta^{(2,S^\dagger)} = \beta^{(1,S^\dagger)}~\text{with}~ \beta^{(1,S^\dagger)}_j = \indicator\{j \in S^\dagger\} \\
     &\overset{(b)}{\Longleftrightarrow} ~~~  S^\dagger = \mathring{S} \cup \{7k'+1,\ldots,8k'\} ~\text{with}~|\mathring{S}\cap\{7i-6,\ldots, 7i\}|=1,~  \forall 1\le i\le k' \\
     &~~~~~~~~~~~~\text{and}~A_{j,j'} = 0~~\forall j,j'\in \mathring{S}\subseteq[7k']\\
     &\overset{(c)}{\Longleftrightarrow} ~~~ \mathring{S} = \{7(i-1)+a_i\}_{i=1}^{k'}~\text{with}~a_i \in [7]~\text{s.t. adopt action ID}~ a_i \\
     &~~~~~~~~~~~~ \text{in clause } i\in [k']~\text{will lead to a valid solution}~v\in \mathcal{S}_{x'}\text{.} 
\end{split}
\end{align}
Combining \eqref{eq:append-sparse-equiv-reduction} and {\sc Step 1}, we have $|\mathcal{S}_x|=|\mathcal{S}_{x'}|=|\mathcal{S}_y|$. Since $d=8k'=\mathrm{poly}(k)$ and such construction can be done in polynomial time, this mapping admits a deterministic polynomial-time reduction from {\sc 3Sat} to the problem we construct. Therefore, we can conclude that the problem we construct is NP-hard.

Proof of \eqref{eq:append-sparse-equiv-reduction}(a) is essentially identical to the proof of \eqref{eq:equiv-reduction}(a) in \cref{lemma:reduction}. Now we prove \eqref{eq:append-sparse-equiv-reduction}(b) and (c).

\noindent \underline{\sc Proof of \eqref{eq:append-sparse-equiv-reduction}(b)} The proof idea is similar to \eqref{eq:equiv-reduction}(b). The direction $\Leftarrow$ is obvious. For the $\Rightarrow$ direction, we first assert that 
\begin{align}
    \label{eq:sparse-construction-nonempty}
    \{7k'+1,\ldots,8k'\}\cap S^\dagger\neq \emptyset
\end{align}
We use the proof by contradiction argument. If $\{7k'+1,\ldots,8k'\}\cap S^\dagger= \emptyset$, there must exist an index  $j\in[7k']\cap S^\dagger$ since $S^\dagger$ is nonempty. Combined with the fact $\beta_j^{(2,S^\dagger)}=\beta_j^{(1,S^\dagger)}=\indicator\{j \in S^\dagger\}$, the equation $ \Sigma^{(2)}_{j,S^\dagger}\beta^{(2,S^\dagger)}_{S^\dagger}=u^{(2)}_j$ tells
\begin{align*}
    1000+\sum_{j'=1}^{7k'} A_{j,j'}\beta_{j'}^{(2,S^\dagger)}=1000+\frac{1}{2}
\end{align*}
The L.H.S. is an integer while the R.H.S. is not an integer. This leads to a contradiction. This proves \eqref{eq:sparse-construction-nonempty}.

Now we consider the element $i+7k'\in\{7k'+1,\ldots,8k'\}\cap S^\dagger$. Then the equation $ \Sigma^{(2)}_{S^\dagger}\beta^{(2,S^\dagger)}_{S^\dagger}=u^{(2)}_{S^\dagger}$ tells
\begin{align*}
    \frac{1}{2}\sum_{7i-6<j\le 7i} \beta_j^{(2,S^\dagger)} +\frac{1}{8}\sum_{i':B_{i,i'}=1} \beta_{i'+7k'}^{(2,S^\dagger)} +1000=1000+\frac{1}{2}+\frac{1}{8}\sum_{i':B_{i,i'}=1}1.
\end{align*}
Since $\beta_j^{(2,S^\dagger)}=\beta_j^{(1,S^\dagger)}=\indicator\{j \in S^\dagger\}$,
then $\sum_{i':B_{i,i'}=1} \beta_{i'+7k'}^{(2,S^\dagger)}$ can only take values $0,1$ or $2$. Through taking both sides of the equation modulo $1/2$ we can then obtain
\begin{align*}
    \sum_{i':B_{i,i'}=1} \beta_{i'+7k'}^{(2,S^\dagger)} =\sum_{i':B_{i,i'}=1}1.
\end{align*}
This indicates that all the neighbors of $i$ (with respect to the adjacency matrix $B$) should be simultaneously contained in $S^\dagger$. Since $B$ represents the adjacency matrix of a connected graph, we can then inductively prove that $\{7k'+1,\ldots,8k'\}\subset S^\dagger$. Given this, the equation $\Sigma_S^{(2)}\beta_S^{(2,S^\dagger)}=u^{(2)}_{S^\dagger}$ now becomes
 \begin{align*}
 \frac{1}{2}K 1_{\mathring{S}}=\frac{1}{2}1_{\mathring{S}}\qquad{\Longrightarrow}\qquad |\mathring{S}\cap\{7i-6,\ldots, 7i\}|=1,~\text{for}~\forall 1\le i\le k'
 \end{align*}
 and
 \begin{align*}
    A_{\mathring{S}} 1_{k'} = 0 \qquad\overset{(i)}{\Longrightarrow}\qquad A_{j',j}=0,~\text{for}~\forall j',j\in \mathring{S}
\end{align*} where $(i)$ follows from the fact that $A\in \{0,1\}^{7k'\times 7k'}$. 

\noindent \underline{\sc Proof of \eqref{eq:append-sparse-equiv-reduction}(c)} The direction $\Rightarrow$ follows from the proof of \eqref{eq:equiv-reduction}(c). For the direction $\Leftarrow$, it follows from the proof of \eqref{eq:equiv-reduction}(c) and the fact that $\mathring{S} = \{7(i-1)+a_i\}_{i=1}^{k'}$ with $a_i \in [7]$ naturally implies $|\mathring{S}\cap\{7i-6,\ldots, 7i\}|=1$ for $i\in[k']$.
\qed

\subsection{Proof of \cref{thm:cic}}

It suffices to construct a polynomial-time reduction from {\sc 3Sat} to \cref{prob:cic}. Let $x$ be any {\sc 3Sat} instance with input size $k$, following the notation in \cref{lemma:reduction}, we let $y=T(x)$ be an instance of \cref{prob:cic} that $d=7k$, $\Sigma^{(1)} = 5d I_d$, $u^{(1)} = 5d 1_d$, and $\Sigma^{(2)} = 5dI_d + A$, $u^{(2)} = 5d 1_d$. Now the constraint $u^{(1)}_S = u^{(2)}_S$ trivially holds for any $S\subseteq [d]$. We claim that
\begin{align*}
    S\in \mathcal{S}_y \qquad &\Longleftrightarrow \qquad |S|=k ~~\text{and} ~~ A_{j,j'}=0, ~~ \forall j,j'\in S \\
    \qquad &\Longleftrightarrow \qquad \mathring{S} = \{7(i-1)+a_i\}_{i=1}^k\text{ with }a_i \in [7]\text{ s.t. adopting action ID } a_i \\
     &~~~~~~~~~~~~ \text{ in clause } i\in [k] \text{ will lead to a valid solution }v\in \mathcal{S}_x\text{.} 
\end{align*} The proof of equivalence is identical to that in \cref{lemma:reduction}. This completes the proof.

\section{Proofs for the Population-level Results}
\label{sec:proof:population}

\subsection{Proof of \cref{prop:minimax-k1}}
\label{proof:minimax-k1}

Applying \cref{thm:minimax-k} with $k=1$ completes the proof of \eqref{eq:minimax-k1}. To establish the causal identification result, it suffices to verify \cref{cond:restricted-invariance} with $k=1$. 

To see this, under \eqref{model:lip} and \eqref{ident:lip}, if \cref{cond:ortho} further holds, we have, for each $j \in S^\star$ and $e\in \mathcal{E}$,
\begin{align*}  
    \beta^{(e, \{j\})} = \frac{\mathbb{E}[X_j^{(e)} Y^{(e)}]}{\mathbb{E}[X_j^{(e)} X^{(e)}_j]}
    = \frac{\mathbb{E}[X_j^{(e)} (\sum_{i\in S^\star} X_i^{(e)} \beta_i^\star) + \varepsilon^{(e)})]}{\mathbb{E}[X_j^{(e)} X^{(e)}_j]} \overset{(a)}{=} \beta^\star_j
\end{align*} where (a) follows from
\begin{align*}
    \forall i,j\in S^\star ~\text{with}~ i\neq j, ~~\mathbb{E}[X_i^{(e)} X_j^{(e)}] = 0 \qquad \text{and} \qquad \mathbb{E}[X_j^{(e)} \varepsilon^{(e)}] = 0,
\end{align*} provided \cref{cond:ortho} and \eqref{model:lip}, respectively. This completes the proof. \qed

\subsection{Proof of \eqref{eq:geo}}
\label{proof:geo}

Denote $q=(w_1(1),\ldots, w_1(d)) \in \mathbb{R}^d$, it follows from \cref{prop:minimax-k1} that
\begin{align*}
\beta^\gamma &= \argmin_{\beta} \max_{\mu \in \mathcal{P}_\gamma(\Sigma, u)} \left\{\mathbb{E}_{\mu}[|Y - \beta^\top X|^2] - \mathbb{E}_\mu[Y^2] \right\} \\
&= \argmin_{\beta} \max_{\mu \in \mathcal{P}_\gamma(\Sigma, u)} \left\{\beta^\top \mathbb{E}_{\mu} [XX^\top] \beta - 2\beta^\top \mathbb{E}[X Y]\right\} \\ 
&= \argmin_{\beta} \max_{\tilde{u}: |\tilde{u} - u| \le \gamma \cdot q} \left\{\beta^\top \Sigma \beta - 2\beta^\top \tilde{u} \right\} \\
&= \argmin_{\beta} \max_{\tilde{\beta} \in \Theta_\gamma} \left\{\beta^\top \Sigma \beta - 2\beta^\top \Sigma \tilde{\beta} \right\}.
\end{align*}
It is easy to check that the convex hull of $\Theta_\gamma$ is itself, applying Theorem 1 of \cite{meinshausen2015maximin} completes the proof.

\subsection{Proof of \cref{thm:minimax-k}}
\label{proof:minimax-k}
\noindent {\sc Proof of \eqref{eq:minimax-k}.} The existence and uniqueness of optimal solution follows from \cref{prop:sc}. We will show that
\begin{align*}
    \mathsf{Q}_{k,\gamma}(\beta) = \sup_{\mu \in \mathcal{P}_{k,\gamma}(\Sigma, u)} \mathbb{E}_{\mu} \left[ |Y - \beta^\top X|^2 - |Y|^2\right].
\end{align*}
For given fixed $\mu \in \mathcal{P}_{k,\gamma}$, one has
\begin{align*}
    \mathbb{E}_{\mu} \left[ |Y - \beta^\top X|^2 - |Y|^2\right] &= \beta^\top \mathbb{E}_{\mu} \left[XX^\top\right] \beta - 2\beta^\top \mathbb{E}_{\mu}[XY] \\
    &= \beta^\top \Sigma \beta - 2\beta^\top \mathbb{E}_{\mu}[XY].
\end{align*} On the other hand, it follows from the definition of $\Sigma$, ${u}$ and $\mathsf{Q}_{k,\gamma}(\beta)$ that
\begin{align*}
    \mathsf{Q}_{k,\gamma}(\beta) = \frac{1}{2} \beta^\top \Sigma \beta - \beta^\top {u} + \gamma v^\top |\beta| \qquad \text{with} \qquad v=(w_k(1),\ldots, w_k(d)).
\end{align*}
Now it suffices to show that for any $\beta \in \mathbb{R}^d$,
\begin{align}
\label{eq:quad-equiv}
    \frac{1}{2} \beta^\top \Sigma \beta - \beta^\top {u} + \gamma v^\top |\beta| = \sup_{\tilde{u}: |\tilde{u}-{u}| \le v} \frac{1}{2} \beta^\top \Sigma \beta - \gamma \beta^\top \tilde{u}.
\end{align} To see this, it is easy to verify that, for any given $x, a\in \mathbb{R}$ and $b \in \mathbb{R}^+$, one has
\begin{align*}
    \sup_{y\in [a - b, a + b]} -x y = -ax + b|x|,
\end{align*} then we can obtain
\begin{align*}
    \sup_{\tilde{u}: |\tilde{u}-{u}| \le v} \frac{1}{2} \beta^\top \Sigma \beta - \beta^\top \tilde{u} = \frac{1}{2} \beta^\top \Sigma \beta + \sum_{j=1}^d \sup_{\tilde{u}_j \in [{u}_j - \gamma v_j, {u}_j + \gamma v_j]} \left(-\beta_j \tilde{u}_j \right) = \frac{1}{2} \beta^\top \Sigma \beta - \sum_{j=1}^d {u}_j \beta_j + |\beta_j| v_j \gamma,
\end{align*} this verifies \eqref{eq:quad-equiv} and thus completes the proofs of the claim \eqref{eq:minimax-k}. \qed

\subsection{Proof of \cref{thm:causal-ident}}

\noindent {\sc Proof of the Causal Identification Result.} It follows from \cref{cond:restricted-invariance} and the definition of $w_k(j)$ that $w_k(j) = 0$ for any $j\in S^\star$. It also follows from \eqref{ident:lip} that
\begin{align*}
    w_k(j) \neq 0 \qquad \forall j\in [d]~ \text{with} ~\sum_{e\in \mathcal{E}}\mathbb{E}[X_j^{(e)} \varepsilon^{(e)}] \neq 0.  
\end{align*}
Therefore, for any $\beta \in \mathbb{R}^d$,
\begin{align*}
    \mathsf{Q}_{k,\gamma}(\beta) - \mathsf{Q}_{k,\gamma}(\beta^\star) &= \frac{1}{2|\mathcal{E}|} \sum_{e\in \mathcal{E}} \mathbb{E} \left[ |Y^{(e)} - \beta^\top X^{(e)}|^2 - |Y^{(e)} - (\beta^\star)^\top X^{(e)}|^2\right] + \sum_{j=1}^d \gamma |\beta_j| w_k(j)\\
    &= \frac{1}{2}(\beta - \beta^\star)^\top \Sigma (\beta - \beta^\star) - (\beta - \beta^\star) \Sigma (\bar{\beta} - \beta^\star) + \gamma \sum_{j=1}^d |\beta_j| w_k(j) \\
    &\overset{(a)}{\ge} \sum_{j\in G} |\beta_j| \left\{\gamma \cdot w_k(j) \right\} - \beta_j \left(\frac{1}{|\mathcal{E}|} \sum_{e\in \mathcal{E}} \mathbb{E}[X_j^{(e)} \varepsilon^{(e)}] \right).
\end{align*}
where $G$ is defined in \eqref{eq:set-g} and $\bar{\beta}=\Sigma^{-1} u$. Here (a) follows from the fact that the first quadratic term is non-negative, and the identity
\begin{align*}
    \Sigma (\bar{\beta} - \beta^\star) = \frac{1}{|\mathcal{E}|} \sum_{e\in \mathcal{E}} \left\{\mathbb{E}[X^{(e)} Y^{(e)}] - \mathbb{E}[X^{(e)} (X^{(e)})^\top \beta^\star] \right\} = \frac{1}{|\mathcal{E}|} \sum_{e\in \mathcal{E}} \mathbb{E}[X^{(e)} \varepsilon^{(e)}].
\end{align*} Therefore, we have
\begin{align*}
    \mathsf{Q}_{k,\gamma}(\beta) - \mathsf{Q}_{k,\gamma}(\beta^\star) \ge 0 \qquad \text{if} \qquad \gamma \ge \max_{j\in G} \frac{\left|\frac{1}{|\mathcal{E}|} \sum_{e\in \mathcal{E}} \mathbb{E}[\varepsilon^{(e)} X_j^{(e)}]\right|}{w_k(j)} := \gamma_k^\star,
\end{align*} this completes the proof. 

We finally establish the upper bound on $\gamma_k^\star$. It follows from the definition of $w_k$ that
\begin{align*}
    (\gamma_k^\star)^2 &= \max_{j\in G} \frac{\left|\frac{1}{|\mathcal{E}|} \sum_{e\in \mathcal{E}} \mathbb{E}[\varepsilon^{(e)} X_j^{(e)}]\right|^2}{\{w_k(j)\}^2} \\
    &= \max_{j\in G} \frac{\left|\frac{1}{|\mathcal{E}|} \sum_{e\in \mathcal{E}} \mathbb{E}[\varepsilon^{(e)} X_j^{(e)}]\right|^2}{\inf_{S: j\in S} \frac{1}{|\mathcal{E}|} \sum_{e\in \mathcal{E}} \|\beta^{(e,S)}_S - \beta^{(S)}_S\|_{\Sigma_S^{(e)}}^2} \\
    &\le \max_{j \in G} \max_{S: j\in S} \frac{\left\|\frac{1}{|\mathcal{E}|} \sum_{e\in \mathcal{E}} \mathbb{E}[\varepsilon^{(e)} X_S^{(e)}]\right\|_2^2}{\frac{1}{|\mathcal{E}|} \sum_{e\in \mathcal{E}} \|\beta^{(e,S)}_S - \beta^{(S)}_S\|_{\Sigma_S^{(e)}}^2} \\
    &= \max_{S: S\cap G\neq \emptyset} \frac{\left\|\frac{1}{|\mathcal{E}|} \sum_{e\in \mathcal{E}} \mathbb{E}[\varepsilon^{(e)} X_S^{(e)}]\right\|_2^2}{\frac{1}{|\mathcal{E}|} \sum_{e\in \mathcal{E}} \|\beta^{(e,S)}_S - \beta^{(S)}_S\|_{\Sigma_S^{(e)}}^2}
\end{align*}

Let $\kappa_{\min} = \min_{e\in \mathcal{E}} \lambda_{\min}(\Sigma^{(e)})$, we have
\begin{align*}
    \frac{1}{|\mathcal{E}|} \sum_{e\in \mathcal{E}} \|\beta^{(e,S)}_S - \beta^{(S)}_S\|_{\Sigma_S^{(e)}}^2 &\ge \kappa_{\min} \frac{1}{|\mathcal{E}|} \sum_{e\in \mathcal{E}} \|\beta^{(e,S)}_S - {\beta}^{(S)}_S\|_2^2 \\
    &\ge \kappa_{\min} \inf_{\beta: \beta_{S^c} = 0} \frac{1}{|\mathcal{E}|} \sum_{e\in \mathcal{E}} \|\beta^{(e,S)} - \beta\|_2^2 \ge \kappa_{\min}  \frac{1}{|\mathcal{E}|} \sum_{e\in \mathcal{E}} \|\beta^{(e,S)} - \bar{\beta}^{(S)}\|_2^2
\end{align*} with $\bar{\beta}^{(S)} = \frac{1}{|\mathcal{E}|} \sum_{e\in \mathcal{E}} \beta^{(e,S)}$. Plugging it back into the upper bounded on $(\gamma_k^\star)^2$, we conclude that
\begin{align*}
 (\gamma_k^\star)^2 \le (\kappa_{\min})^{-1} \max_{S: S\cap G\neq \emptyset} \frac{\left\|\frac{1}{|\mathcal{E}|} \sum_{e\in \mathcal{E}} \mathbb{E}[\varepsilon^{(e)} X_S^{(e)}]\right\|_2^2}{\frac{1}{|\mathcal{E}|} \sum_{e\in \mathcal{E}} \|\beta^{(e,S)} - \beta^{(S)}\|_{2}^2} = \gamma^* \kappa_{\min}^2,
\end{align*} where
\begin{align}
\label{eq:gamma-star}
    \gamma^* = (\kappa_{\min})^{-3} \max_{S: S\cap G\neq \emptyset} \frac{\left\|\frac{1}{|\mathcal{E}|} \sum_{e\in \mathcal{E}} \mathbb{E}[\varepsilon^{(e)} X_S^{(e)}]\right\|_2^2}{\frac{1}{|\mathcal{E}|} \sum_{e\in \mathcal{E}} \|\beta^{(e,S)} - \beta^{(S)}\|_{2}^2}
\end{align} is the quantity defined in (4.5) of \cite{fan2023environment}.

\qed
\section{Proofs for Non-asymptotic Results}
\label{sec:proof:non-asymptotic}

\subsection{Preliminaries}

We first introduce some notations. Recall the definition of $(\Sigma, u)$ in \eqref{eq:covariance}, we denote their empirical counterparts as
\begin{align}
    \hat{\Sigma} = \frac{1}{n \cdot |\mathcal{E}|} \sum_{i\in [n], e\in \mathcal{E}} X_i^{(e)} (X_i^{(e)})^\top \qquad \text{and} \qquad \hat{u} = \frac{1}{n\cdot |\mathcal{E}|} \sum_{i\in [n], e\in \mathcal{E}} X_i^{(e)} Y_i^{(e)}.
\end{align}
We define
\begin{align*}
    \hat{\mathsf{R}}(\beta) &= \frac{1}{2|\mathcal{E}|} \sum_{e\in \mathcal{E}} \hat{\mathbb{E}}\left[|Y_i^{(e)} - \beta^\top X_i^{(e)}|^2 \right] = \frac{1}{2} \beta^\top \hat{\Sigma} \beta - \beta^\top \hat{u} + \frac{1}{2} \hat{u}^\top \hat{u}, \\
    \mathsf{R}(\beta) &= \frac{1}{2|\mathcal{E}|} \sum_{e\in \mathcal{E}} {\mathbb{E}}\left[|Y_i^{(e)} - \beta^\top X_i^{(e)}|^2 \right] = \frac{1}{2} \beta^\top {\Sigma} \beta - \beta^\top u + \frac{1}{2} u^\top u.
\end{align*}
We let
\begin{align*}
    v(S) = \frac{1}{|\mathcal{E}|} \sum_{e\in \mathcal{E}} \left\| \beta^{(e,S)}_S - \beta^{(S)}_S \right\|_{\Sigma_S^{(e)}}^2 \qquad \text{and} \qquad \hat{v}(S) = \frac{1}{|\mathcal{E}|} \sum_{e\in \mathcal{E}} \left\| \hat{\beta}^{(e,S)}_S - \hat{\beta}^{(S)}_S \right\|_{\hat{\Sigma}_S^{(e)}}^2.
\end{align*} One can expect $|v(S) - \hat{v}(S)| \asymp (|S|/n)^{1/2}$ by CLT. However, applying such a crude bound will result in a slower rate. Instead, the next proposition targets to establish a shaper instance-dependent error bound for the difference. We define
\begin{align}
\label{eq:rho1}
    \rho(s, t) = \frac{(s\log (4d/s)) + \log (|\mathcal{E}|) + t}{n} \qquad \text{and} \qquad \zeta(s, t) = \frac{(s\log (4d/s)) + t}{n \cdot |\mathcal{E}|}
\end{align} with $s\in [d]$ and $t>0$. 

We also define some concepts that will be used throughout the proof.
\begin{definition}[Sub-Gaussian Random Variable]\label{def:sub-gaussian} A random variable $X$ is a sub-Gaussian random variable with parameter $\sigma\in\mathbb{R}^{+}$ if
\begin{align*}
    \forall\lambda\in\mathbb{R},\qquad \mathbb{E}[\exp\left(\lambda (X-\mathbb{E}[X])\right)]\le\exp\left(\frac{\lambda}{2}\sigma^2\right).
\end{align*}
\end{definition}

\begin{definition}[Sub-exponential Random Variable]\label{def:sub-exponential} A random variable $X$ is a sub-exponential random variable with parameter $(\nu,\alpha)\in\mathbb{R}^{+}\times\mathbb{R}^{+}$ if
\begin{align*}
    \forall|\lambda|<1/\alpha,\qquad \mathbb{E}[\exp\left(\lambda (X-\mathbb{E}[X])\right)]\le\exp\left(\frac{\lambda}{2}\nu^2\right).
\end{align*}
\end{definition}
It is easy to verify that the product of two sub-Gaussian random variables is a sub-exponential random variable, and the dependence of the parameters can be written as follows.
\begin{lemma}[Product of Two Sub-Gaussian Random Variables]\label{lemma:product_of_two_subGaussian}
Suppose $X_1$ and $X_2$ are two zero-mean
sub-Gaussian random variables with parameters $\sigma_1$ and $\sigma_2$, respectively. Then $X_1 X_2$ is a sub-exponential
random variable with parameter $(C\sigma_1\sigma_2, C\sigma_1\sigma_2)$, where $C>0$ is some universal constant.
\end{lemma}
We also have the following lemma stating the concentration inequality for the sum of independent sub-exponential random variables.
\begin{lemma}[Sum of Independent Sub-exponential Random Variables]\label{lemma:Bernstein}
Suppose $X_1,\ldots X_N$ are independent sub-exponential random variables with parameters $\{(\nu_i,\alpha_i)\}_{i=1}^N$, respectively. There exists some universal constant $C>0$ such that the following holds,
\begin{align*}
    \mathbb{P}\left[\left|\sum_{i=1}^N(X_i-\mathbb{E}[X_i])\right|\ge C \left\{\sqrt{t\times\sum_{i=1}^{N}\nu_i^2}+t\times\max_{i\in[N]}\alpha_i\right\}\right]\le 2e^{-t}.
\end{align*}
\end{lemma}
The next proposition provides upper bounds for $|v(S)-\hat{v}(S)|$.
\begin{proposition}[Instance-dependent Error Bounds on $|v(S)-\hat{v}(S)|$]
\label{prop:instance-dependent-v2}
Suppose \cref{cond:regularity} hold. There exists some universal constant $C$ such that, for any $t>0$ and $\epsilon>0$, if $C \sigma_x^4 \rho(k, t) \le 1$, then the following event
\begin{align*}
    \forall S\subseteq [d], |S| \le k \qquad | v(S) - \hat{v}(S) | \le C\left\{\sqrt{v(S)\cdot \sigma_x^4 \sigma_y^2 b\rho(k, t)} + \sigma_x^4 \sigma_y^2 b \rho(k, t) \right\}
\end{align*} occurs with probability at least $1-e^{-t}$.
\end{proposition}

The above inequality is instance-dependent in that both L.H.S. and R.H.S. of the inequality contain $v(S)$ dependent on $S$. The next proposition claims that one can establish strong convexity around $\beta^{k,\gamma}$.

\begin{proposition}
    \label{prop:sc}
    Under \cref{cond:normalize}, for any $k\in [d]$ and $\gamma \ge 0$, $\mathsf{Q}_{k,\gamma}(\beta)$ is uniquely minimized by some $\beta^{k,\gamma}$. Moreover, for any $\beta \in \mathbb{R}^d$. 
    \begin{align*}
        \mathsf{Q}_{k, \gamma}(\beta) - \mathsf{Q}_{k,\gamma}(\beta^{k,\gamma}) \ge \frac{1}{2} \|\Sigma^{1/2}(\beta - \beta^{k,\gamma})\|_2^2
    \end{align*}
\end{proposition}

The next lemma shows the explained variance of $\beta^{k,\gamma}$ is smaller than the explained variance of population-level least squares, the latter is smaller than $\sigma_y^2$. 
\begin{lemma}
    \label{lemma:shrinkage}
    Let $\beta^{k,\gamma}$ be the unique minimizer of $\mathsf{Q}_{k,\gamma}(\beta)$, and $\bar{\beta}$ be the unique minimizer of $\mathsf{R}(\beta)$. Then we have
    \begin{align}
    \label{eq:norm}
        \|\Sigma^{1/2} \beta^{k,\gamma}\|_2 \le \|\Sigma^{1/2} \bar{\beta}\|_2\le\sigma_y.
    \end{align}
\end{lemma}

\subsection{Proof of \cref{thm:main-lowdim}}

We need the following technical lemma.
\begin{lemma}
\label{lemma:sqrt}
    For any $x, y, \delta \ge 0$ and $C>0$, if $|x - y| \le C (\delta^2 + \delta \sqrt{y})$, then
    \begin{align*}
        |\sqrt{x} - \sqrt{y}| \le 2(C+1) \delta.
    \end{align*}
\end{lemma}

We are ready to prove \cref{thm:main-lowdim}.
\begin{proof}[Proof of \cref{thm:main-lowdim}]
    We consider the following decomposition, for any $\beta \in \mathbb{R}^d$, 
    \begin{align*}
        \mathsf{Q}_{k,\gamma}(\beta) - \mathsf{Q}_{k,\gamma}(\beta^{k,\gamma}) &= \mathsf{Q}_{k,\gamma}(\beta) - \hat{\mathsf{Q}}_{k,\gamma}(\beta) + \hat{\mathsf{Q}}_{k,\gamma}(\beta) - \hat{\mathsf{Q}}_{k,\gamma}(\beta^{k,\gamma}) \\
        &~~~~~~~~~~ + \hat{\mathsf{Q}}_{k,\gamma}(\beta^{k,\gamma}) - \mathsf{Q}_{k,\gamma}(\beta^{k,\gamma}) \\
        &= \hat{\mathsf{Q}}_{k,\gamma}(\beta) - \hat{\mathsf{Q}}_{k,\gamma}(\beta^{k,\gamma}) \\
        &~~~~~~~~~~ + \mathsf{R}(\beta) - \mathsf{R}(\beta^{k,\gamma}) - \left(\hat{\mathsf{R}}(\beta) - \hat{\mathsf{R}}(\beta^{k,\gamma}) \right)\\
        &~~~~~~~~~~ + \gamma \sum_{j=1}^d \left(w_k(j) - \hat{w}_k(j)\right) \left(|\beta_j| - |\beta_{j}^{k,\gamma}|\right) \\
        &= \mathsf{T}_1(\beta) + \mathsf{T}_2(\beta) + \gamma\mathsf{T}_3(\beta).
    \end{align*}
    For $\mathsf{T}_2(\beta)$, under the event $\mathcal{A}_3(d, t)$ and $\mathcal{A}_4(d,t)$ defined in \eqref{eq:event2}, the following holds with a universal constant $C_1$
    \begin{align*}
        \forall \beta, ~~ \mathsf{T}_2(\beta)&=\mathsf{R}(\beta) - \mathsf{R}(\beta^{k,\gamma}) - \left(\hat{\mathsf{R}}(\beta) - \hat{\mathsf{R}}(\beta^{k,\gamma})\right) \\
        &= \frac{1}{2} (\beta - \beta^{k,\gamma})^\top \Sigma (\beta - \beta^{k,\gamma}){+} (\beta - \beta^{k,\gamma})^\top (\Sigma \beta^{k,\gamma} - u) \\
        &~~~~~~~~ - \frac{1}{2} (\beta - \beta^{k,\gamma})^\top \hat{\Sigma} (\beta - \beta^{k,\gamma}) - (\beta - \beta^{k,\gamma})^\top (\hat{\Sigma} \beta^{k,\gamma} - \hat{u}) \\
        &= \frac{1}{2} \{\Sigma^{1/2}(\beta - \beta^{k,\gamma})\}^\top \left\{I - \Sigma^{-1/2} \hat{\Sigma} \Sigma^{-1/2}\right\} \{\Sigma^{1/2}(\beta - \beta^{k,\gamma})\} \\
        &~~~~~~~~ - \{\Sigma^{1/2}(\beta - \beta^{k,\gamma})\}^\top \left\{I - \Sigma^{-1/2} \hat{\Sigma} \Sigma^{-1/2}\right\}(\Sigma^{1/2} \beta^{k,\gamma}) \\
        &~~~~~~~~ +  \{\Sigma^{1/2}(\beta - \beta^{k,\gamma})\}^\top \Sigma^{-1/2} (\hat{u} - u) \\
        &\le C_1 \Bigg\{ \frac{1}{2} \|\Sigma^{1/2}(\beta - \beta^{k,\gamma})\|_2^2 \cdot \sigma_x^2 \sqrt{b\cdot \zeta(d, t)} \\
        &~~~~~~~~~~~~ +  \|\Sigma^{1/2}(\beta - \beta^{k,\gamma})\|_2 \cdot \left(\|\Sigma^{1/2} \beta^{k,\gamma}\|_2 \sigma_x^2 + \sigma_x^2 \sigma_y\right) \sqrt{b\cdot\zeta(d, t)} \Bigg\}\\
        &\le C_1 \Bigg\{ \frac{1}{2} \|\Sigma^{1/2}(\beta - \beta^{k,\gamma})\|_2^2 \cdot \sigma_x^2 \sqrt{b\cdot \zeta(d, t)} \\
        &~~~~~~~~~~~~ +  \frac{1}{8C_1}\|\Sigma^{1/2}(\beta - \beta^{k,\gamma})\|_2^2\\
        &~~~~~~~~~~~~ +  2C_1\left(\|\Sigma^{1/2} \beta^{k,\gamma}\|_2 \sigma_x^2 + \sigma_x^2 \sigma_y\right)^2 {b\cdot\zeta(d, t)} \Bigg\}.
    \end{align*}
    Here we substitute the upper bounds in \eqref{eq:event2} and use condition that $n \ge 3 (d + t)$ such that $b \cdot{\zeta(d, t)} \le (t+d\cdot\log4)/n\le1$. We also use $\|\Sigma^{1/2} \beta^{k,\gamma}\|_2 \le \sigma_y$ derived in \cref{lemma:shrinkage}.    Substituting the low-dimension structure, if $n \cdot |\mathcal{E}| \ge 64C_1^2 b \sigma_x^4 (d+t)$, we can obtain
    \begin{align*}
        \forall \beta, \qquad \mathsf{T}_2(\beta) \le \frac{1}{4} \|\Sigma^{1/2} (\beta - \beta^{k,\gamma})\|_2^2 + C_2 b \sigma_x^4 \sigma_y^2 \cdot \frac{d + t}{n \cdot |\mathcal{E}|}
    \end{align*} using the fact $xy \le x^2/\epsilon + \epsilon y^2$ for any $x, y, \epsilon>0$. 

    For $\mathsf{T}_3(\beta)$, we have
    \begin{align*}
        \forall \beta, ~~
        \mathsf{T}_3(\beta) &\le  \sum_{j=1}^d |w_k(j) - \hat{w}_k(j)| |\beta_j - \beta_{j}^{k,\gamma}| \\
        &\le  \sup_{j\in [d]} \left|\sqrt{v(\hat{S}_j)} - \sqrt{\hat{v}(\hat{S}_j)}\right|\cdot \|\beta - \beta^{k,\gamma}\|_1 \\
        &\overset{(a)}{\le} C_3  \sqrt{\sigma_x^4 \sigma_y^2 b \rho(k, t)} \sqrt{d} \|\beta - \beta^{k,\gamma}\|_2.
    \end{align*} Here in $(a)$ we use the facts $\|x-y\|_1\le \sqrt{d} \|x-y\|_2$ and 
    \begin{align}
        \label{eq:v-delta}
        \forall S ~\text{with} ~|S| \le k, \qquad |\sqrt{\hat{v}(S)} - \sqrt{v(S)}| \lesssim \delta \qquad \text{with} \qquad \delta = \sqrt{\sigma_x^4 \sigma_y^2 b \rho(k, t)}
    \end{align} by first applying \cref{prop:instance-dependent-v2} and then applying \cref{lemma:sqrt} provided $C \sigma_x^4 \rho(k, t) \le 1$ following from $n\ge \tilde{C}(k\log(d)+t)$ and $|\mathcal{E}|<n^{c_1}$.

    Now we plug in $\beta = \hat{\beta}^{k,\gamma}$, under which $\mathsf{T}_1(\hat{\beta}^{k,\gamma}) \le 0$, denote $\clubsuit = b \sigma_x^4 \sigma_y^2$, we have
    \begin{align}
    \label{eq:lowdim-ub}
    \begin{split}
        &\mathsf{Q}_{k,\gamma}(\hat{\beta}^{k,\gamma}) - \mathsf{Q}_{k,\gamma}(\beta^{k,\gamma}) \\ 
        &~~~~\le C_4 \left\{\clubsuit \cdot \frac{d + t}{n \cdot |\mathcal{E}|} + {\gamma} \sqrt{\clubsuit \cdot \frac{(k\log(d) + \log(|\mathcal{E}|) + t)\cdot d}{n}} \|\hat{\beta}^{k,\gamma} - \beta^{k,\gamma}\|_2 \right\}  \\
        &~~~~~~~~~~ + \frac{1}{4} {\|\Sigma^{1/2} (\hat{\beta}^{k,\gamma} - \beta^{k,\gamma})\|_2^2}.
    \end{split}
    \end{align}
    On the other hand, it follows from \cref{prop:sc} that
    \begin{align}
    \label{eq:lowdim-lb}
    \mathsf{Q}_{k,\gamma}(\hat{\beta}^{k,\gamma}) - \mathsf{Q}_{k,\gamma}(\beta^{k,\gamma}) \ge \frac{\|\Sigma^{1/2}(\hat{\beta}^{k,\gamma} - \beta^{k,\gamma})\|_2^2}{2}.
    \end{align}
    Combining \eqref{eq:lowdim-lb} and \eqref{eq:lowdim-ub} and recalling that we assume $|\mathcal{E}|\le n^{c_1}$, we obtain
    \begin{align*}
    \|\hat{\beta}^{k,\gamma} - \beta^{k,\gamma}\|_2^2 \le C_5 \cdot \clubsuit \left(\frac{\gamma^2}{\kappa^2} \frac{(k\log(d) + c_1\log(n)+ t) d}{n}+\frac{d  + t}{\kappa n \cdot |\mathcal{E}|}\right).
    \end{align*}
We complete the proof with $\tilde{C}=C_6\max\{\sigma_x^4c_1^2,b^{1/2}\sigma_x^2\sigma_yc_1^{1/2}\}$. 
\end{proof}

\begin{proof}[Proof of \cref{lemma:sqrt}]
We divide it into two cases.

\noindent \emph{Case 1. $y\le \delta^2$. } In this case, it follows from triangle inequality that
\begin{align*}
    |x| \le |x-y| + |y| \le C \left(\delta^2 + \delta^2\right) \le 2C \delta^2,
\end{align*} then we can obtain
\begin{align*}
    |\sqrt{x} - \sqrt{y}| \le |\sqrt{x}| + |\sqrt{y}| \le (\sqrt{2C} + 1) \delta.
\end{align*}
\noindent \emph{Case 2. $y \ge \delta^2$. } In this case, it follows from the upper bound on $|x-y|$ and the assumption $y\ge \delta^2$ that
\begin{align*}
    |\sqrt{x} - \sqrt{y}| = \frac{|x - y|}{\sqrt{x} + \sqrt{y}} \le \frac{|x - y|}{\sqrt{y}} \le \frac{C\delta^2 + \delta \sqrt{y}}{\sqrt{y}} = \delta + C \frac{\delta^2}{y} \le (1+C) \delta.
\end{align*}
Combining the above two cases completes the proof.
\end{proof}

\subsection{Proof of \cref{thm:high-dim}}

The next several lemmas are standard in high-dimensional linear regression analysis, and we simply adapt it to the multi-environment setting.

\begin{lemma}
    \label{lemma:highdim-T1-bound}
    Suppose \cref{cond:regularity} holds. Let $\beta^{k,\gamma}$ be the unique minimizer of $\mathsf{Q}_{k,\gamma}(\beta)$. Then there exist some universal constants $C$ such that, the following event
    \begin{align*}
        \forall j\in [d], \qquad &\left| \frac{1}{|\mathcal{E}|} \sum_{e\in \mathcal{E}} \left(\hat{\mathbb{E}}\left[\left(Y^{(e)}-(X^{(e)})^\top\beta^{k,\gamma}\right)X_j^{(e)}\right] - \mathbb{E}\left[\left(Y^{(e)}-(X^{(e)})^\top\beta^{k,\gamma}\right)X_j^{(e)}\right] \right) \right| \\
        &\qquad\qquad\qquad\qquad\qquad
        \le Cb\sigma_x^2\sigma_y
        \left(\sqrt{\zeta(1, t)} + \zeta(1, t) \right)
    \end{align*}
    happens with probability at least $1-e^{-t}$. 
\end{lemma}

\begin{lemma}
\label{lemma:rsc} Suppose \cref{cond:regularity} holds. Then there exist some universal constants $C,c>0$ such that, for any constant $\alpha>0$, the following event
\begin{align*}
    \forall \theta \in \bigcup_{S\subseteq [d], |S| \le s}\{\Delta\in \mathbb{R}^d: \|\Delta_{S^c}\|_{1} \le \alpha \| \Delta_{S}\|_1 \} \qquad \frac{1}{|\mathcal{E}|} \sum_{e\in \mathcal{E}} \hat{\mathbb{E}}[|\theta^\top X^{(e)} |^2] \ge \frac{1}{2} \kappa\|\theta\|_2^2 
\end{align*}
occurs with probability at least $1-3\exp\left(-\tilde{n}/(C\sigma_x)^4\right)$, where $s=c(1+\alpha)^{-2} \sigma_x^{-4}{\kappa}\cdot n |\mathcal{E}|/(b\cdot\log d)$.
\end{lemma}
Now we are ready to prove \cref{thm:high-dim}.

\begin{proof}[Proof of \cref{thm:high-dim}]
Denote $\hat{\beta}=\hat{\beta}^{k,\gamma,\lambda}$, $\beta_\star = (\beta_{\star, 1},\ldots, \beta_{\star, d})^\top = \beta^{k,\gamma}$, $S_\star=\supp(\beta_\star)$ and $\hat{\Delta} = \hat{\beta} - \beta_\star$. 

First, one can observe that
\begin{align*}
    \lambda(\|\hat{\Delta}_{S_\star}\|_1-\|\hat{\Delta}_{S_\star^c}\|_1)
    \overset{(a)}{\ge}\lambda(\|\beta_\star\|_1-\|\hat{\beta}\|_1)
    \overset{(b)}{\ge}\hat{\mathsf{Q}}_{k,\gamma}(\hat{\beta}) - \hat{\mathsf{Q}}_{k,\gamma}(\beta_\star).
\end{align*}
Here $(a)$ follows from the triangle inequality; and $(b)$ follows from the fact that $\hat{\beta}$ minimizes $\hat{\mathsf{Q}}_{k,\gamma}(\beta) + \lambda\|\beta\|_1$. At the same time, it follows from the definition of $\hat{\mathsf{Q}}_{k,\gamma}(\beta)$ that
\begin{align}
\label{eq:highdim-introduce_T1T2}
\begin{split}
    &\hat{\mathsf{Q}}_{k,\gamma}(\hat{\beta}) - \hat{\mathsf{Q}}_{k,\gamma}(\beta_\star)\\
    &=\hat{\mathsf{R}}(\beta)+\sum_{j=1}^d \hat{w}_k (j)\cdot |\hat{\beta}_j|-\hat{\mathsf{R}}(\beta_\star)-\sum_{j=1}^d \hat{w}_k (j)\cdot |\beta_{\star,j}|\\
    &=\frac{1}{2}\hat{\beta}^\top\hat{\Sigma}\hat{\beta}-\hat{\beta}^\top \hat{u}-\frac{1}{2}\beta_\star^\top\hat{\Sigma}\beta_\star+\beta_\star^\top \hat{u}+\gamma \sum_{j=1}^d \hat{w}_k (j) \left(|\hat{\beta}_j|-|\beta_{\star,j}|\right)\\
    &=\frac{1}{2}(\hat{\beta}-\beta_\star+\beta_\star)^\top \hat{\Sigma}(\hat{\beta}-\beta_\star+\beta_\star)-\frac{1}{2}\beta_\star^\top \hat{\Sigma}\beta_\star-(\hat{\beta}-\beta_\star)^\top \hat{u}+\gamma\sum_{j=1}^d \hat{w}_k(j)\left(|\hat{\beta}_j|-|\beta_{\star,j}|\right)\\
    &=\frac{1}{2}\hat{\Delta}^\top\hat{\Sigma}\hat{\Delta}+\hat{\Delta}^\top\hat{\Sigma}\beta_\star-\hat{\Delta}^\top\hat{u}+ \gamma\sum_{j=1}^d \hat{w}_k(j)\left(|\hat{\beta}_j|-|\beta_{\star,j}|\right)\\
    &= \frac{1}{2}\hat{\Delta}^\top\hat{\Sigma}\hat{\Delta} - \hat{\Delta}^\top(\hat{u}-\hat{\Sigma}\beta_\star) + \hat{\Delta}^\top(u-\Sigma\beta_\star) - \hat{\Delta}^\top(u-\Sigma\beta_\star) + \gamma\sum_{j=1}^d \hat{w}_k(j)\left(|\hat{\beta}_j|-|\beta_{\star,j}|\right)\\
    &\overset{(a)}{=} \frac{1}{2}\hat{\Delta}^\top\hat{\Sigma}\hat{\Delta} - \hat{\Delta}^\top \left\{ (\hat{u}-\hat{\Sigma}\beta_\star)-(u-\Sigma\beta_\star)
 \right\} \\
 &\qquad\qquad\qquad+ \gamma \left( -\sum_{j=1}^d \xi_j w_k(j)\left(\hat{\beta}_j-\beta_{\star,j}\right) + \sum_{j=1}^d \hat{w}_k(j)\left(|\hat{\beta}_j|-|\beta_{\star,j}|\right) \right)\\
    &= \frac{1}{2}\hat{\Delta}^\top\hat{\Sigma}\hat{\Delta}+\mathsf{T}_1+\gamma\mathsf{T}_2(\hat{\beta}).
\end{split}
\end{align}
Here $(a)$ follows from the KKT condition that
\begin{align}
\label{eq:kkt}
    \Sigma\beta_\star-u=\Sigma(\beta_\star - \bar{\beta}) = - \gamma \cdot v \odot \xi \qquad \text{with} \qquad v=(w_k(1),\ldots, w_k(d))^\top
\end{align} and
\begin{align*}
    \xi_j \in \begin{cases}
        \{\sign(\beta_\star)\} \qquad &j \in S_\star \\
        [-1, 1] \qquad & j \notin S_\star 
    \end{cases}.
\end{align*}

For $\mathsf{T}_1$, note that the $j$-th coordinate of $\left\{ (\hat{u}-\hat{\Sigma}\beta_\star)-(u-\Sigma\beta_\star)
 \right\}$ is
 \begin{align*}
     \frac{1}{|\mathcal{E}|} \sum_{e\in \mathcal{E}} \left(\hat{\mathbb{E}}\left[(Y^{(e)}-(X^{(e)})^\top\beta^{k,\gamma})X_j^{(e)}\right] - \mathbb{E}\left[(Y^{(e)}-(X^{(e)})^\top\beta^{k,\gamma})X_j^{(e)}\right] \right).
 \end{align*}
Denote $\spadesuit=b\sigma_x^2\sigma_y$. Applying \cref{lemma:highdim-T1-bound} with $t=C\log(n\cdot d)$ yields
\begin{align}
    \label{eq:highdim-T1-lowerbound}
    \mathsf{T}_1\ge -\|(\hat{u}-\hat{\Sigma}\beta_\star)-(u-\Sigma\beta_\star)\|_\infty\|\hat{\Delta}\|_1 \ge - C\spadesuit\sqrt{\frac{\log d+\log n}{n\cdot|\mathcal{E}|}}\|\hat{\Delta}\|_1
\end{align} with probability at least $1-(nd)^{-20}$. 

For $\mathsf{T}_2$, one has
\begin{align*}
    \mathsf{T}_2&= \sum_{j\in S_\star}\left(|\hat{\beta}_j|-|\beta_{\star,j}|\right)\hat{w}_k(j)-\sum_{j\in S_\star}\left(\hat{\beta}_j-\beta_{\star,j}\right)w_k(j)\sign(\beta_{\star,j})\\
    &\qquad\qquad +\sum_{j\notin S_\star}|\hat{\beta}_j|\cdot\hat{w}_k(j)-\sum_{j\notin S_\star}\hat{\beta}_j\cdot w_k(j)\xi_j\\
    &\ge \sum_{j\in S_\star}\left(|\hat{\beta}_j|-|\beta_{\star,j}|\right)\hat{w}_k(j)-\sum_{j\in S_\star}\left(\hat{\beta}_j-\beta_{\star,j}\right)w_k(j)\sign(\beta_{\star,j})\\
    &\qquad\qquad +\sum_{j\notin S_\star}|\hat{\beta}_j|\cdot\hat{w}_k(j)-\sum_{j\notin S_\star}|\hat{\beta}_j|\cdot w_k(j)|\xi_j|\\
    &= \sum_{j\in S_\star^c}\left[|\hat{\beta}_j|\cdot \left( \hat{w}_k(j)-w_k(j) \right)+|\hat{\beta}_j|\cdot w_k(j)(1-|\xi_j|)\right]\\
    &\qquad\qquad +\sum_{j\in S_\star}\left[\left(|\hat{\beta}_j|-|\beta_{\star,j}|\right)\cdot\hat{w}_k(j)+\left(|\beta_{\star,j}|-(\hat{\beta}_j)\sign(\beta_{\star,j})\right)w_k(j)\right]\\
    &\overset{(a)}{\ge} \sum_{j\in S_\star^c}\left[|\hat{\beta}_j|\cdot \left( \hat{w}_k(j)-w_k(j) \right)\right]+\sum_{j\in S_\star}\left[\left(|\hat{\beta}_j|-|\beta_{\star,j}|\right)\cdot\hat{w}_k(j)+\left(|\beta_{\star,j}|-|\hat{\beta}_j|\right)w_k(j)\right]\\
    &=\sum_{j\in S_\star^c}\left[(|\hat{\beta}_j|-|\beta_{\star,j}|)\cdot \left( \hat{w}_k(j)-w_k(j) \right)\right]+\sum_{j\in S_\star}\left[(|\hat{\beta}_j|-|\beta_{\star,j}|)\cdot \left( \hat{w}_k(j)-w_k(j) \right)\right]\\
    &\ge -\sum_{j=1}^d |\hat{\beta}_j-\beta_{\star,j}|\cdot|\hat{w}_k(j)-w_k(j)|.
\end{align*}
Here $(a)$ follows from the fact that $1-|\xi_j|\ge 0$ and $(\hat{\beta}_j)\sign(\beta_{\star,j})\ge -|\hat{\beta}_j|$. It follows from the upper bound of $\|\hat{w}_k-w_k\|_\infty$ derived in \eqref{eq:v-delta} that, provided $C\sigma_x^4\rho(k,t)\le 1$, the following holds with probability at least $1-e^{-t}$
\begin{align*}
    \mathsf{T}_2\ge
    -\|\hat{\beta}-\beta_\star\|_1\|\hat{w}_k-w_k\|_\infty
    \ge 
    -\|\hat{\Delta}\|_1\sqrt{\clubsuit\cdot \rho(k, t)}.
\end{align*}
Set $t=C\log(n\cdot d)$ and recall that we assume $|\mathcal{E}|\le n^{c_1}$. Then the following holds with probability at least $1-(n\cdot d)^{-20}$
\begin{align}
    \label{eq:highdim-T2-lowerbound}
    \mathsf{T}_2\ge -C\sqrt{\clubsuit}\|\hat{\Delta}\|_1\cdot \sqrt{\frac{k\log d+(c_1+1)\log n}{n}}.
\end{align}
Combining \eqref{eq:highdim-introduce_T1T2}, \eqref{eq:highdim-T1-lowerbound} and \eqref{eq:highdim-T2-lowerbound}, we obtain
\begin{align*}
    \lambda(\|\hat{\Delta}_{S_\star}\|_1-\|\hat{\Delta}_{S_\star^c}\|_1)\ge -\|\hat{\Delta}\|_1\underbrace{\left( C \sqrt{\clubsuit}\cdot\gamma\sqrt{\frac{k\log d+(c_1+1)\log n}{n}} + C\spadesuit\sqrt{\frac{\log d+\log n}{n\cdot|\mathcal{E}|}}\right)}_{\lambda^\star}.
\end{align*}
This immediately implies $\|\hat{\Delta}_{{S_\star^c}}\|_1\left( 
\lambda - \lambda^\star\right)\le (\lambda+\lambda^\star)\cdot \|\hat{\Delta}_{S_\star}\|_1$, then the following holds
\begin{align}
\label{eq:3alpha}
    \|\hat{\Delta}_{{S_\star^c}}\|_1 \le 3\|\hat{\Delta}_{S_\star}\|_1
\end{align} provided $\lambda \ge 2\lambda^\star$. Given \eqref{eq:3alpha}, we can apply the restricted strong convexity derived from \cref{lemma:rsc} with $\alpha=3$ and combine \eqref{eq:highdim-introduce_T1T2}, which yields
\begin{align*}
    \lambda(\|\hat{\Delta}_{{S_\star}}\|_1-\|\hat{\Delta}_{S_\star^c}\|_1)\ge \frac{1}{2}\hat{\Delta}^\top\hat{\Sigma}\hat{\Delta} -\lambda^\star\|\hat{\Delta}\|_1 \ge \frac{\kappa}{4}\|\hat{\Delta}\|_2^2-\lambda^\star\|\hat{\Delta}\|_1
\end{align*}
with probability over $1-3\exp(-c{n}/\sigma_x^4)\ge 1-(n\cdot d)^{-20}$. This further implies
\begin{align*}
\kappa \|\hat{\Delta}\|_2^2&\le 4\lambda^\star\|\hat{\Delta}\|_1+4\lambda\|\hat{\Delta}_{S_\star}\|_1
    \overset{(a)}{\le} 12\lambda \|\hat{\Delta}_{S_\star}\|_1
    \overset{(b)}{\le} 12\lambda\sqrt{|S_\star|}\|\hat{\Delta}\|_2.
\end{align*}
Here $(a)$ follows from $\lambda\ge 2\lambda^\star$ and $\|\hat{\Delta}_{{S_\star^c}}\|_1 \le 3\|\hat{\Delta}_{S_\star}\|_1$; and $(b)$ follows from Cauchy-Schwarz inequality. By letting $\tilde{C}=C_3\max\{\sigma_x^4c_1^2,(c_1b)^{1/2} \sigma_x^2\sigma_y ,b\sigma_x^2\sigma_y\}$, we can conclude that
\begin{align*}
    \|\hat{\Delta}\|_2\le \frac{12\lambda}{\kappa}\sqrt{|S_\star|}.
\end{align*}
\end{proof}

\begin{proof}[Proof of \cref{lemma:highdim-T1-bound}]
Note that
\begin{align}
    \frac{1}{|\mathcal{E}|} \sum_{e\in \mathcal{E}} \left(\hat{\mathbb{E}}\left[\left(Y^{(e)}-(X^{(e)})^\top\beta^{k,\gamma}\right)X_j^{(e)}\right] - \mathbb{E}\left[\left(Y^{(e)}-(X^{(e)})^\top\beta^{k,\gamma}\right)X_j^{(e)}\right] \right) 
\end{align}
 is the recentered average of mean-zero independent random variables, each of which is the product of two sub-Gaussian variables. By \cref{cond:regularity}, the product of sub-Gaussian parameters of $(Y^{(e)}-(X^{(e)})^\top\beta^{k,\gamma})$ and $X_j^{(e)}$ is no more than
 \begin{align*}
     C\left(\sigma_y+\sigma_x\max_{e\in\mathcal E}\|(\Sigma^{(e)})^{1/2}\beta^{k,\gamma}\|_2\right)\sigma_x\sqrt{\max_{e\in\mathcal{E},j\in[d]}|\Sigma_{jj}^{(e)}|}
     &\overset{(a)}{\le} C(\sigma_y+\sqrt{b}\sigma_x\|\Sigma^{1/2}\beta^{k,\gamma}\|_2)\sigma_x\sqrt{b}\\
 &\overset{(b)}{\le} C(\sigma_y+\sqrt{b}\sigma_x\sigma_y)\sqrt{b}\sigma_x\le 2C{b}\sigma_x^2\sigma_y.
 \end{align*}
 Here the $(a)$ follows from \cref{cond:regularity}; and $(b)$ follows from \cref{lemma:shrinkage}.
  Consequently, it follows from the concentration inequality \cref{lemma:Bernstein} that 
\begin{align*}
        \forall j\in [d], \qquad &\left| \frac{1}{|\mathcal{E}|} \sum_{e\in \mathcal{E}} \left(\hat{\mathbb{E}}\left[(Y^{(e)}-(X^{(e)})^\top\beta^{k,\gamma})X_j^{(e)}\right] - \mathbb{E}\left[(Y^{(e)}-(X^{(e)})^\top\beta^{k,\gamma})X_j^{(e)}\right] \right) \right| \\
        &\qquad\qquad\qquad\qquad\qquad
        \le 2C{b}\sigma_x^2\sigma_y
        \left(\sqrt{\zeta(1, t)} + \zeta(1, t) \right)
    \end{align*}
    happens with probability at least $1-e^{-t}$.
\end{proof}

\subsection{Proof of \cref{prop:instance-dependent-v2}}

\cref{prop:instance-dependent-v2} is based on instance-dependent decomposition of the response $Y^{(e)}$. We denote the residual defined by the least squares solution constrained on $X_S$ using all the data as 
\begin{align*}
    R^{(e,S)} := Y^{(e)} - (\beta^{(S)})^\top X^{(e)}.
\end{align*}
Define the random vector
\begin{align}
    U^{(e,S)} = X_S^{(e)} R^{(e,S)}.
\end{align} 
The following deterministic lemma unveils the relationship between the calculated weight $v(S)$ and population-level FAIR loss proposed by \cite{gu2024causality}, respectively.

\begin{lemma}
\label{lemma:variational-representation-v2}
    Suppose $\Sigma^{(e)} \succ 0$ for any $e\in \mathcal{E}$. We have
    \begin{align}
    \label{eq:variational-v2s}
        {v}(S) &= \min_{u\in \mathbb{R}^S} \frac{1}{|\mathcal{E}|} \sum_{e\in \mathcal{E}} \left\|({\Sigma}^{(e)})^{-1/2} {\mathbb{E}}\left[(Y^{(e)} - u^\top X_S^{(e)}) X_S^{(e)}\right] \right\|_2^2 \\
    \label{eq:v2-in-r}
        &= \frac{1}{|\mathcal{E}|} \sum_{e\in \mathcal{E}} \left\|({\Sigma}^{(e)})^{-1/2} {\mathbb{E}}\left[U^{(e,S)}\right] \right\|_2^2
    \end{align} Moreover,
    \begin{align}
    \frac{1}{|\mathcal{E}|} \sum_{e\in \mathcal{E}} {\mathbb{E}}\left[X_S^{(e)} R^{(e,S)}\right] = 0.
    \label{eq:v2:first-order-cond}
    \end{align}
\end{lemma}

\cref{prop:instance-dependent-v2} is a deterministic result after defining the following high-probability events. 
\begin{lemma}
\label{lemma:highprob1}
Suppose \cref{cond:regularity} hold. Then there exists some universal constants $C_1, C_2>0$ such that the following two events
\begin{align}
\label{eq:event1}
\begin{split}
\mathcal{A}_1(s, t) &= \Bigg\{ \forall e\in \mathcal{E}, S\subseteq[d] \text{~with~} |S| \le s,  \\
&~~~~~~~~~~~~~~ \left\| (\Sigma^{(e)}_S)^{-1/2} (\hat{\mathbb{E}}[U^{(e,S)}] - \mathbb{E}[U^{(e,S)}]) \right\|_2 \le C_1 \sqrt{b} \sigma_x^2 \sigma_y \left(\sqrt{\rho(s, t)} + \rho(s,t)\right) \Bigg\}  \\
\mathcal{A}_2(s, t) &= \Bigg\{ \forall e\in \mathcal{E}, S\subseteq[d] \text{~with~} |S| \le s, \\
&~~~~~~~~~~~~~~ \left\| (\Sigma^{(e)}_S)^{-1/2} (\hat{\Sigma}_S^{(e)}) (\Sigma^{(e)}_S)^{-1/2} - I \right\|_2 \le C_2 \sigma_x^2 \left(\sqrt{\rho(s, t)} + \rho(s,t)\right)  \Bigg\} \\
\end{split}
\end{align} occurs with probability at least $1-e^{-t}$.
\end{lemma}
\begin{proof}[Proof of \cref{lemma:highprob1}] See \cref{sec:proof:lemma:highprob1}.
\end{proof}

\begin{lemma}
\label{lemma:highprob2}
Suppose \cref{cond:regularity} hold. Then there exists some universal constants $C_1, C_2>0$ such that the following two events
\begin{align}
\label{eq:event2}
\begin{split}
\mathcal{A}_3(s, t) &= \Bigg\{ \forall S\subseteq[d] \text{~with~} |S| \le s,  \\
&~~~~~~~~~~ \left\| (\Sigma_S)^{-1/2} \frac{1}{|\mathcal{E}|} \sum_{e\in \mathcal{E}} (\hat{\mathbb{E}}[U^{(e,S)}] - \mathbb{E}[U^{(e,S)}]) \right\|_2 \le C_1  \sigma_x^2 \sigma_y \left(\sqrt{b\zeta(s, t)} + b\zeta(s,t)\right) \Bigg\}  \\
\mathcal{A}_4(s, t) &= \Bigg\{ \forall S\subseteq[d] \text{~with~} |S| \le s, ~~ \left\| \Sigma_S^{-1/2} \hat{\Sigma}_S (\Sigma_S)^{-1/2} - I \right\|_2 \le C_2 \sigma_x^2 \left(\sqrt{b\zeta(s, t)} + b\zeta(s,t) \right) \Bigg\} \\
\end{split}
\end{align} occurs with probability at least $1-e^{-t}$.
\end{lemma}
\begin{proof}[Proof of \cref{lemma:highprob2}] See \cref{sec:proof:lemma:highprob2}.
\end{proof}

Now we are ready to prove \cref{prop:instance-dependent-v2}.
\begin{proof}[Proof of \cref{prop:instance-dependent-v2}]
The proof proceeds when $\mathcal{A}_1(k, t)$ -- $\mathcal{A}_4(k, t)$ happens and $C \sigma_x^4 \rho(k, t) \le 1$ for some large enough universal constant $C$. In this case we have $\Sigma^{(e)}\succ 0$ and $\rho(k, t) \le \sqrt{\rho(k, t)}$. It follows similar to the proof of \cref{lemma:variational-representation-v2} that $\hat{v}(S) = \min_{u \in \mathbb{R}^{|S|}} \hat{q}_S(a)$ with
\begin{align*}
    \hat{q}_S(a) &= (a - \beta^{(S)}_S)^\top \hat{\Sigma} (a - \beta^{(S)}_S) - 2(a - \beta^{(S)}_S)^\top \left\{\frac{1}{|\mathcal{E}|} \sum_{e\in \mathcal{E}} \hat{\mathbb{E}}[R^{(e,S)} X_S^{(e)}]\right\} \\
     & ~~~~~~~~ + \frac{1}{|\mathcal{E}|} \sum_{e\in \mathcal{E}} \left\|({\hat{\Sigma}}^{(e)})^{-1/2} \hat{\mathbb{E}}\left[X_S^{(e)} R^{(e,S)}\right] \right\|_2^2
\end{align*} provided $\hat{\Sigma}^{(e)}_S \succ 0$ for any $e\in \mathcal{E}$. We can claim that $\hat{q}_S(a)$ can be minimized by 
\begin{align*}
    \hat{a} = \hat{\beta}^{(S)} = \beta_S^{(S)} + (\hat{\Sigma})^{-1} \left\{\frac{1}{|\mathcal{E}|} \sum_{e\in \mathcal{E}} \hat{\mathbb{E}}[R^{(e,S)} X_S^{(e)}]\right\} = \beta_S^{(S)} + (\hat{\Sigma})^{-1} \left\{\frac{1}{|\mathcal{E}|} \sum_{e\in \mathcal{E}} \hat{\mathbb{E}}[U^{(e,S)}] \right\},
\end{align*} substituting it into $\hat{q}_S$, we obtain
\begin{align*}
    \hat{v}(S) = \hat{q}_S(\hat{a}) &= \frac{1}{|\mathcal{E}|} \sum_{e\in \mathcal{E}} \hat{\mathbb{E}}\left[U^{(e,S)}\right]^\top ({\hat{\Sigma}}^{(e)})^{-1} \hat{\mathbb{E}}\left[U^{(e,S)}\right] \\
    & ~~~~~~~~ - \left\{\frac{1}{|\mathcal{E}|} \sum_{e\in \mathcal{E}} \hat{\mathbb{E}}[U^{(e,S)}]\right\} (\hat{\Sigma}_S)^{-1}\left\{\frac{1}{|\mathcal{E}|} \sum_{e\in \mathcal{E}} \hat{\mathbb{E}}[U^{(e,S)}]\right\}   = \hat{\mathsf{T}}_1 + \hat{\mathsf{T}}_2.
\end{align*}

For $\hat{\mathsf{T}}_1$, we do the following decomposition,
\begin{align*}
    \hat{\mathsf{T}}_1 &= \frac{1}{|\mathcal{E}|} \sum_{e\in \mathcal{E}} \left\{(\Sigma^{(e)}_S)^{-1/2}\left(\hat{\mathbb{E}}[U^{(e,S)}] - \mathbb{E}[U^{(e,S)}] + \mathbb{E}[U^{(e,S)}] \right)\right\}^\top ((\Sigma^{(e)}_S)^{-1/2}\hat{\Sigma}^{(e)}_S(\Sigma^{(e)}_S)^{-1/2})^{-1} \\
    &~~~~~~~~~~~~~~~~~~~~ \times \left\{(\Sigma^{(e)}_S)^{-1/2}\left(\hat{\mathbb{E}}[U^{(e,S)}] - \mathbb{E}[U^{(e,S)}] + \mathbb{E}[U^{(e,S)}] \right)\right\} 
\end{align*}
We let $\Delta_1^{(e)} = (\Sigma^{(e)}_S)^{-1/2}(\hat{\mathbb{E}}[U^{(e,S)}] - \mathbb{E}[U^{(e,S)}]) \in \mathbb{R}^{|S|}$, and $\Delta_2^{(e)} = (\Sigma^{(e)}_S)^{-1/2}\hat{\Sigma}^{(e)}_S(\Sigma^{(e)}_S)^{-1/2} - I$ satisfying $\|\Delta_2^{(e)}\| \le 0.5$ by our assumption on $n$. Then it follows from Weyl's theorem that
\begin{align*}
    \lambda_{\min} \left((\Delta_2^{(e)} + I)^{-1}\right) &\ge \frac{1}{\lambda_{\max}\left(\Delta_2^{(e)} + I\right)} \ge \frac{1}{1 + \|\Delta_2^{(e)}\|_2} \ge 1 - 2\|\Delta_2^{(e)}\|_2 \\
    \lambda_{\max} \left((\Delta_2^{(e)} + I)^{-1}\right) &\le \frac{1}{\lambda_{\min}(\Delta_2^{(e)} + I)} \le \frac{1}{1 - \|\Delta_2^{(e)}\|_2} \le 1+2\|\Delta_2^{(e)}\|_2,
\end{align*} where the last inequalities follows from the fact that $1/(1-x) \le 1+2x$ and $1/(1+x) \ge 1-2x$ when $x\in [0,0.5]$. We thus have
\begin{align}
\label{eq:bound-delta-2}
    \left\|(\Delta_2^{(e)} + I)^{-1} - I\right\|_2 \le 2\|\Delta_2\|_2 \qquad \text{and} \qquad \left\|(\Delta_2^{(e)} + I)^{-1}\right\|_2 \le 1.5
\end{align}
Therefore, it follows from the triangle inequality and Cauchy-Schwarz inequality that
\begin{align*}
    \left| \hat{\mathsf{T}}_1 - v(S) \right| &= \Bigg|\frac{1}{|\mathcal{E}|} \sum_{e\in \mathcal{E}} \left\{ \Delta_1^{(e)} +  (\Sigma^{(e)}_S)^{-1/2} \mathbb{E}[U^{(e,S)}] \right\}^\top (\Delta_2 + I)^{-1} \\
    &~~~~~~~~~~~~~~~~~~\times \left\{\Delta_1^{(e)} + (\Sigma^{(e)}_S)^{-1/2} \mathbb{E}[U^{(e,S)}] \right\} - v(S)\Bigg|\\
    &\le \frac{1}{|\mathcal{E}|} \sum_{e\in \mathcal{E}} 2 \|\Delta_1^{(e)}\|_2 \left\|(\Delta_2 + I)^{-1}\right\|_2 \left\|(\Sigma^{(e)}_S)^{-1/2} \mathbb{E}[U^{(e,S)}]\right\|_2 \\
    &~~~~~~~~~~ + \frac{1}{|\mathcal{E}|} \sum_{e\in \mathcal{E}} \|\Delta_1^{(e)}\|_2^2 \left\|(\Delta_2 + I)^{-1}\right\|_2 \\
    &~~~~~~~~~~ + \frac{1}{|\mathcal{E}|} \sum_{e\in \mathcal{E}}  \left\|(\Delta_2+I)^{-1} - I \right\|_2 \left\| (\Sigma^{(e)}_S)^{-1/2} \mathbb{E}[U^{(e,S)}]\right\|_2^2 \\
    &\overset{(a)}{\le} 3 \sqrt{\frac{1}{|\mathcal{E}|} \sum_{e\in \mathcal{E}} \left\|(\Sigma^{(e)}_S)^{-1/2} \mathbb{E}[U^{(e,S)}]\right\|_2^2} \sqrt{\frac{1}{|\mathcal{E}|} \sum_{e\in \mathcal{E}} \|\Delta_1^{(e)}\|_2^2} + 1.5 \frac{1}{|\mathcal{E}|} \sum_{e\in \mathcal{E}} \|\Delta_1^{(e)}\|_2^2  \\
    &~~~~~~~~~~ + \left(\sup_{e\in \mathcal{E}}  2 \|\Delta_2^{(e)}\|_2^2 \right) \frac{1}{|\mathcal{E}|} \sum_{e\in \mathcal{E}} \left\| (\Sigma^{(e)}_S)^{-1/2} \mathbb{E}[U^{(e,S)}]\right\|_2^2 \\
    &\overset{(b)}{\le} 3 \sqrt{v(S) \cdot C_1^2 \sigma_x^4\sigma_y^2 b \rho(k, t)} + 1.5 C_1^2 \sigma_x^4 \sigma_y^2 b \rho(k, t) + 2C_2 \sigma_x^2 \sqrt{\rho(k,t)} \cdot v(S) \\
    &\overset{(c)}{\le} 4 \sqrt{v(S)} \cdot \sqrt{(C_1^2 + C_2^2) \sigma_x^4 \sigma_y^2 b \cdot \rho(k, t)} + 1.5 C_1^2 \sigma_x^4 \sigma_y^2 b\rho(k, t)
\end{align*} where $(a)$ follows from the inequalities  \eqref{eq:bound-delta-2} and Cauchy-Schwarz inequality, $(b)$ follows from \eqref{eq:event1}, and $(c)$ follows from the fact
\begin{align*}
\|(\Sigma_S^{(e)})^{-1/2} \mathbb{E}[U^{(e,S)}]\|_2 &\le \|(\Sigma_S^{(e)})^{-1/2} \mathbb{E}[X^{(e)} Y^{(e)}] \|_2 + \|(\Sigma_S^{(e)})^{1/2} \beta^{(S)}\|_2 \\
&\overset{(d)}{\le} \sqrt{(\mathbb{E}[X_S^{(e)} Y^{(e)}])^\top (\Sigma_S^{(e)})^{-1} (\mathbb{E}[X_S^{(e)} Y^{(e)}])} \\
&~~~~~~~~ + \|(\Sigma_S^{(e)})^{1/2} \Sigma_S^{-1/2}\|_2 \|\Sigma_S^{1/2} \beta^{(S)}\|_2 \\
&\le \sigma_y + \sqrt{b} \sigma_y \le 2\sqrt{b} \sigma_y,
\end{align*} which further implies that
\begin{align*}
    v(S) = \sqrt{\frac{1}{|\mathcal{E}|} \|(\Sigma_S^{(e)})^{-1/2} \mathbb{E}[U^{(e,S)}]\|_2^2} \cdot \sqrt{v(S)} \le \sqrt{4b \sigma_y^2} \cdot \sqrt{v(S)}.
\end{align*}
Here $(d)$ follows from the fact that the covariance matrix of $[X^{(e)}_S, Y^{(e)}]$ are positive semi-definite thus the Schur complement satisfies
\begin{align*}
    \sigma_y^2 - (\mathbb{E}[X_S^{(e)} Y^{(e)}])^\top (\Sigma_S^{(e)})^{-1} (\mathbb{E}[X_S^{(e)} Y^{(e)}]) \ge 0, 
\end{align*} and a similar argument to the covariance matrix of the mixture distribution $[X_S, Y] \sim \frac{1}{|\mathcal{E}|} \sum_{e\in \mathcal{E}} \mu_{(x_S, y)}^{(e)}$.

For $\hat{\mathsf{T}}_2$, observe that $\frac{1}{|\mathcal{E}|} \sum_{e\in \mathcal{E}} \mathbb{E}[X_S^{(e)} R^{(e,S)}] = 0$, then following \eqref{eq:event2}, \eqref{eq:bound-delta-2} and the fact that $b\sigma_x^4\zeta(k, t)\le \sigma_x^4\rho(k, t) \le 1$ since $b\le|\mathcal{E}|$ by \cref{cond:regularity},
\begin{align*}
    |\hat{\mathsf{T}}_2| &\le \left\| (\Sigma_S)^{-1/2} \frac{1}{|\mathcal{E}|} \sum_{e\in \mathcal{E}} (\hat{\mathbb{E}}[U^{(e,S)}] - \mathbb{E}[U^{(e,S)}]) \right\|_2^2 \cdot \left\|(\Sigma^{-1/2}_S \hat{\Sigma}_S \Sigma^{-1/2}_S)^{-1}\right\|_2 \\
    &\le 1.5  (C_2^2\sigma_x^2\sigma_y)^2\left( \sqrt{b\zeta(s, t)} + b\zeta(s,t) \right)^2 \le C_3 \sigma_x^4\sigma_y^2 \rho(k, t).
\end{align*}

Putting all the pieces together, we can conclude that
\begin{align*}
     \left| \hat{\mathsf{T}}_1 + \hat{\mathsf{T}}_2 - v(S) \right| &\le \left| \hat{\mathsf{T}}_1 - v(S) \right| + \left|\hat{\mathsf{T}}_2 \right| \\
     &\le C_4 \left(b\sigma_x^4 \sigma_y^2\rho(k, t) + \sqrt{v(S)} \sqrt{b\sigma_x^4 \sigma_y^2\rho(k, t)}\right).
\end{align*} This completes the proof.
\end{proof}

\subsection{Proof of \cref{prop:sc}}

We first establish the existence and uniqueness of $\beta^{k,\gamma}$. The existence of an optimal solution follows from the fact that $\mathsf{Q}_{k,\gamma}(\beta)$ is continuous in $\mathbb{R}^d$, and its optimal solution can be attained on the closed set $F=\{\beta: \|\beta - \bar{\beta}\|_2 \le (1/\lambda_{\min} (\Sigma))^{1/2} \bar{\beta}^\top \Sigma \bar{\beta}\}$ given 
\begin{align*}
    2\mathsf{Q}_{k,\gamma}(\beta) \ge (\beta - \bar{\beta})^\top \Sigma (\beta - \bar{\beta}) \ge \bar{\beta}^\top \Sigma \bar{\beta} = 2\mathsf{Q}_{k,\gamma}(0) \qquad \forall ~ \beta \in F^c.
\end{align*}
The uniqueness will be established using the proof-by-contradiction argument. Let $\beta'$ and $\beta^\dagger$ be two optimal solutions with $\beta' \neq \beta^\dagger$, then
\begin{align*}
    \mathsf{Q}_{k,\gamma}\left(\frac{\beta'+\beta^\dagger}{2} \right) &= \mathsf{R}\left(\frac{\beta'+\beta^\dagger}{2} \right) + \sum_{j=1}^d \gamma w_k(j) \left|\frac{\beta'_j+\beta^\dagger_j}{2} \right| \\
    &\overset{(a)}{<} \frac{1}{2} \left\{\mathsf{R}(\beta') + \mathsf{R}(\beta^\dagger)\right\} + \frac{1}{2} \left\{\sum_{j=1}^d \gamma w_k(j) (|\beta'_j|+|\beta^\dagger_j|)\right\} \\
    &{\le} \frac{1}{2} \mathsf{Q}_{k,\gamma}\left(\beta' \right) + \mathsf{Q}_{k,\gamma}\left(\beta^\dagger \right).
\end{align*} Here (a) follows from the fact that $\mathsf{R}(\beta)$ is quadratic function with positive eigenvalues and hence is further strongly convex. This is contrary to the fact that $\beta'$ and $\beta^\dagger$ are optimal solutions. 

Finally, we show the loss is strong convex with respect to $\beta^{k,\gamma}$. Let $S^{k,\gamma} = \supp(\beta^{k,\gamma})$.
Observe that
\begin{align*}
    \mathsf{R}(\beta) - \mathsf{R}(\beta^{k,\gamma}) &= (\beta - \bar{\beta})^\top \Sigma (\beta - \bar{\beta}) - (\beta^{k,\gamma} - \bar{\beta})^\top \Sigma (\beta^\gamma - \bar{\beta}) \\
    &= \frac{1}{2} (\beta - \beta^{k,\gamma})^\top \Sigma (\beta - \beta^{k,\gamma}) - (\beta - \beta^{k,\gamma})^\top \Sigma (\bar{\beta} - \beta^{k,\gamma}).
\end{align*} Putting these pieces together, we obtain
\begin{align*}
    \mathsf{Q}_{k,\gamma}(\beta) - \mathsf{Q}_{k,\gamma}(\beta^{k,\gamma}) &\overset{(a)}{=} \frac{1}{2} (\beta - \beta^{k,\gamma})^\top \Sigma (\beta - \beta^{k,\gamma}) - (\beta - \beta^{k,\gamma})^\top \Sigma (\bar{\beta} - \beta^{k,\gamma}) \\
    & ~~~~~~~~ + \gamma \sum_{j=1}^d v_j \left(|\beta_j| - |\beta^{k,\gamma}_j| \right) \\
    &\overset{(b)}{=} \frac{1}{2} (\beta - \beta^{k,\gamma})^\top \Sigma (\beta - \beta^{k,\gamma}) \\
    &~~~~~~~~ + \gamma \sum_{j=1}^d v_j \left\{-(\beta_j - \beta^{k,\gamma}_j) \xi_j + |\beta_j| - |\beta_j^{k,\gamma}|\right\} \\
    &= \frac{1}{2} (\beta - \beta^{k,\gamma})^\top \Sigma (\beta - \beta^{k,\gamma}) + \gamma \sum_{j=1}^d v_j \left\{|\beta_j| - \xi_j \beta_j\right\} \\
    &\overset{(c)}{\ge} \frac{1}{2} \|\Sigma^{1/2}(\beta - \beta^{k,\gamma})\|_2^2
\end{align*} Here $(a)$ follows from the calculation of $\mathsf{R}(\beta) - \mathsf{R}(\beta^{k,\gamma})$, $(b)$ follows from the KKT condition \eqref{eq:kkt}, $(c)$ follows from the fact that $\xi_j\in [-1,1]$. This completes the proof.

\subsection{Proof of \cref{lemma:variational-representation-v2}}

    It follows from the identity $\beta^{(e,S)}_S = (\Sigma_S^{(e)})^{-1} \mathbb{E}[Y^{(e)} X_S^{(e)}]$ that
    \begin{align*}
    v(S) &= \frac{1}{|\mathcal{E}|} \sum_{e\in \mathcal{E}} (\beta_S^{(e,S)} - \beta_S^{(S)})^\top \Sigma_S^{(e)} (\beta_S^{(e,S)} - \beta_S^{(S)}) \\
    &= \frac{1}{|\mathcal{E}|} \sum_{e\in \mathcal{E}} \mathbb{E}[Y^{(e)} X_S^{(e)}]^\top (\Sigma_S^{(e)})^{-1} \mathbb{E}[Y^{(e)} X_S^{(e)}] - \mathbb{E}[Y^{(e)}X_S^{(e)}]^\top \beta_S^{(S)} + \beta_S^{(S)} \Sigma_S^{(e)} \beta_S^{(S)} \\
    &= \frac{1}{|\mathcal{E}|} \sum_{e\in \mathcal{E}} \left\|(\Sigma_S^{(e)})^{-1/2} \left(\mathbb{E}[Y^{(e)} X_S^{(e)}] - \Sigma^{(e)}_S \beta_S^{(S)} \right) \right\|_2^2.
    \end{align*}
    At the same time, for any $a\in \mathbb{R}^{|S|}$, plugging $Y^{(e)} = (\beta_S^{(S)})^\top X_S^{(e)} + R^{(e,S)}$ gives
    \begin{align*}
    q_S(a) = & \frac{1}{|\mathcal{E}|} \sum_{e\in \mathcal{E}} \left\|(\Sigma_S^{(e)})^{-1/2} \left(\mathbb{E}[Y^{(e)} X_S^{(e)}] - \Sigma^{(e)}_S a \right) \right\|_2^2 \\
    = &  \frac{1}{|\mathcal{E}|} \sum_{e\in \mathcal{E}} \left\|(\Sigma_S^{(e)})^{-1/2} \left(\mathbb{E}[ R^{(e,S)} X_S^{(e)}] - \Sigma^{(e)}_S (a - \beta^{(S)}_S) \right) \right\|_2^2 \\
    = & (a - \beta^{(S)}_S)^\top \Sigma (a - \beta^{(S)}_S) - 2(a - \beta^{(S)}_S)^\top \left\{\frac{1}{|\mathcal{E}|} \sum_{e\in \mathcal{E}} \mathbb{E}[R^{(e,S)} X_S^{(e)}]\right\} \\
     & ~~~~~~~~ + \frac{1}{|\mathcal{E}|} \sum_{e\in \mathcal{E}} \left\|({\Sigma}^{(e)})^{-1/2} {\mathbb{E}}\left[X_S^{(e)} R^{(e,S)}\right] \right\|_2^2.
    \end{align*} It follows from the definition of $R^{(e,S)}$ and the definition of $\beta^{(S)}$ that
    \begin{align*}
        \frac{1}{|\mathcal{E}|} \sum_{e\in \mathcal{E}} \mathbb{E}[X_S^{(e)} R^{(e,S)}] &= \frac{1}{|\mathcal{E}|} \sum_{e\in \mathcal{E}} \mathbb{E}[X_S^{(e)} (Y^{(e)} - (\beta^{(S)}_S)^\top X_S^{(e)})] \\
        &= \frac{1}{|\mathcal{E}|} \sum_{e\in \mathcal{E}} \mathbb{E}[X_S^{(e)} Y^{(e)}] - \Sigma_S \beta^{(S)}_S = 0.
    \end{align*} This verifies \eqref{eq:v2:first-order-cond}. Therefore $a^\star = \beta^{(S)}_S$ attains the global minima of $q_S(a)$,  this verifies \eqref{eq:variational-v2s} and \eqref{eq:v2-in-r}.

\subsection{Proof of \cref{lemma:highprob1}}
\label{sec:proof:lemma:highprob1}

\noindent {\it High probability error bound in $\mathcal{A}_1(k, t)$. } For any $S\subseteq [d]$ with $|S|\le s$, let $w^{(S,1)}, \ldots, w^{(S,N_S)}$ be an $1/4-$covering of unit ball $\mathcal{B}_S = \{x \in \mathbb{R}^d: x_{S^c} = 0, \|x\|_2 \le 1\}$, that is, for any $w \in \mathcal{B}_S$, there exists some $\pi(w) \in [N_S]$ such that
\begin{align}
\label{eq:proof-cover1}
	\|w - w^{(S,\pi(w))}\|_2 \le 1/4.
\end{align}
	It follows from standard empirical process result that $N_S \le 9^{|S|}$, then 
\begin{align}
\label{eq:calculate-cover-number}
\begin{split}
	N = \sum_{|S| \le s} N_S \le \sum_{|S|\le s} 9^{|S|} &\le \sum_{i=0}^s 9^{i} \binom{d}{i} \\
            &\le \left(\frac{9d}{s}\right)^{s} \sum_{i=0}^s \left(\frac{s}{d}\right)^i \binom{d}{i} \le \left(\frac{9d}{s}\right)^{s} \sum_{i=0}^d \left(\frac{s}{d}\right)^i \binom{d}{i} \\
            & = \left(\frac{9d}{s}\right)^{s} \left(1 + \frac{s}{d}\right)^d \le \left(\frac{9 \times 4 d}{s}\right)^{s}.
\end{split}
\end{align} 

At the same time, for fixed $e$ and $S$, denote $\xi = (\Sigma^{(e)}_S)^{-1/2} (\hat{\mathbb{E}}[U^{(e,S)}] - \mathbb{E}[U^{(e,S)}])$. It follows from the variational representation of the $\ell_2$ norm that
\begin{align*}
\|\xi\|_2 &= \sup_{w \in \mathcal{B}_S} w_S^\top \xi \le \sup_{\ell \in [N_S]} (w_S^{(S,\ell)})^\top \xi + \sup_{w \in \mathcal{B}_S} (w_S - w_S^{(S,\pi(w))})^\top \xi \le \sup_{\ell \in [N_S]} (w_S^{(S,\ell)})^\top \xi + \frac{1}{4}\|\xi\|_2,
\end{align*} where the last inequality follows from the Cauchy-Schwarz inequality and our construction of covering in \eqref{eq:proof-cover1}. This implies $\|\xi\|_2 \le 2 \sup_{\ell \in [N_S]} (w_S^{(S,\ell)})^\top \xi$, thus
\begin{align}
\label{eq:proof-sup1}
\begin{split}
	&\sup_{e\in \mathcal{E}} \sup_{|S| \le s} \left\|(\Sigma^{(e)}_S)^{-1/2} (\hat{\mathbb{E}}[U^{(e,S)}] - \mathbb{E}[U^{(e,S)}])\right\|_2  \\
	&~~~~~~ \le 2 \sup_{e\in \mathcal{E}, |S|\le s, \ell\in [N_S]} \underbrace{(w_S^{(S,\ell)})^\top (\Sigma^{(e)}_S)^{-1/2} \frac{1}{n} \sum_{i=1}^n \left(X_{i,S}^{(e)} R_i^{(e,S)} - \mathbb{E}[X_{S}^{(e)} R^{(e,S)}] \right)}_{Z_1(e, S,\ell)}. 
\end{split}
\end{align} 
Note for fixed $e,S$ and $\ell$, $Z_1(e, S,\ell)$ is the recentered average of independent random variables, each of which is the product of two sub-Gaussian variables. By \cref{cond:regularity}, $(w_S^{(S,\ell)})^\top (\Sigma^{(e)}_S)^{-1/2} X_{S}^{(e)}$ has sub-Gaussian parameter at most $\sigma_x$, and the sub-Gaussian parameter of
$R^{(e,S)} := Y^{(e)} - (\beta^{(S)})^\top X^{(e)}$ is no more than 
\begin{align}
\label{eq:subgaussian-of-R}
    \begin{split}
        \sigma_y+\sigma_x\left\|(\Sigma_S^{(e)})^{1/2}\beta^{(S)}\right\|_2
        &\overset{(a)}{\le}\sigma_y+\sigma_x\left\|(\Sigma^{(e)}_S)^{1/2} \Sigma_S^{-1/2} \right\|_2 \left\|\Sigma_S^{-1/2} \left(\frac{1}{|\mathcal{E}|} \sum_{e\in \mathcal{E}} \mathbb{E}[X_S^{(e)} Y^{(e)}] \right) \right\|_2 \\
        &\overset{(b)}{\le} \sigma_y+\sigma_x \left\|(\Sigma^{(e)}_S)^{1/2} \Sigma_S^{-1/2} \right\|_2 \sqrt{\frac{1}{|\mathcal{E}|} \sum_{e\in \mathcal{E}}\mathbb{E}[(Y^{(e)})^2]}\\
        &\overset{(c)}{\le} \sigma_y+\sigma_x \sqrt{b} \sigma_y.
    \end{split}
\end{align}
Here $(a)$ follows from the property of the operator norm and the definition of $\beta^{(S)}$; $(b)$ follows from the Cauchy-Schwarz inequality; and $(c)$ follows from \cref{cond:regularity}. Therefore, $(w_S^{(S,\ell)})^\top (\Sigma^{(e)}_S)^{-1/2}X_{S}^{(e)} R^{(e,S)}$ is the product of two sub-Gaussian variables with parameter no more than $\sigma_x$ and $\sigma_y+\sigma_x \sqrt{b} \sigma_y$.
Then it follows from the tail bound for sub-exponential random variable that
\begin{align*}
    \forall e\in \mathcal{E}, |S|\le s, \ell\in [N_S], \quad
    \mathbb{P}\left[|Z_1(e, S,\ell)| \ge C' b^{1/2} \sigma_x^2 \sigma_y \left(\frac{u}{n} + \sqrt{\frac{u}{n}}\right)\right]\le 2e^{-u},\quad \forall u>0.
\end{align*}
Letting $u=t+\log(2N|\mathcal{E}|) \le 6 \left(t + s\log(4d/s) + \log(|\mathcal{E}|)\right)$, we obtain
\begin{align*}
    &\mathbb{P}\left[\sup_{e\in \mathcal{E}, |S|\le s, \ell\in [N_S]} |Z_1(e, S, \ell)| \ge 6C'b^{1/2} \sigma_x^2 \sigma_y \left(\sqrt{\rho(s, t)}+\rho(s, t) \right) \right]\\ 
    &~~~~ \le  N|\mathcal{E}|\times 2e^{-\log(2N|\mathcal{E}|)-t} \le e^{-t}.
\end{align*} Combining with the argument \eqref{eq:proof-sup1} concludes the proof of the claim with $C_1=12C'$.

\noindent {\it High probability error bound in $\mathcal{A}_2(k, t)$. } For any symmetric matrix $Q\in\mathbb{R}^{d\times d}$, it follows from the variational representation of the operator norm that,  
\begin{align*}
    \|Q_S\|_2=\sup_{w \in \mathcal{B}_S}  w_S^\top Q_S w_S &\le \sup_{l\in [N_S]} (w_S^{(S,\ell)})^\top Q_S(w_S^{(S,\ell)})\\
    &\qquad\qquad + \sup_{w \in \mathcal{B}_S} 2  (w_S-w_S^{(S,\pi(w))})^\top Q_Sw_S^{(S,\pi(w))}\\
    &\qquad\qquad+ \sup_{w \in \mathcal{B}_S}(w_S-w_S^{(S,\pi(w))})^\top Q_S(w_S-w_S^{(S,\pi(w))}).\\
    &\le \sup_{l\in [N_S]} (w_S^{(S,\ell)})^\top Q_S(w_S^{(S,\ell)})+\frac{1}{2}\|Q_S\|_2+\frac{1}{16}\|Q_S\|_2,
\end{align*}
which implies $\|Q_S\|_2\le 3\sup_{\ell\in[N_S]}(w_S^{(S,\ell)})^\top Q_S(w_S^{(S,\ell)})$, thus
\begin{align}
\label{eq:proof-sup2}
\begin{split}
    &\sup_{e\in \mathcal{E}} \sup_{|S| \le s} \left\| (\Sigma^{(e)}_S)^{-1/2} (\hat{\Sigma}_S^{(e)}) (\Sigma^{(e)}_S)^{-1/2} - I \right\|_2\\
    &~~~~~~\le 3 \sup_{e\in \mathcal{E}, |S|\le s, \ell\in [N_S]}\underbrace{(w_S^{(S,\ell)})^\top\left[ (\Sigma^{(e)}_S)^{-1/2} (\hat{\Sigma}_S^{(e)}) (\Sigma^{(e)}_S)^{-1/2} - I \right](w_S^{(S,\ell)})}_{Z_2(e,S,\ell)}.
\end{split}
\end{align}
Note for fixed $e,S$ and $\ell$, $Z_2(e,S,\ell)$ is the recentered average of independent random variables, each of wich is the square of a sub-Gaussian variable $(w_S^{(S,\ell)})^\top(\Sigma_S^{(2)})^{-1/2}X_S^{(e)}$ with parameter at most $\sigma_x$, by \cref{cond:regularity}. Then it follows from the tail bound for exponential random variable that
\begin{align*}
    \forall e\in \mathcal{E}, |S|\le s, \ell\in [N_S], \quad
    \mathbb{P}\left[|Z_2(e, S,\ell)| \ge C'  \sigma_x^2 \left(\frac{u}{n} + \sqrt{\frac{u}{n}}\right)\right]\le 2e^{-u},\quad \forall u>0.
\end{align*}
Letting $u=t+\log(2N|\mathcal{E}|) \le 6 \left(t + s\log(4d/s) + \log(|\mathcal{E}|)\right)$, we obtain
\begin{align*}
    \mathbb{P}\left[\sup_{e\in \mathcal{E}, |S|\le s, \ell\in [N_S]} |Z_2(e, S, \ell)| \ge 6C'  \sigma_x^2 \left(\sqrt{\rho(s, t)}+\rho(s, t) \right) \right] \le  N|\mathcal{E}|\times 2e^{-\log(2N|\mathcal{E}|)-t} \le e^{-t}.
\end{align*} 
Combining with the argument \eqref{eq:proof-sup2} concludes the proof of the claim with $C_2=18C'$.

\subsection{Proof of \cref{lemma:highprob2}}
\label{sec:proof:lemma:highprob2}

\noindent {\it High probability error bound in $\mathcal{A}_3(k, t)$. } The proof idea is almost identical to \cref{lemma:highprob1}. For any $S\subseteq [d]$ with $|S|\le s$, let $w^{(S)}_1, \ldots, w^{(S)}_{N_S}$ be an $1/4-$covering of unit ball $\mathcal{B}_S = \{x \in \mathbb{R}^d: x_{S^c} = 0, \|x\|_2 \le 1\}$. Recall that in \cref{lemma:highprob1} we obtain $\|\xi\|_2 \le 2 \sup_{\ell \in [N_S]} (w_S^{(S,\ell)})^\top \xi$ by the variational representation of $\ell_2$ norm. This immediately yields,
\begin{align}
\label{eq:proof-sup3}
\begin{split}
	&\sup_{|S| \le s} \left\| (\Sigma_S)^{-1/2} \frac{1}{|\mathcal{E}|} \sum_{e\in \mathcal{E}} (\hat{\mathbb{E}}[U^{(e,S)}] - \mathbb{E}[U^{(e,S)}]) \right\|_2  \\
	&~~~~~~ \le 2 \sup_{|S|\le s, \ell\in [N_S]} \underbrace{(w_S^{(S,\ell)})^\top (\Sigma_S)^{-1/2}\frac{1}{n\cdot\mathcal{E}} \sum_{e\in\mathcal{E}}\sum_{i=1}^n \left(X_{i,S}^{(e)} R_i^{(e,S)} - \mathbb{E}[X_{S}^{(e)} R^{(e,S)}] \right)}_{Z_3(S,\ell)}. 
\end{split}
\end{align} 
Note by \cref{cond:regularity}, for fixed $e,S$ and $\ell$, $(w_S^{(S,\ell)})^\top(\Sigma_S)^{-1/2}X_{S}^{(e)}$ is a sub-Gaussian variable with parameter 
$\sigma_{e,S,\ell}=\left((w_S^{(S,\ell)})^\top(\Sigma_S)^{-1/2}(\Sigma_S^{(e)})(\Sigma_S)^{-1/2}w_S^{(S,\ell)}\right)^{1/2}\sigma_x$,
which satisfies 
\begin{align*}
    \left(\sum_{e\in\mathcal{E}}\sum_{i=1}^n (\sigma_{e,S,\ell})^2\right)^{1/2}
    &=
    \left(n\sum_{e\in\mathcal{E}}(\sigma_{e,S,\ell})^2\right)^{1/2}\\
    &=\left(n\sum_{e\in\mathcal{E}}(w_S^{(S,\ell)})^\top(\Sigma_S)^{-1/2}(\Sigma_S^{(e)})(\Sigma_S)^{-1/2}w_S^{(S,\ell)}\right)^{1/2}\sigma_x\\
    &=\left(n\cdot(w_S^{(S,\ell)})^\top(\Sigma_S)^{-1/2}\left(\sum_{e\in\mathcal{E}}\Sigma_S^{(e)}\right)(\Sigma_S)^{-1/2}w_S^{(S,\ell)}\right)^{1/2}\sigma_x
    \\
    &=(n\cdot|\mathcal{E}|)^{1/2}\sigma_x.
\end{align*}
Also, from \cref{cond:regularity}, we have
\begin{align}
    \label{eq:upper-bound-sigma_esl}
    \begin{split}
        \forall e\in \mathcal{E}, |S|\le s, \ell\in [N_S],\quad
    \sigma_{e,S,\ell}&=\left((w_S^{(S,\ell)})^\top(\Sigma_S)^{-1/2}(\Sigma_S^{(e)})(\Sigma_S)^{-1/2}w_S^{(S,\ell)}\right)^{1/2}\sigma_x\\
    &\le \sqrt{b\cdot(w_S^{(S,\ell)})^\top w_S^{(S,\ell)}}\cdot\sigma_x\\
    &=\sqrt{b}\cdot\sigma_x.
    \end{split}
\end{align}
While for fixed $e$ and $S$, $R^{(e,S)}$ is a sub-Gaussian variable with parameter $\sigma_y(1+\sigma_x\sqrt{b})$, as obtained in \eqref{eq:subgaussian-of-R}. Thus $Z_3(e, S, \ell)$ is the recentered average of independent random variables, each of which is the product of two sub-Gaussian variables with parameters $\sigma_{e,S,\ell}$ and $\sigma_y (1+ \sigma_x \sqrt{b})$.
Then it follows from the tail bound for exponential random variable that
\begin{align*}
    |S|\le s, \ell\in [N_S],\quad\mathbb{P}\left[|Z_3(S,\ell)| \ge C' \sqrt{b} \sigma_x^2 \sigma_y \left(\sqrt{b}\frac{u}{n\cdot |\mathcal{E}|} + \sqrt{\frac{u}{n\cdot |\mathcal{E}|}}\right)\right]\le 2e^{-u},\quad \forall u>0.
\end{align*}
Letting $u=t+\log(2N) \le 6 \left(t + s\log(ed/s)\right)$, we obtain
\begin{align*}
    \mathbb{P}\left[\sup_{|S|\le s, \ell\in [N_S]} |Z_3(S, k)| \ge 6C'  \sigma_x^2 \sigma_y \left(\sqrt{b\zeta(s, t)}+b\zeta(s, t) \right) \right] \le  N\times 2e^{-\log(2N)-t} \le e^{-t}.
\end{align*} Combining with the argument \eqref{eq:proof-sup3} concludes the proof of the claim with $C_1=12C'$.

\noindent {\it High probability error bound in $\mathcal{A}_4(k, t)$. }
Recall that in \cref{lemma:highprob1} we obtain that for any symmetric matrix $Q\in\mathbb{R}^{d\times d}$, $\|Q_S\|_2 \le 3 \sup_{\ell \in [N_S]} (w_S^{(S,\ell)})^\top Q_S w_S^{(S,\ell)}$, by the variational representation of the operator norm. This immediately yields,
\begin{align}
\label{eq:proof-sup4}
\begin{split}
    & \sup_{|S| \le s} \left\| (\Sigma_S)^{-1/2} (\hat{\Sigma}_S) (\Sigma_S)^{-1/2} - I \right\|_2\\
    &\qquad\qquad\qquad\le 3 \sup_{|S|\le s, \ell\in [N_S]}\underbrace{(w_S^{(S,\ell)})^\top\left[ (\Sigma_S)^{-1/2} (\hat{\Sigma}_S) (\Sigma_S)^{-1/2} - I \right](w_S^{(S,\ell)})}_{Z_4(S,\ell)}. 
\end{split}
\end{align}
Note for fixed $S$ and $\ell$, 
\begin{align*}
    Z_4(S,\ell)&=(w_S^{(S,\ell)})^\top\left[ (\Sigma_S)^{-1/2} (\hat{\Sigma}_S) (\Sigma_S)^{-1/2} - I \right](w_S^{(S,\ell)})\\
    &=\frac{1}{n\cdot |\mathcal{E}|}\sum_{e\in \mathcal{E}}\sum_{i=1}^n\left( (w_S^{(S,\ell)})^\top(\Sigma_S)^{-1/2}X_i^{(e)}\right)^2-1,
\end{align*} is the recentered average of independent random variables, each of which is the square of sub-Gaussian variable with parameter $\sigma_{e,S,\ell}=\left((w_S^{(S,\ell)})^\top(\Sigma_S)^{-1/2}(\Sigma_S^{(e)})(\Sigma_S)^{-1/2}w_S^{(S,\ell)}\right)^{1/2}\sigma_x$. We have $\sigma_{e,S,\ell}\le\sqrt{b}\cdot\sigma_x$ as obtained in \eqref{eq:upper-bound-sigma_esl}, and
\begin{align*}
    &\left(\sum_{e\in\mathcal{E}}\sum_{i=1}^n (\sigma_{e,S,\ell})^4\right)^{1/2}\\
    &=
    \left(n\sum_{e\in\mathcal{E}}(\sigma_{e,S,\ell})^4\right)^{1/2}\\
    &\le
    \sqrt{n}\cdot(\max_{e,S,\ell}\sigma_{e,S,\ell})\left(\sum_{e\in\mathcal{E}} \sigma_{e,S,\ell}^2\right)^{1/2}\\
    &\le\sqrt{n}\cdot\sqrt{b}\cdot\sigma_x\left(\sum_{e\in\mathcal{E}} \sigma_{e,S,\ell}^2\right)^{1/2}\\
    &=\sqrt{n}\cdot\sqrt{b}\cdot\sigma_x\left(\sum_{e\in\mathcal{E}}(w_S^{(S,\ell)})^\top(\Sigma_S)^{-1/2}(\Sigma_S^{(e)})(\Sigma_S)^{-1/2}w_S^{(S,\ell)}\right)^{1/2}\sigma_x\\
    &=\sqrt{n}\cdot\sqrt{b}\cdot\sigma_x\left((w_S^{(S,\ell)})^\top(\Sigma_S)^{-1/2}\left(\sum_{e\in\mathcal{E}}\Sigma_S^{(e)}\right)(\Sigma_S)^{-1/2}w_S^{(S,\ell)}\right)^{1/2}\sigma_x
    \\
    &=\sqrt{n}\cdot\sqrt{b}\cdot\sqrt{|\mathcal{E}|}\cdot\sigma_x^2.
\end{align*}
Then it follows from the tail bound for the sub-exponential random variable that
\begin{align*}
    \forall |S|\le s, \ell\in [N_S],\quad
    \mathbb{P}\left[|Z_4( S,\ell)| \ge C'  \sigma_x^2 \left(b\frac{u}{n\cdot |\mathcal{E}|} + \sqrt{b\frac{u}{n\cdot |\mathcal{E}|}}\right)\right]\le 2e^{-u}, \quad \forall u>0.
\end{align*}
Letting $u=t+\log(2N) \le 6 \left(t + s\log(ed/s)\right)$, we obtain
\begin{align*}
    \mathbb{P}\left[\sup_{ |S|\le s, \ell\in[N_S]} |Z_4(S, \ell)| \ge 6C'  \sigma_x^2 \left(\sqrt{b\zeta(s, t)}+b\zeta(s, t) \right) \right] \le  N\times 2e^{-\log(2N)-t} \le e^{-t}.
\end{align*}
Combining with the argument \eqref{eq:proof-sup4} concludes the proof of the claim with $C_1=18C'$.

\subsection{Proof of \cref{lemma:rsc}}








We first introduce some notation and outline the sketch of the proof. For any given fixed $v \in \mathbb{R}^d$, we define the random variables $Z_v$ and $W_v$ be
\begin{align*}
    Z_v = \frac{1}{n \cdot |\mathcal{E}|} \sum_{i\in [n], e\in \mathcal{E}} (v^\top X_i^{(e)})^2 - v^\top \Sigma v \qquad \text{and}\qquad W_v = \sqrt{Z_v + v^\top\Sigma v},
\end{align*} respectively. 
Given any fixed $\alpha>0$, let $s=c(1+\alpha)^{-2} \sigma_x^{-4}{\kappa}\cdot n |\mathcal{E}|/(b\cdot\log d)$ where $c$ is a universal constant. We also define the set
\begin{align*}
\Theta = \Theta_{s,\alpha} := \bigcup_{S\subseteq [d], |S| \le s}\{\theta\in \mathbb{R}^d: \|\theta_{S^c}\|_{1} \le \alpha \| \theta_{S}\|_1\} 
\end{align*} and abbreviate it as $\Theta$ given our analysis focused on any fixed $(\alpha, s(\alpha))$. Note that $\Theta$ is a cone, in the sense that for any $\theta\in\Theta$ and $t>0$ we also have $t\cdot \theta\in \Theta$, and note that the result we want to prove is quadratic in $\theta$ on both sides. Therefore it suffices to consider $\{v\in\Theta: \|v\|_\Sigma=1\}$, and we define the following set 
\newcommand{\surface}{{\mathcal{B}}}
\newcommand{\sigmasurface}{{\mathcal{B}^\Sigma}}
\begin{align*}
    \surface :=\surface_{s,\alpha}:&= \Theta\cap\{\theta\in\mathbb{R}^d:\|\theta\|_\Sigma=1\}\\
    &=\bigcup_{S\subseteq [d], |S| \le s}\{\theta\in \mathbb{R}^d: \|\theta_{S^c}\|_{1} \le \alpha \| \theta_{S}\|_1, \|\theta\|_\Sigma=1 \}
\end{align*}
and abbreviate it as $\surface$. We also define the following metric on $\mathbb{R}^d$ 
\begin{align*}
    \mathsf{d}(v, v') = \sigma_x \|v - v'\|_\Sigma
\end{align*} and simply let $\mathsf{d}(v, \mathcal{T}) = \inf_{a\in \mathcal{T}} \mathsf{d}(v, a)$ for some set $\mathcal{T}$. It suffices to show that there exists some universal constant $C$ such that
\begin{align*}
    \mathbb{P}\left(\inf_{v\in \surface} Z_v+1 \ge \frac{1}{2}\right) \ge 1-3\exp(-\tilde{n}/(C\sigma_x)^4),
\end{align*}
where 
\begin{align*}
    \tilde{n}:=\frac{n\cdot|\mathcal{E}|}{b}.
\end{align*}
It is obvious that $\tilde{n}\ge n$ follows from $b\le|\mathcal{E}|$ derived in \eqref{eq:b<E}. Our proof is divided into three steps.

In the first step, we establish concentration inequalities for any fixed $v$ and $v'$. To be specific, we show that for some universal constant $C>0$, the following holds: for any $t>0$,
\begin{align}
&\mathbb{P}\left[|Z_v-Z_{v'}|>C\mathsf{d}(v,-v')\mathsf{d}(v,v')\left(\sqrt{\frac{t}{\tilde{n}}}+\frac{t}{\tilde{n}}\right)\right]\le 2e^{-t} ,\quad &{\forall }v,v'\in \mathbb{R}^d;\label{eq:ZDconcentration}\\
&\mathbb{P}\left[|Z_v|>C\sigma_x^2\left(\sqrt{\frac{t}{\tilde{n}}}+\frac{t}{\tilde{n}}\right)\right]\le 2e^{-t},\quad &\forall v\in \surface;\label{eq:Zconcentration}\\
    &\mathbb{P}\left[W_{v}> C\mathsf{d}(v,0)\left(\sqrt{\frac{t}{\tilde{n}}}+1\right)\right]\le  2e^{-t},\quad &\forall v\in \mathbb{R}^d.\label{eq:WDconcentration}
\end{align}

In the second step, we establish an upper bound on the Talagrand's $\gamma_2$ functional~\citep{vershynin2018high} of $\Theta$, which is defined as
\begin{align}\label{eq:def-talagrand-functional}
    \gamma_2(\Theta, \mathsf{d}) := \inf_{\{\surface_k\}_{k=0}^\infty: |\surface_0|=1, |\surface_k| \le 2^{2^k}} \sup_{v\in \Theta} \sum_{k=0}^\infty 2^{k/2} \mathsf{d}(v, \surface_k).
\end{align}
To be specific, we show that
\begin{align}
        \label{eq:gamma_2_bound}\gamma_2(\surface,\mathsf{d})\le C\sigma_x(1+\alpha)\kappa^{-1/2}\sqrt{s\log d}
    \end{align}
where $C$ is a universal constant.


Finally, we combine the concentration inequalities and the complexity measure $\gamma_2(\Theta, \mathsf{d})$ to bound the supremum $\sup_{v\in \Theta} |Z_v|$. Specifically, we show that if  $\tilde{n}^{1/2}\ge C\sigma_x \gamma_2(\surface,\mathsf{d})$
, then
\begin{align}
    \label{eq:zv-upper-bound}\mathbb{P}\left[\sup_{v\in \surface}|Z_v|> 1/2\right]\le 3\exp\left(-\tilde{n}/(C\sigma_x)^4\right).
\end{align}
\noindent {\sc Step 1. Establish Concentration Inequalities for Fixed $v$.} In this step we prove the concentration inequalities \eqref{eq:ZDconcentration},\eqref{eq:Zconcentration} and \eqref{eq:WDconcentration}.
For \eqref{eq:ZDconcentration}, it follows from the definition of $Z$ that
\begin{align*}
    Z_v-Z_{v'}&=\frac{1}{n \cdot |\mathcal{E}|} \sum_{i\in [n], e\in \mathcal{E}} \left((v^\top X_i^{(e)})^2 -((v')^\top X_i^{(e)})^2 \right) - \left(v^\top\Sigma v-v'^\top\Sigma v'\right)\\
    &=\frac{1}{n\cdot|\mathcal E|}\sum_{i\in [n], e\in \mathcal{E}}
\left((v+v')^\top X_i^{(e)}\right)\cdot\left((v-v')^\top X_i^{(e)}\right)-\left(v^\top\Sigma v-v'^\top\Sigma v'\right).
\end{align*}
It is the recentered average of independent random variables, each of which is the product of two sub-Gaussian variables with parameter $\sigma_{e,v+v'}$ and $\sigma_{e,v-v'}$ satisfying
\begin{align*}
    \sigma_{e,v+v'}\sigma_{e,v-v'}&\overset{(a)}{\le} (\sigma_x\|v+v'\|_{\Sigma^{(e)}})\cdot(\sigma_x\|v-v'\|_{\Sigma^{(e)}})\\
    &\overset{(b)}{\le} \sqrt{b}\cdot\mathsf{d}(v,-v')\cdot\sigma_x\|v-v'\|_{\Sigma^{(e)}}\\
    &\overset{(c)} {\le} b\cdot\mathsf{d}(v,-v')\mathsf{d}(v,v');\\
    \sum_{i\in [n], e\in \mathcal{E}} (\sigma_{e,v+v'}\sigma_{e,v-v'})^2&\overset{(d)}{\le} 
    \sum_{i\in [n], e\in \mathcal{E}}\left(b\cdot\mathsf{d}(v,-v')^2\cdot\sigma_x^2\cdot\|v-v'\|_{\Sigma^{(e)}}^2\right)\\
    &\overset{(e)}{=} n|\mathcal{E}|b\mathsf{d}(v,-v')^2\sigma_x^2\|v-v'\|_{\Sigma}^2\\
    &= n|\mathcal{E}|b\mathsf{d}(v,-v')^2\mathsf{d}(v,v')^2.
\end{align*}
Here $(a)$ follows from the data generating process \cref{cond:regularity}(c); $(b)$ and $c$ follow from the fact that $\|v\|_{\Sigma^{(e)}} \le \sqrt{b} \|v\|_{\Sigma}$ by $\lambda_{\max}(\Sigma^{-1/2}\Sigma^{(e)}\Sigma^{-1/2})\le b$; $(d)$ follows directly from $(b)$. and $(e)$ follows from $\frac{1}{|\mathcal{E}|}\sum_{e\in\mathcal{E}}\|\cdot\|_{\Sigma^{(e)}}^2=\|\cdot\|_{\Sigma}^2$ since $|\mathcal{E}|\cdot\Sigma=\sum_{e\in\mathcal{E}}\Sigma^{(e)}$.
Using \cref{lemma:product_of_two_subGaussian} and \cref{lemma:Bernstein}, we can obtain that for all $v,v'\in\mathbb{R}^d$ and $t>0$,
\begin{align*}
\mathbb{P}\left[|Z_v-Z_{v'}|>C\mathsf{d}(v,-v')\mathsf{d}(v,v')\left(\sqrt{\frac{t}{\tilde{n}}}+\frac{t}{\tilde{n}}\right)\right]\le 2e^{-t}
\end{align*}
for some universal constant $C>0$. This completes the proof of \eqref{eq:ZDconcentration}.

\eqref{eq:Zconcentration} is a corollary of \eqref{eq:ZDconcentration}, following from assigning $v'=0$ and noticing that $\mathsf{d}(v,0)=\sigma_x$ for all $v\in\surface$.

For \eqref{eq:WDconcentration}, observe that, $W_v^2=Z_v+v^\top\Sigma v$. Combining \eqref{eq:ZDconcentration} we can conclude that for all $v\in\mathbb{R}^d$ and $t>0$,
\begin{align*}
    &\mathbb{P}\left[W_v>\sqrt{C}\mathsf{d}(v,0)\left(\sqrt{\frac{t}{\tilde{n}}}+1\right)\right]\\
    &=\mathbb{P}\left[W_v^2>C\mathsf{d}(v,0)^2\left(2\sqrt{\frac{t}{\tilde{n}}}+\frac{t}{\tilde{n}}+1\right)\right]\\
    &\overset{(a)}{\le} \mathbb{P}\left[W_v^2>C\mathsf{d}(v,0)^2\left(\sqrt{\frac{t}{\tilde{n}}}+\frac{t}{\tilde{n}}\right)+v^\top\Sigma v\right]\\
    &=\mathbb{P}\left[Z_v>C\mathsf{d}(v,0)^2\left(\sqrt{\frac{t}{\tilde{n}}}+\frac{t}{\tilde{n}}\right)\right]\\
    &\le 2e^{-t}.
\end{align*} Here in $(a)$ we use the fact that $C>1$ and that $\mathsf{d}(v,0)^2=\sigma_x^2 v^\top\Sigma v\ge v^\top\Sigma v$ since $\sigma_x\ge 1$. 

\noindent {\sc Step 2. Bounding the $\gamma_2$-functional.} In this step we prove \eqref{eq:gamma_2_bound}.
We define another set
\begin{align*}
    \sigmasurface&:=\Sigma^{1/2}\surface=\{x\in \mathbb{R}^d: \|x\|_2=1, \Sigma^{-1/2}x\in \Theta\}.
\end{align*}
Since $(\surface,\mathsf{d})$ is isometric to $(\sigmasurface,\sigma_x\|\cdot\|_2)$ and $(\sigmasurface,\sigma_x\|\cdot\|_2)$ is isometric to $(\sigma_x\sigmasurface,\|\cdot\|_2)$. From the fact that $\gamma_2$ functional is invariant under isometries, we have
\begin{align}
\label{eq:Tala-under-iso}
    \gamma_2(\surface,\mathsf{d})= \gamma_2(\sigmasurface,\sigma_x\|\cdot\|_2)=\gamma_2(\sigma_x \sigmasurface,\|\cdot\|_2).
\end{align}
Also, the $\gamma_2$ functional respects scaling in the sense that
\begin{align}
    \label{eq:gamma2-scaling}
    \gamma_2(\sigma_x \sigmasurface,\|\cdot\|_2)=\sigma_x\gamma_2(\sigmasurface,\|\cdot\|_2).
\end{align}
Additionally, it follows from Talagrand’s majorizing measure theorem~\citep{talagrand2005generic} that there exists some universal constant $C>0$ such that
\begin{align}
\label{eq:Tala-comparison}
    \gamma_2(\sigmasurface,\|\cdot\|_2)\le C\cdot \mathbb{E}_{g\sim N(0,I_d)} \left[\sup_{x\in \sigmasurface}g^\top x\right].
\end{align}
So, it remains to obtain an upper bound the right-hand side as follows: 
    \begin{align*}
    \mathbb{E}_{g\sim N(0,I_d)} \left[\sup_{x\in \sigmasurface}g^\top x\right] 
        &= \mathbb{E}_{g\sim N(0,I_d)} \left[\sup_{x\in \sigmasurface}(\Sigma^{1/2}g)^\top(\Sigma^{-1/2}x)\right]\\
        &\le \sup_{x\in \sigmasurface}\|\Sigma^{-1/2}x\|_1\cdot\mathbb{E}_{g\sim N(0,I_d)}[\|\Sigma^{1/2}g\|_\infty]\\
        &\overset{(a)}{\le} (1+\alpha)\sqrt{s}\sup_{x\in \sigmasurface}\|\Sigma^{-1/2}x\|_2\cdot\mathbb{E}_{g\sim N(0,I_d)}[\|\Sigma^{1/2}g\|_\infty]\\
        &\overset{(b)}{\le} (1+\alpha)\sqrt{s}\kappa^{-1/2}\mathbb{E}_{g\sim N(0,I_d)}[\|\Sigma^{1/2}g\|_\infty]\\
        &\overset{(c)}{\le}
        (1+\alpha)\sqrt{s}\kappa^{-1/2}\cdot50\sqrt{\log d}.
    \end{align*}
    Here $(a)$ follows from the fact that $\Sigma^{-1/2}x\in\surface$ and for any $v\in\surface$, we have 
    \begin{align*}
        \|v\|_1= \|v_{S}\|_1+\|v_{S^c}\|_1\le(1+\alpha)\|v_S\|_1\le (1+\alpha)\sqrt{s}\|v_{S}\|_2\le(1+\alpha)\sqrt{s}\|v\|_2 
    \end{align*} for some subset $|S|\le s$ by the definition of $\Theta$; $(b)$ follows from $\|\Sigma^{-1/2}x\|_2\le \kappa^{-1/2}\|x\|_2= \kappa^{-1/2}$; and $(c)$ follows from $\mathbb{E}_{g\sim N(0,I_d)}[\|\Sigma^{1/2}g\|_\infty]\le \mathbb{E}_{g\sim N(0,I_d)}[\|g\|_\infty]\le 50\sqrt{\log d}$ by Sudakov-Fernique’s inequality~\citep{fernique1975regularite} and \cref{cond:regularity}(b). Combining with \eqref{eq:Tala-under-iso}, \eqref{eq:Tala-comparison} and \eqref{eq:gamma2-scaling}, we complete the proof of \eqref{eq:gamma_2_bound}.

\noindent {\sc Step 3. Bounding the maximum of $|Z_{v}|$:}
In this step, we prove \eqref{eq:zv-upper-bound} following \citet{mendelson2007reconstruction}.
It follows from the definition of $\gamma_2$-functional that there exists a sequence of subsets $\{\surface_k: k \ge 0\}$ of $\surface$ with $|\surface_0|=1$ and $|\surface_k| \le 2^{2^k}$ such that for every $v \in \surface$,
\begin{align*}
    \sum_{k=0}^\infty
2^{k/2} \mathsf{d}(v,\pi_{k}(v)) \le 1.01 \gamma_2(\surface,  \mathsf{d} ),
\end{align*}
where $\pi_{k}(v)$ denotes the nearest element of $v$ in $\surface_k$. This immediately implies
\begin{align}
\label{eq:sum-of-distance-by-tala}
\begin{split}
    \sum_{k=0}^\infty
2^{k/2} \mathsf{d}(\pi_{k+1}(v),\pi_{k}(v))
&\le \sum_{k=0}^\infty
2^{k/2}\left(\mathsf{d}(v,\pi_{k}(v))+\mathsf{d}(v,\pi_{k+1}(v))\right)\\
&\le (1+2^{-1/2})\sum_{k=0}^\infty
2^{k/2} \mathsf{d}(v,\pi_{k}(v))\\
&\le 2 \gamma_2(\surface,  \mathsf{d} ).
\end{split}
\end{align}
Let the integer $k_0$ satisfy $2\tilde{n} \ge 2^{k_0} > \tilde{n}$. It follows from triangle inequality and the definition of $W_v$ and $Z_v$ that
\begin{align}
    \label{eq:separate-Zv}
    |Z_v|\le |Z_v-Z_{\pi_{k_0}(v)}|+|Z_{\pi_{k_0}(v)}|= |W_{v}^2-W_{\pi_{k_0}(v)}^2|+|Z_{\pi_{k_0}(v)}|.
\end{align}
From Minkowski's inequality, we can observe that $W_v$ is sub-additive with respect to $v$, that is, for any $v_1,v_2\in\mathbb{R}^d$, 
\begin{equation}\label{eq:sub-additive-of-W0}
\begin{split}
    W_{v_1+v_2}&=\left[{\frac{1}{n \cdot |\mathcal{E}|} \sum_{i\in [n], e\in \mathcal{E}} \left(v_1^\top X_i^{(e)}+v_2^\top X_i^{(e)}\right)^2}\right]^{1/2}\\
    &\le
    \left[{\frac{1}{n \cdot |\mathcal{E}|} \sum_{i\in [n], e\in \mathcal{E}} \left(v_1^\top X_i^{(e)}\right)^2}\right]^{1/2}+\left[{\frac{1}{n \cdot |\mathcal{E}|} \sum_{i\in [n], e\in \mathcal{E}} \left(v_2^\top X_i^{(e)}\right)^2}\right]^{1/2}\\
    &=W_{v_1}+W_{v_2}.
\end{split}
\end{equation}
This helps us to obtain
\begin{align*}
    (W_{\pi_{k_0}(v)}-W_{v-\pi_{k_0}(v)})^2-W_{\pi_{k_0}(v)}^2\le W_{v}^2-W_{\pi_{k_0}(v)}^2&\le(W_{\pi_{k_0}(v)}+W_{v-\pi_{k_0}(v)})^2-W_{\pi_{k_0}(v)}^2.
\end{align*}
Then we can derive
\begin{align}
    \label{eq:separate-Zv-2}
   |W_{v}^2-W_{\pi_{k_0}(v)}^2|\le W_{v-\pi_{k_0}(v)}^2+2W_{v-\pi_{k_0}(v)}W_{\pi_{k_0}(v)}.
\end{align}
Therefore, combining \eqref{eq:separate-Zv} and \eqref{eq:separate-Zv-2}, and letting the positive integer $k_1< k_0$ be determined later, we can upper bound $\sup_{v \in \surface}|Z_v|$ as follows
\begin{equation}
\begin{split}
    \sup_{v \in \surface} |Z_v| &\le \sup_{v \in \surface} W_{v-\pi_{k_0}(v)}^2
    + 2\sup_{v \in \surface} W_{v-\pi_{k_0}(v)}  \sup_{{v_0} \in \surface_{k_0}} W_{{v_0}}
     + \sup_{{v_0} \in \surface_{k_0}}|Z_{v_0}|\\
     &= \sup_{v \in \surface} W_{v-\pi_{k_0}(v)}^2
    + 2\sup_{v \in \surface} W_{v-\pi_{k_0}(v)}  \sup_{{v_0} \in \surface_{k_0}} \sqrt{Z_{{v_0}}+1}
     + \sup_{{v_0} \in \surface_{k_0}}|Z_{v_0}|\\
     &\le \sup_{v \in \surface} W_{v-\pi_{k_0}(v)}^2
    + 2\sup_{v \in \surface} W_{v-\pi_{k_0}(v)}  \sup_{{v_0} \in \surface_{k_0}}\left( |Z_{{v_0}}|+1\right)
     + \sup_{{v_0} \in \surface_{k_0}}|Z_{v_0}|\\
     &= \sup_{v \in \surface} W_{v-\pi_{k_0}(v)}^2+2\sup_{v \in \surface} W_{v-\pi_{k_0}(v)}+\left(2\sup_{v \in \surface} W_{v-\pi_{k_0}(v)}+1\right)\sup_{{v_0} \in \surface_{k_0}}|Z_{v_0}|\\
     &\le \sup_{v \in \surface} W_{v-\pi_{k_0}(v)}^2+2\sup_{v \in \surface} W_{v-\pi_{k_0}(v)}\\
     &\qquad\qquad+\left(2\sup_{v \in \surface} W_{v-\pi_{k_0}(v)}+1\right)\left(\sup_{v_0\in\surface_{k_0}}|Z_{v_0}-Z_{\pi_{k_1}(v_0)}|+\sup_{v_1\in\surface_{k_1}}|Z_{v_1}|\right).\\ 
\end{split}
  \label{eq:Z_v-three-terms}
  \end{equation}
Then it remains to upper bound $\sup_{v \in \surface} W_{v-\pi_{k_0}(v)}$, $\sup_{v_0\in\surface_{k_0}}|Z_{v_0}-Z_{\pi_{k_1}(v_0)}|$ and $\sup_{v_1\in\surface_{k_1}}|Z_{v_1}|$.

First, we upper bound $\sup_{v \in \surface} W_{v-\pi_{k_0}(v)}$. It follows from the sub-additivity of $W_v$ that
\begin{align}\label{eq:sub-add_of_W}
    W_{v-\pi_{k_0}(v)} \le \sum_{k=k_0}^\infty
W_{\pi_{k+1}(v)-\pi_k(v)},\qquad \forall v\in\surface.
\end{align}
For each $k\ge k_0$, we define the following event
\begin{align*}
    \mathcal{U}_1(k)=\left\{\sup_{v\in \surface} W_{\pi_{k+1}(v)-\pi_k(v)} \le 8C
\sqrt{{2^{k}}/{\tilde{n}}}\cdot\mathsf{d}\left(\pi_{k+1}(v),\pi_k(v)\right) \right\}
\end{align*}
where the constant $C$ is the same as that in \eqref{eq:WDconcentration}. Since $|\surface_k| \le 2^{2^k}$, there are at most
$2^{2^k}\times2^{2^{k+1}}\le 2^{2^{k+2}}$ distinct pairs of $(\pi_{k+1}(v),\pi_k(v))$. Thus, we can take a union bound over all such pairs, combine with \eqref{eq:WDconcentration} and use the fact that $2^k>\tilde{n}$ provided $k\ge k_0$ to obtain
\begin{equation}
    \begin{split}
        &\mathbb{P}\left(\overline{\mathcal{U}_1(k)}\right)\\
        &\le\sum_{(\pi_{k+1}(v), \pi_k(v))}\mathbb{P}\left[W_{\pi_{k+1}(v)-\pi_k(v)} > 8C
\sqrt{{2^{k}}/{\tilde{n}}}\cdot\mathsf{d}\left(\pi_{k+1}(v),\pi_k(v)\right)\right]\\
&\le\sum_{(\pi_{k+1}(v), \pi_k(v))}\mathbb{P} \left[W_{\pi_{k+1}(v)-\pi_k(v)} > C
\left(\sqrt{{16\cdot2^{k}}/{\tilde{n}}}+1\right)\cdot\mathsf{d}\left(\pi_{k+1}(v)-\pi_k(v),0\right)\right]\\
 &\le 2^{2^{k+2}}\cdot 2 \exp(-16\cdot 2^{k})\le \exp (-8\cdot 2^{k}).
    \end{split}\label{eq:u1-prob}
\end{equation}
Under the event $\bigcap_{k\ge k_0} \mathcal{U}_1(k)$, it follows from \eqref{eq:sum-of-distance-by-tala} and \eqref{eq:sub-add_of_W} that
\begin{equation}
\label{eq:Zv-part1-upperbound}
    \begin{split}
        \sup_{v \in \surface} W_{v-\pi_{k_0}(v)}\le \sum_{k=k_0}^\infty
\sup_{v\in\surface} W_{\pi_{k+1}(v)-\pi_k(v)}&\le 8C{\tilde{n}}^{-1/2}\sum_{k=k_0}^\infty 2^{k/2}\mathsf{d}(\pi_{k+1}(v),\pi_k(v))\\
&\le 16C{\tilde{n}}^{-1/2}\gamma_2(\surface,\mathsf{d}).
    \end{split}
\end{equation}

For $\sup_{v_0\in\surface_{k_0}}|Z_{v_0}-Z_{\pi_{k_1}(v_0)}|$, we first define the following event for each $0\le k \le k_0-1$,
\begin{align*}
    \mathcal{U}_2(k)=\left\{\sup_{v\in \surface} \left|Z_{\pi_{k+1}(v)}-Z_{\pi_k(v)}\right| \le C\cdot40\sigma_x
\mathsf{d}(\pi_{k+1}({v}),\pi_k({v}))\sqrt{{2^k }/\tilde{n}} \right\}
\end{align*}
where the constant $C$ is the same as that in \eqref{eq:ZDconcentration}. Since $|\surface_k| \le 2^{2^k}$, there are at most
$2^{2^k}\times2^{2^{k+1}}\le 2^{2^{k+2}}$ distinct pairs of $(\pi_{k+1}(v),\pi_k(v))$. Thus, we can take a union bound over all such pairs, combine with \eqref{eq:ZDconcentration} and use the fact that $2^k\le\tilde{n}$ provided $k\le k_0-1$ to obtain
\begin{equation}
    \begin{split}
        &\mathbb{P}\left(\overline{\mathcal{U}_2(k)}\right)\\
        &\le\sum_{(\pi_{k+1}(v), \pi_k(v))}\mathbb{P}\left[\left|Z_{\pi_{k+1}({v})}-Z_{\pi_{k}({v})}\right| >C\cdot2\sigma_x
\mathsf{d}(\pi_{k+1}({v}),\pi_k({v}))\cdot20\sqrt{{2^k }/\tilde{n}}\right]\\
&\le\sum_{(\pi_{k+1}(v), \pi_k(v))}\mathbb{P} \bigg[\left|Z_{\pi_{k+1}({v})}-Z_{\pi_{k}({v})}\right| > C\mathsf{d}\left(\pi_{k+1}({v}),-\pi_k({v})\right)
\mathsf{d}(\pi_{k+1}({v}),\pi_k({v}))\\
&\qquad\qquad\qquad\qquad\qquad\qquad\qquad\qquad\qquad\times \left(\sqrt{{(16\cdot 2^k )}/\tilde{n}}+(16\cdot2^k)/\tilde{n}\right)\bigg]\\
 &\le 2^{2^{k+2}}\cdot 2 \exp(-16\cdot 2^{k})\le \exp (-8\cdot 2^{k}).
    \end{split}\label{eq:u2-prob}
\end{equation}
From triangle inequality, under the event $\bigcap_{k=k_1}^{k_0-1}\mathcal{U}_2(k)$, we have
\begin{equation}\label{eq:Zv-part2-upperbound}
    \begin{split}
        \sup_{v_0\in\surface_{k_0}}|Z_{v_0}-Z_{\pi_{k_1}(v_0)}|&\le \sum_{k=k_1}^{k_0-1}\sup_{v\in\surface}|Z_{\pi_{k+1}(v)}-Z_{\pi_{k}(v)}|\\
    &\le 40C\sigma_x\tilde{n}^{-1/2}\sum_{k=k_1}^{k_0-1}2^{k/2}\mathsf{d}(\pi_{k+1}(v),\pi_k(v))\\
    &\le 40C\sigma_x\tilde{n}^{-1/2}2\gamma_2(\surface,\mathsf{d}).
    \end{split}
\end{equation}

For $\sup_{v_1\in\surface_{k_1}}|Z_{v_1}|$, we define the following event for each $0\le k\le k_0-1$, 
\begin{equation}\label{eq:u3-def}
    \begin{split}
        \mathcal{U}_3(k)=\left\{\sup_{v\in \surface} \left|Z_{\pi_k(v)}\right| \le C\sigma_x^2
\cdot32\sqrt{{2^k}/\tilde{n}} \right\}
    \end{split}
\end{equation}
where the constant $C$ is the same as that in \eqref{eq:Zconcentration}.
We take a union bound over all elements in $\surface_{k}$, combine with \eqref{eq:Zconcentration} and use the fact that $2^k\le\tilde{n}$ to obtain
\begin{equation}
    \begin{split}
            \mathbb{P}(\overline{\mathcal{U}_3(k)})
    &\le\sum_{\pi_k(v)}\mathbb{P} \left[\left|Z_{\pi_{k}({v})}\right| >C\sigma_x^2
\cdot32\sqrt{{2^k}/\tilde{n}}\right] \\
&\le\sum_{\pi_k(v)}\mathbb{P} \left[\left|Z_{\pi_{k}({v})}\right| >C\sigma_x^2\left(\sqrt{{(16\cdot 2^k )}/\tilde{n}}+(16\cdot2^k)/\tilde{n}\right)\right]
\\
    &\le 2^{2^{k}}\cdot 2\exp(-16\cdot 2^k)\le \exp\left(-8\cdot2^k\right).
    \end{split}\label{eq:u3-prob}
\end{equation}

Now we choose $k_1$ such that $\tilde{n}/(2^{23}C^2\sigma_x^4)\le 2^{k_1}<\tilde{n}/(2^{22}C^2\sigma_x^4)$. Then combining with \eqref{eq:Zv-part1-upperbound}, \eqref{eq:Zv-part2-upperbound} and \eqref{eq:u3-def}, there exists a universal constant $C'$ such that, provided $\tilde{n}^{1/2}\ge C'\sigma_x\gamma_2(\surface,\mathsf{d})$, the following holds under the event $\left(\bigcap_{k=k_0}^{\infty}\mathcal{U}_1(k)\right)\cap\left(\bigcap_{k=k_1}^{k_0-1}\mathcal{U}_2(k)\right)\cap\mathcal{U}_3(k_1)$
\begin{align*}
    \sup_{v\in\surface} |Z_{v}|&\le \sup_{v \in \surface} W_{v-\pi_{k_0}(v)}^2+2\sup_{v \in \surface} W_{v-\pi_{k_0}(v)}\\
     &\qquad\qquad+\left(2\sup_{v \in \surface} W_{v-\pi_{k_0}(v)}+1\right)\left(\sup_{v_0\in\surface_{k_0}}|Z_{v_0}-Z_{\pi_{k_1}(v_0)}|+\sup_{v_1\in\surface_{k_1}}|Z_{v_1}|\right)\\
     &\le \left(\frac{1}{64}\right)^2+2\cdot \frac{1}{64}+\left(2\cdot\frac{1}{64}+1\right)\left(\frac{1}{64}+\frac{1}{64}\right)<\frac12.
\end{align*}
Therefore, combine this with \eqref{eq:u1-prob}, \eqref{eq:u2-prob} and \eqref{eq:u3-prob}, we can conclude that the event
\begin{align*}
    \left\{\sup_{v\in\surface} |Z_{v}|<\frac12\right\}
\end{align*}
occurs with probability at least
\begin{align*}
    &1-\sum_{k=k_0}^\infty\mathbb{P}\left[\overline{\mathcal{U}_1(k)}\right]-\sum_{k=k_1}^{k_0-1}\mathbb{P}\left[\overline{\mathcal{U}_2(k)}\right]-\mathbb{P}\left[\overline{\mathcal{U}_3(k_1)}\right]\\
    &\ge1-\sum_{k=k_0}^\infty \exp (-8\cdot 2^{k})-\sum_{k=k_0}^\infty \exp(-8\cdot 2^k)-\exp\left(-8\cdot2^{k_1}\right)\\
    &\ge 1-2\exp(-4\cdot 2^{k_0})-\exp\left(-8\cdot2^{k_1}\right)\\
    &\ge 1-3\exp\left(-4\cdot2^{k_1}\right)\\
    &\ge 1-3\exp\left(-\tilde{n}/(C'\sigma_x^4)\right).
\end{align*}

With these results, we are ready to prove \cref{lemma:rsc}.
\begin{proof}[Proof of \cref{lemma:rsc}]
    Combining the results in Step~2 and Step~3, we can conclude that if $\tilde{n}\ge C \sigma_x^4(1+\alpha)^2\kappa^{-1} s\log d$, i.e. $ s\le C^{-2} (1+\alpha)^{-2}\sigma_x^{-4} {\kappa}\cdot{n|\mathcal{E}|}/{(b\log d)}$ where $C$ is a universal constant, then
\begin{align*}
\mathbb{P}\left(\inf_{v\in \surface} Z_v+1 \ge \frac{1}{2}\right)\ge\mathbb{P}\left(\sup_{v\in\surface}\left| Z_v\right|<1/2\right)\ge 1-3\exp(-\tilde{n}/(C\sigma_x)^4).
\end{align*}
    Therefore with probability over $1-3\exp(-\tilde{n}/(C\sigma_x)^4)$, the following holds: for all $\theta \in \mathbb{R}^d\setminus\{0\}$:
    \begin{align*}
    \frac{1}{|\mathcal{E}|} \sum_{e\in \mathcal{E}} \hat{\mathbb{E}}[|\theta^\top X^{(e)} |^2]=Z_{\theta}+\theta^\top\Sigma\theta=\|\theta\|_\Sigma^2 (Z_{\theta/\|\theta\|_{\Sigma}}+1) \ge 0.5\|\theta\|_\Sigma^2\ge 0.5  \kappa\|\theta\|_2^2.
\end{align*}
\end{proof}

\subsection{Proof of \cref{lemma:shrinkage}}

    The R.H.S. of the inequality follows from the fact that the augmented covariance matrix $\begin{bmatrix} \Sigma & u \\
    u^\top &\sigma_y^2
    \end{bmatrix}$ is the positive semi-definite matrix and thus 
    For the L.H.S., we apply the proof-by-contradiction argument. To be specific, we will show that if $\|\Sigma^{1/2} \beta^{k,\gamma}\|_2 > \|\Sigma^{1/2} \bar{\beta} \|_2$, then $\beta^{k,\gamma}$ will not be the unique minimizer of $\mathsf{Q}_{k,\gamma}(\beta)$, which is contrary to the claim in \cref{thm:minimax-k}. To see this, let 
    \begin{align}
    \label{eq:proj}
        \tilde{\beta} = \argmin_{\beta = t \cdot \beta^{k,\gamma}, t\in \mathbb{R}} \|\Sigma^{1/2}( \beta - \bar{\beta})\|_2 \qquad \text{with} \qquad \bar{\beta} = \Sigma^{-1} u.
    \end{align} Observe that $\tilde{\beta}$ is the projection on the subspace $\{t\cdot \beta^{k,\gamma}: t\in \mathbb{R}\}$ with respect to $\|\Sigma^{1/2} \cdot \|_2$ norm, this implies that
    \begin{align}
    \label{eq:pf-identity}
        (\tilde{\beta} - \bar{\beta})^\top \Sigma v = 0 \qquad \forall v\in \{t\cdot \beta^{k,\gamma}: t\in \mathbb{R}\}.
    \end{align} Then we can obtain that
    \begin{align*}
        \|\Sigma^{1/2} \tilde{\beta} \|_2 = \frac{\tilde{\beta}^\top \Sigma \tilde{\beta}}{\|\Sigma^{1/2} \tilde{\beta} \|_2} \overset{(a)}{=} \frac{\bar{\beta}^\top \Sigma \tilde{\beta}}{\|\Sigma^{1/2} \tilde{\beta} \|_2} &\overset{(b)}{\le} \frac{\|\Sigma^{1/2} \bar{\beta} \|_2 \|\Sigma^{1/2} \tilde{\beta} \|_2}{\|\Sigma^{1/2} \tilde{\beta}\|_2} \\
        &= \|\Sigma^{1/2} \bar{\beta} \|_2 \overset{(c)}{<}  \|\Sigma^{1/2} {\beta}^{k,\gamma} \|_2, 
    \end{align*} which means $\tilde{\beta} = \tilde{t} \cdot \beta^{k,\gamma}$ with $|\tilde{t}| < 1$ because $\lambda_{\min}(\Sigma) >0$. Here $(a)$ we set $v=\tilde{\beta}$ in \eqref{eq:pf-identity}; $(b)$ follows from Cauchy-Schwarz inequality;  and $(c)$ follows from our assumption $\|\Sigma^{1/2} \beta^{k,\gamma}\|_2 > \|\Sigma^{1/2} \bar{\beta} \|_2$. Therefore, we have
    \begin{align*}
        \mathsf{Q}_{k,\gamma}(\tilde{\beta}) - \mathsf{Q}_{k,\gamma}(\beta^{k,\gamma}) &= \|\Sigma^{1/2}(\tilde{\beta} - \bar{\beta})\|_2^2 - \|\Sigma^{1/2}(\beta^{k,\gamma} - \bar{\beta})\|_2^2 + \gamma \sum_{j=1}^n (|\tilde{\beta}_j| - |\beta^{k,\gamma}|) w_k(j) \\
        &\overset{(a)}{\le} 0 + \gamma \sum_{j=1}^n (|\tilde{\beta}_j| - |\beta^{k,\gamma}_j|) w_k(j) \overset{(b)}{<} 0,
    \end{align*} where $(a)$ follows from the minimization program in \eqref{eq:proj}, $(b)$ follows from $\tilde{\beta} = \tilde{t} \cdot \beta^{k,\gamma}$ with $|\tilde{t}| < 1$. This is contrary to the fact that $\beta^{k,\gamma}$ uniquely minimize $\mathsf{Q}_{k,\gamma}(\beta)$. Then we can conclude that $\|\Sigma^{1/2} \beta^{k,\gamma}\|_2 \le \|\Sigma^{1/2} \bar{\beta} \|_2\le\sigma_y$. \qquad

\section{Implementation Details and Omitted Results in Experiments}\label{sec:implement-detail}
In this section we elaborate more on the implementation details.

\subsection{Pre-Processing in Climate Dynamic Prediction}\label{sec:pre-processing}
For Climate Dynamic Prediction, we follow the approach of \citet{runge2015identifying} and conduct preprocessing as follows. For each of the four tasks, we perform the cosine transform on the grid dat. Specifically, for a measurement $x$ at a grid with latitude $\phi \in [-\pi, \pi]$, we apply the following transformation:
\begin{align*}
    x_{\cos} = x * \sqrt{\cos(\phi)}.
\end{align*}
The cosine transform compensates for the varying areas that grids at different latitudes represent, helping to avoid over-compression or over-amplification of grids at higher latitudes.
Next, we estimate the covariance matrix on the training data and compute the eigenvectors, which are then rotated using the Varimax~\citep{kaiser1958varimax,vejmelka2015non} criterion. We select $ N = 60 $ top significant components based on a comparison of the eigenvalues of the original data with those of the surrogate data that only represent the autocorrelation structure.
Finally, for each task, the component weight matrix computed from the training dataset is multiplied with cosine-transformed daily-gridded time series. The resulting product is then normalized to have zero mean and unit variance based on the training data.

\subsection{Construction of the Target Variables in Climate Dynamic Prediction}\label{sec:construction-of-target}
For each task $ a\in\{\texttt{air}, \texttt{csulf},\texttt{slp},\texttt{pres}\}$, we use $X_{t}$ to regress $Z_{t,j}$ tentatively for each $j\in[60]$ on all training data and evaluate $R^2$, which is defined as
\begin{align*}
    1-\frac{\sum_{(X_t,Z_{t,j})\in\mathcal{D}_1\cup \mathcal{D}_2}(Z_{t,j}-\hat{Y}(X_t))^2}{\sum_{(X_t,Z_{t,j})\in\mathcal{D}_1\cup \mathcal{D}_2}(Z_{t,j})^2}.
\end{align*}
We add $j$ to the set of target variables $\mathcal{Y}$ if $R^2$ exceeds a predefined threshold. We set the threshold to $0.75$ for \texttt{a
ir} and \texttt{csulf}, and $0.9$ for \texttt{pres} and \texttt{slp}. 
The selected target variables are shown in \cref{table: climate-target}. We do so since we only care about those target variables that have strong correlations with explain variables.
We use the same hyper-parameters when predicting multiple targets, while different tasks do not share the same hyper-parameters.

\subsection{The Procedure of Applying PCMCI$^+$ or Granger Causality}\label{sec:procedure-pcmci}
When applying PCMCI$^+$ or Granger causality in \cref{sec:stock} and \cref{sec:climate}, we perform the analysis on the entire training data and use a significance level of $\alpha = 0.01$. We fix the set of selected covariates and then conduct $100$ random trials, with the $L_2$ regularization parameter set to $0.1$.

\begin{table}[htb!]
    \centering
    \footnotesize
    \begin{tabular}{cccccc}
    \hline
        Data & \texttt{air} &\texttt{csulf} & \texttt{pres} & \texttt{slp}\\
        \hline
        Target Variables & \makecell {[1, 2, 6, 9, 13, 15,\\19, 20, 21, 23, 24, 27,\\ 31, 33, 37, 38, 40, \\47, 48, 49, 54, 55, 58]}&  \makecell{
[1, 2, 6, 7,\\ 10, 14, 16, 17,\\ 20, 22, 28, 30,\\33, 36, 42, 45,\\49, 54,55,58]}
& 
\makecell{[1, 2, 8, 17, \\25, 27, 45, 53]} &\makecell{[1, 3, 7, 8,\\ 13, 28, 32, 33,\\ 45, 46, 50, 53,\\ 54, 55, 56]
}
\\
        \hline
    \end{tabular}
    \caption{Selected target variables for the four tasks air temperature (\texttt{air}),  clear sky upward solar flux (\texttt{csulf}), surface pressure (\texttt{pres}) and sea level pressure (\texttt{slp}) over $100$ replications.}
    \label{table: climate-target}
\end{table}

\subsection{Comparison of Different $k$}\label{sec:different-k}
In this sectioon we show the performance of varying $k\in\{1,2,3\}$ for our method \emph{invariance-guided regularization (IGR)} in the same settings as \cref{sec:stock} and \cref{sec:climate}. The results presented in \cref{table:stock-different-k} and \cref{table: climate-different-k} indicate that the performance across different $k$ is similar.

\begin{table}[htb!]
    \footnotesize
    \centering
    \begin{tabular}{crr}
      \hline
      \multicolumn{1}{c}{Data} &
      \multicolumn{1}{c}{\texttt{AMT}} & 
      \multicolumn{1}{c}{\texttt{SPG}}\\
      \hline
      $k=1$ & {$0.135\pm0.068$}& {$0.036\pm 0.034$}\\
      $k=2$&{$0.131\pm0.074$}& {$0.048\pm 0.039$}\\
      $k=3$ &{$0.129\pm0.094$}& {$0.051\pm 0.040$}\\ 
      \hline
    \end{tabular}
	\caption{The average $\pm$ standard deviation of the worst-case out-of-sample $R^2$ \eqref{eq:worst-exp} for predicting the stocks $\mathtt{AMT}$ and $\mathtt{SPG}$ using IGR with different $k$.}
    \label{table:stock-different-k}
\end{table}
\begin{table}[htb!]
    \centering
    \footnotesize
    \begin{tabular}{cccccc}
    \hline
        Data & \texttt{air} &\texttt{csulf} & \texttt{pres} & \texttt{slp}\\
        \hline
        $k=1$&  $3.7882\pm0.3416$&$2.0431\pm0.0673$  &$ 1.5892\pm0.1952$&$3.0392\pm0.2569$\\
        $k=2$&$ 3.7838 \pm 0.3281$&$2.0523\pm 0.0883$ &$1.6077\pm 0.1122$&$3.0466\pm0.1955$\\
        $k=3$&  $3.7652\pm0.4016 $ &$2.0637\pm 0.0415$& $1.6004\pm 0.1042$&$3.0379\pm 0.2157$\\
        \hline
    \end{tabular}
    \caption{The average $\pm$ standard deviation of the mean squared error~\eqref{eq:MSE} of the four tasks air temperature (\texttt{air}),  clear sky upward solar flux (\texttt{csulf}), surface pressure (\texttt{pres}) and sea level pressure (\texttt{slp}) using IGR with different $k$.}
    \label{table: climate-different-k}
\end{table}

\subsection{Causal Relation Identified by Our Method in Climate Dynamic Data}
\label{appendix:climate-causal}

To qualitatively evaluate our method for causal discovery, we present the paths identified by our approach among six regions (No. 20, 23, 38, 40, 48, and 49) in the air temperature task (\textit{air}) in \cref{fig:map}. In particular, the causal path from the Arabian Sea (No. 38) to the eastern limb of ENSO (No. 40) via the Indian Ocean (No. 49) is verified by \citet{doi:10.1126/science.284.5423.2156} and~\citet{timmermann2018nino}. Additionally, the paths between East Asia (No. 48 and No. 23) and the high surface pressure sector of the Indian Monsoon region (No. 38) align with the known relationship between the sea surface temperatures of the Indian Ocean and the Asian Summer Monsoon~\citep{li2001relationship}. These results demonstrate that our method is capable of effectively identifying causal relationships.

\begin{figure}[htb!]
    \centering
    \includegraphics[width=0.4\linewidth]{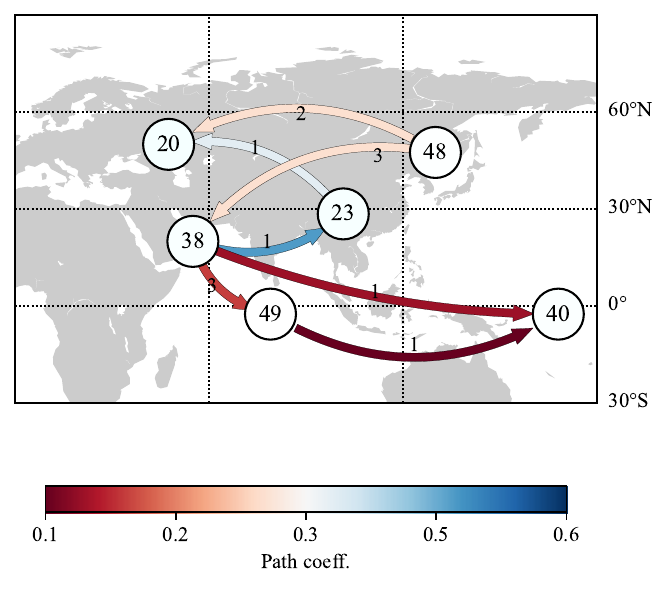}
    \caption{The paths identified by our approach among the six regions (No. 20, 23, 38, 40, 48, and 49) in the air temperature task (\textit{air}). The edge colors represent the path coefficients, while the labels indicate the time lags in days.}
    \label{fig:map}
\end{figure}


\end{document}